\documentclass[12pt, 3p
			  ]{elsarticle}
\usepackage{amsmath,amsbsy,amsfonts}
\allowdisplaybreaks[4]
\usepackage{bm}
\usepackage{amssymb}

\usepackage{booktabs}	
\usepackage[colorlinks]{hyperref}
\usepackage{enumerate}
\usepackage{lineno}
\usepackage{stmaryrd} 

\usepackage{caption,subcaption}
\usepackage{graphicx}
\usepackage{graphics}
\usepackage{epstopdf}
\usepackage{pstricks-add}
\usepackage{CJKutf8,CJKnumb}
\AtBeginDocument{\begin{CJK}{UTF8}{gbsn}}
\AtEndDocument{\clearpage \end{CJK}}
\usepackage{cases}
\usepackage{float}
\usepackage{cleveref}
\usepackage{diagbox}
\usepackage{geometry}
\usepackage{multirow}

\newfam\msbfam
\font\tenmsb=msbm10    \textfont\msbfam=\tenmsb
\font\sevenmsb=msbm7 \scriptfont\msbfam=\sevenmsb
\font\fivemsb=msbm5 \scriptscriptfont\msbfam=\fivemsb

\newfam\bigfam
\font\tenbig=msbm10 scaled \magstep2   \textfont\bigfam=\tenbig
\font\sevenbig=msbm7 scaled \magstep2 \scriptfont\bigfam=\sevenbig
\font\fivebig=msbm5 scaled \magstep2
\scriptscriptfont\bigfam=\fivebig

\newtheorem{theorem}{Theorem}[section]
\newtheorem{lemma}[theorem]{Lemma}

\newtheorem{algorithm}{Algorithm}

\newdefinition{remark}{Remark}[section]
\newdefinition{example}{Example}[section]
\newenvironment{proof}[1][Proof]{\noindent\textbf{#1. }}{\hfill $\Box$}

\definecolor{mygray}{gray}{.85}

\numberwithin{equation}{section}

\journal{}

\begin{document}
\begin{frontmatter}

\title{A Local Parallel Finite Element Method for Super-Hydrophobic Proppants
in a Hydraulic Fracturing System Based on a 2D/3D Transient Triple-Porosity Navier-Stokes Model
\tnoteref{mytitlenote}}

\author[add1]{Luling Cao}\ead{lulingcao@163.com}
\author[add1]{Jian Li\corref{correspondingauthor}}\ead{jiaaanli@gmail.com}
\cortext[correspondingauthor]{Corresponding author.}
\author[add2]{Zhangxin Chen}\ead{zhachen@ucalgary.ca}
\author[add3]{Guangzhi Du}\ead{gzdu@sdnu.edu.cn}

\address[add1]{School of Mathematics and Data Science, Shaanxi University of Science and Technology, Xi'an 710021, China}
\address[add2]{Department of Chemical \& Petroleum Engineering, Schulich School of Engineering, University of Calgary, Calgary T2N 1N4, Canada}
\address[add3]{School of Mathematics and Statistics, Shandong Normal University, Jinan 250014, China}

\begin{abstract}
A hydraulic fracturing system with super-hydrophobic proppants is characterized by a transient triple-porosity Navier-Stokes model. For this complex multiphysics system, particularly in the context of three-dimensional space, a local parallel and non-iterative finite element method based on two-grid discretizations is proposed.
The underlying idea behind utilizing the local parallel approach is to combine a decoupled method, a two-grid method and a domain decomposition method.
The strategy allows us to initially capture low-frequency data across the decoupled domain using a coarse grid. Then it tackles high-frequency components by solving residual equations within overlapping subdomains by employing finer grids and local parallel procedures at each time step.
By utilizing this approach, a significant improvement in computational efficiency can be achieved.
Furthermore, the convergence results of the approximate solutions from the algorithm are obtained.
Finally, we perform 2D/3D numerical experiments to demonstrate the effectiveness and efficiency of the algorithm as well as to illustrate its advantages in application.

\end{abstract}

\begin{keyword}
super-hydrophobic proppant; hydraulic fracturing system; transient triple-\\
porosity-Navier-Stokes model; finite element method; local and parallel algorithm;
overlapping domain decomposition method
\end{keyword}

\end{frontmatter}

\section{Introduction}
A coupled system of free flow and porous media flow has attracted significant attention in research due to its diverse range of applications.
For instance, this includes processes like underground hydrocarbon recovery \cite{sarma2006new}, geothermal energy production \cite{gawin1995coupled}, purely drinkable water recovery \cite{layton2013analysis}, and various other fields. Notably, it plays a crucial role in the field of petroleum extraction \cite{wei2010coupled, nie2012dual}.
The hydraulic fracturing technology is a technique for oil and gas well development.
Typically, water, sands, and chemicals are injected into a rock formation through wells under high pressure with the aim of creating new fractures (artificial fractures) in the rock. This process increases the size, extent, and connectivity of existing fractures (natural fractures) and serves as a reservoir modification method to enhance fluid flow capabilities within an oil/gas reservoir \cite{tiab2015petrophysics}.

In the hydraulic fracturing process, a fracturing fluid holds a significant responsibility in initiating, enlarging, and maintaining fractures to improve the permeability of a reservoir.
This facilitates the smoother flow of oil/gas into a wellbore, leading to improved production efficiency.
In particular, super-hydrophobic proppants with their capacity to permit the passage of oil/gas while obstructing water can significantly augment an oil/gas recovery rate.
Hence, conducting numerical simulations to assess the influence of different parameters of this material on recovery rates serves as a foundation for optimizing water control fracturing processes and material parameters \cite{liang2016comprehensive, lu2022shaly}.

The above coupled system is usually described by Stokes(Navier-Stokes)-Darcy equations or dual-porosity Stoke(Navier-Stokes) equations.
Up to now, these equations are well-studied resulting in a large number of numerical schemes that have been proposed and investigated.
However, they have limitations in fractured reservoirs for petroleum extraction due to their assumptions about a uniform matrix or fracture network within a Darcy (dual-porosity) system, which may not accurately reflect real reservoir conditions.
Therefore, a triple-porosity Stokes system was proposed \cite{nasu2022new} as a more efficient and practical alternative to the Darcy (dual-porosity) system.
A realistic reservoir possesses a more complex fracture network because the physical properties of different continua and geometrical structures are distinct.
As a result, a triple-porosity region consists of three interconnected and transmittable porous media, known as more permeable macrofractures, less permeable microfractures, and a stagnant-matrix region, respectively.
This triple-porosity region is governed by transient triple-porosity equations.
Accordingly, the conduit region is described by the nonstationary Navier-Stokes equations.
In addition, five physically valid coupling conditions are considered to connect the two distinct models at an interface between the free flow and the porous medium flow.
In this way, a hydraulic fracturing system with super-hydrophobic proppants can be described by a transient triple-porosity Navier-Stokes model.
To facilitate the widespread application of this complex model,
it is necessary to investigate an efficient algorithm for solving it.

To the best of our knowledge, there are few numerical results available for this model.
It has been demonstrated that for complex  multi-physics problems, local parallel finite element methods exhibit efficiency.
The two-grid method, first introduced by Xu for solving semi-linear elliptic equations \cite{xu1996two}, has gained popularity for improving computational efficiency.
Based on his idea, He, Xu, Layton and others developed this algorithm \cite{he2008local, he2006local, layton1993two, xu2000local, xu2001local, xu2002local} for the Stokes and Navier-Stokes equations.
Subsequently, some parallel algorithms for coupled problems have been developed
\cite{shang2011new, shang2015parallel, zheng2021parallel, du2016modified, zuo2018parallel, du2021local, ding2021local, shang2010local, li2021local, li2022local}.
However, there is no study on local parallel finite element discretization algorithms for a time-dependent triple-porosity Navier-Stokes model.
In addition, there are even fewer three-dimensional parallel numerical examples for the coupled problems.
Furthermore, this model has yet to find practical applications in the field of petroleum extraction, including specialized areas such as the use of super-hydrophobic proppants in hydraulic fracturing systems.

In this paper, we present and analyze local parallel finite element discretization algorithms for simulating the behavior of super-hydrophobic proppants in a hydraulic fracturing system. Our approach is based on solving the transient triple-porosity Navier-Stokes equations while considering the Beavers-Joseph interface condition.
The backward Euler scheme is considered for the temporal discretization.
By combining a decoupled method, two-grid method, and domain decomposition method, we achieve excellent parallel performance.
Following the partitioned time-stepping method proposed in \cite{shan2013partitioned, cao2021decoupled, cao2022parallel}, we decouple the entire domain into two subdomains and solve the four decoupled subproblems in parallel on a coarse grid to capture low-frequency data.
Then, we solve the residual equations locally and in parallel within overlapping finer grids to obtain high-frequency components.
During this step, the two-grid method is used to linearize the incompressible Navier-Stokes equations and interface coupling terms.
This approach allows us to enhance computational efficiency.

The rest of this paper is organized as follows: A hydraulic fracturing system with
super-hydrophobic proppants based on the transient triple-porosity Navier-Stokes model is introduced in Section 2.
In Section 3, some preliminaries which are needed in algorithm analysis are provided. A local parallel finite element algorithm is designed and
analyzed in Section 4. The fully discrete local parallel algorithm is proposed in Section 5. Section 6 shows some numerical examples to verify the theoretical results
and we conclude this work through a short conclusion in Section 7.

\section{Model Description}
Let $\Omega=\Omega_p \cup \Omega_c \subset \mathbb{R}^d(d=2,3)$ be a bounded convex domain separated by a common interface $\Gamma=\Omega_p \cap \Omega_c$, where $\Omega_p$ represents the triple-porosity subdomain and $\Omega_c$ is the conduit subdomain~(see Figure \ref{Pic1}).
\begin{figure}[htbp]
  \centering
  \includegraphics[width=12cm]{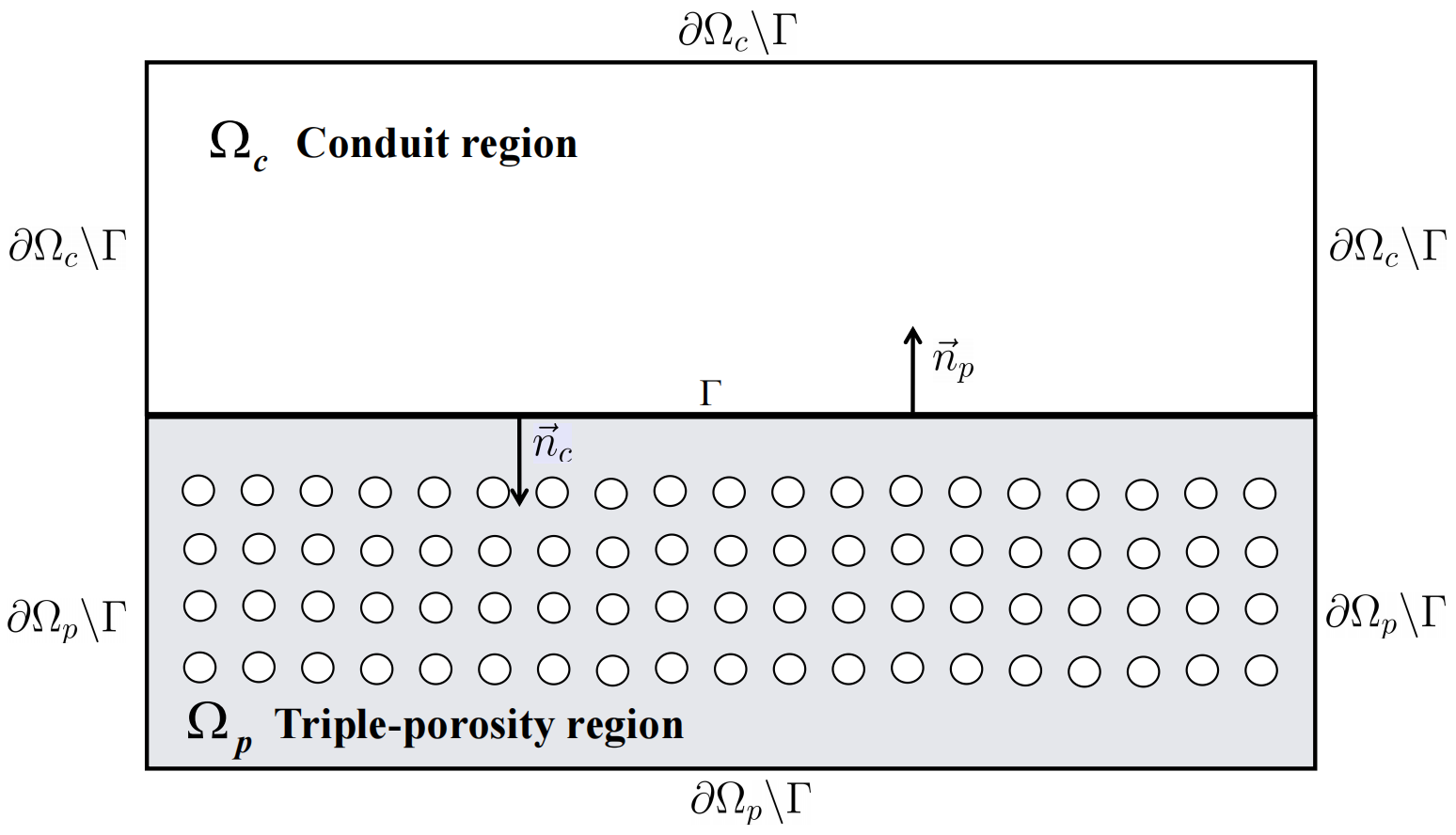}\\
  \caption{\label{Pic1}\small{A sketch of the triple-porosity region~$\Omega_p$, the conduit region~$\Omega_c$~and the interface~$\Gamma$.}}
\end{figure}

In the triple-porosity region $\Omega_p$, the more-permeable macrofracture pressure $p_F(x,t)$, the less-permeable microfracture pressure $p_f(x,t)$ and the stagnant-matrix pressure $p_m(x,t)$ are governed by
\begin{equation}
\begin{split}
\phi_F C_F \frac{\partial p_F}{\partial t} - \nabla \cdot (\frac{k_F}{\tilde{\mu}} \nabla p_F) + \frac{\sigma^{\ast} k_f}{\tilde{\mu}}(p_F - p_f) &= q_F~~~~~~~\text{in}~\Omega_p \times (0, T],\\
p_F(x,0)&=p_F^0(x)~~~\text{in}~\Omega_p,\\
p_F&=0~~~~~~~~~\text{on}~\partial \Omega_p \backslash \Gamma \times (0,T],
\end{split}
\end{equation}
\begin{align}
\phi_f C_f \frac{\partial p_f}{\partial t} - \nabla \cdot (\frac{k_f}{\tilde{\mu}} \nabla p_f) + \frac{\sigma^{\ast} k_f}{\tilde{\mu}}(p_f - p_F) + \frac{\sigma k_m}{\tilde{\mu}}(p_f - p_m)&=q_f~~~~~~\text{in}~\Omega_p \times (0,T],~~~~~~~~~~~~~~~\nonumber\\
p_f(x,0)&=p_f^0(x)~~\text{in}~\Omega_p,\\
p_f&=0~~\text{on}~\partial \Omega_p \backslash \Gamma \times (0,T], \nonumber
\end{align}
\begin{equation}
\begin{split}
\phi_m C_m \frac{\partial p_m}{\partial t} - \nabla \cdot (\frac{k_m}{\tilde{\mu}} \nabla p_m) + \frac{\sigma k_m}{\tilde{\mu}}(p_m - p_f) &= q_m~~~~~~\text{in}~\Omega_p \times (0,T],\\
p_m(x,0) &= p_m^0(x)~~\text{in}~\Omega_p,\\
p_m&=0~~~~~~~~\text{on}~\partial \Omega_p \backslash \Gamma \times (0,T],
\end{split}
\end{equation}
where the porosity, compressibility, intrinsic permeability and source/sink term are denoted by $\phi_i, C_i, k_i, q_i (i=F,f,m)$, respectively.
It is worth noting that in a hydraulic fracturing system, the material of the proppant used can directly impact the magnitude of the intrinsic permeability $k_F$.
In addition, $\tilde{\mu}$ is the dynamic viscosity and $\sigma^{\ast} (\sigma)$ represents the shape factor characterizing the morphology and dimension of the macrofractures(microfractures). The terms $\frac{\sigma^{\ast} k_f}{\tilde{\mu}}(p_F - p_f)$ and $\frac{\sigma k_m}{\tilde{\mu}}(p_m - p_f)$ describe the mass transfer between different fractures and matrix.

In the conduit region $\Omega_c$, the fluid flow velocity $\vec{u}_c(x,t)$ and the kinematic pressure $p(x,t)$ are governed by
\begin{equation}
\begin{split}
\frac{\partial \vec{u}_c}{\partial t} - \nabla \cdot \mathbb{T}(\vec{u}_c, p) + (\vec{u}_c \cdot \nabla) \vec{u}_c &= \vec{f}_c~~~~~~~~\text{in}~\Omega_c \times (0,T],\\
\nabla \cdot \vec{u}_c &=0~~~~~~~~~\text{in}~\Omega_c \times (0,T],\\
\vec{u}_c(x,0)&=\vec{u}_c^0(x)~~~~\text{on}~\Omega_c,\\
\vec{u}_c &=0~~~~~~~~~~\text{on}~\partial \Omega_c \backslash \Gamma \times (0,T].
\end{split}
\end{equation}
Here $\mathbb{T}(\vec{u}_c, p)=2\nu \mathbb{D}(\vec{u}_c) - p \mathbb{I}$ is the stress tensor, $\mathbb{D}(\vec{u}_c)=\frac{1}{2}(\nabla \vec{u}_c + \nabla^{T} \vec{u}_c)$ is the velocity deformation tensor, $\mathbb{I}$~is the identity tensor, $\nu$ is the kinematic viscosity of the fluid and $\vec{f}_c$ is a general body forcing term that includes gravitational acceleration.

Along the interface $\Gamma$, the no-direct fluid-interaction conditions between matrix or microfracture and the conduit region are imposed:
\begin{align}
-\frac{k_m}{\mu}\nabla p_m \cdot \vec{n}_p &=0, \label{km_interface}\\
-\frac{k_f}{\mu}\nabla p_f \cdot \vec{n}_p &=0.
\end{align}
Furthermore, the three well-accepted interface conditions between more-permeable macrofractures and conduit region are as follows:
\begin{align}
\vec{u}_c \cdot \vec{n}_c &= \frac{k_F}{\tilde{\mu}} \nabla p_F \cdot \vec{n}_p,\\
-\vec{n}_c^{T} \mathbb{T}(\vec{u}_c,p_c) \vec{n}_c &= \frac{p_F}{\rho},\\
- P_{\tau} (\mathbb{T}(\vec{u}_c,p_c) \vec{n}_c) &=
\frac{\alpha \nu \sqrt{d}}{\sqrt{\text{trace}(\Pi)}}P_{\tau} (\vec{u}_c + \frac{k_F}{\tilde{\mu}} \nabla p_F), \label{uc_interface}
\end{align}
where $\vec{n}_p$ and $\vec{n}_c$ satisfying $\vec{n}_p=-\vec{n}_c$ on $\Gamma$ are the unit outward normal vectors on $\partial \Omega_p$ and $\partial \Omega_c$.
The last one is the Beavers-Joseph ($\mathbf{BJ}$) interface condition, which describes the tangential components of the stress tensor are proportional to the jump of the tangential velocity across the interface. $P_{\tau}(\vec{v})=\sum_{i=1}^{d-1} (\vec{v} \cdot \tau_i)\tau_i$ denotes the projection onto the local tangent plane on $\Gamma$ with $\tau_i (i=1,2,...,d-1)$ which is the unit tangential vector. In addition, $\alpha$ is a positive constant parameter and $\Pi=k_F \mathbb{I}$ stands for the intrinsic permeability.

\section{Preliminaries}
\subsection{Weak formulation}
In order to introduce the weak formulation, we define the functional spaces
\begin{align*}
X_{pF} &:= \{v_F \in H^1(\Omega_p)| v_F=0~~\text{on}~\partial \Omega_p \backslash \Gamma \},\\
X_{pf} &:= \{v_f \in H^1(\Omega_p)| v_f=0~~\text{on}~\partial \Omega_p \backslash \Gamma \},\\
X_{pm} &:= \{v_m \in H^1(\Omega_p)| v_m=0~~\text{on}~\partial \Omega_p \backslash \Gamma \},\\
X_c &:= \{\vec{v}_c \in [H^1(\Omega_c)]^d| \vec{v}_c=0~~\text{on}~\partial \Omega_c \backslash \Gamma \},\\
Q &:= L^2(\Omega_c).
\end{align*}
For convenience, the norm of the Sobolev space $H^r=W^{r,2}$ is denoted by $\|\cdot\|_{r}$, the semi-norm indexed by $|\cdot|_{r}$ with $r>0$ and the
product spaces are defined by
$$X_p := X_{pF} \times X_{pf} \times X_{pm},~~~~W:=X_p \times X_c.$$
Furthermore, the spaces involving time are defined by $Q_T:=L^2(0,T;Q)$~and~$W_T:=H^1(0,T;$\\
$X_{pF},X_{pF}^{'}) \times H^1(0,T;X_{pf},X_{pf}^{'}) \times H^1(0,T;X_{pm},X_{pm}^{'}) \times H^1(0,T;X_{c},X_{c}^{'})$, where
\begin{align*}
H^1(0,T;X_{pi},X_{pi}^{'})&=\Big\{ v_i: v_i \in L^2(0,T;X_{pi})~\text{and}~\frac{\partial v_i}{\partial t} \in L^2(0,T;X_{pi}^{'}),
\Big\}~~~~i=F,f,m ,\\
H^1(0,T;X_{c},X_{c}^{'})&=\Big\{ \vec{v}_c: \vec{v}_c \in L^2(0,T;X_{c})~\text{and}~\frac{\partial \vec{v}_c}{\partial t} \in L^2(0,T;X_{c}^{'}) \Big\}.
\end{align*}

Following the literature \cite{Cao2010CoupledSM}, the rescaling factor $\eta$ is introduced to make the variational problem well-posed when $\eta$ is small enough.
The weak formulation of the transient coupled triple-porosity Navier-Stokes model reads as:
find $\boldsymbol{u}=[p_F, p_f, p_m, \vec{u}_c]^T \in W_T, p \in Q_T$, for all $\boldsymbol{v}=[v_F, v_f, v_m, \vec{v}_c]^{T} \in W$~and~$q \in Q$, such that
\begin{equation}\label{CoupledFormu}
\begin{split}
\langle \frac{\partial \boldsymbol{u}}{\partial t},\boldsymbol{v} \rangle_{\eta}
+a_{\eta}(\boldsymbol{u},\boldsymbol{v})
+b_{N\eta}(\vec{u}_c,\vec{u}_c, \vec{v}_c)
+b_{\eta}(\vec{v}_c,p)
&=(\boldsymbol{f}, \boldsymbol{v})_{\eta},\\
b_{\eta}(\vec{u}_c,q)&=0,
\end{split}
\end{equation}
where
\begin{align*}
&\mathbf{w}=[p_F, \vec{u}_c]^{T},~~~~
\mathbf{P}=[p_f, p_m]^{T},~~~~
\mathbf{R}=[p_F,p_f]^{T},\\
&\mathbf{\Phi}=[v_F, \vec{v}_c]^{T},~~~~
\mathbf{Q}=[v_f, v_m]^{T},~~~~
\mathbf{T}=[v_F, v_f]^{T},\\
&\langle \frac{\partial \boldsymbol{u}}{\partial t},\boldsymbol{v} \rangle_{\eta}
:=\phi_F C_F(\frac{\partial p_F}{\partial t},v_F)
+\phi_f C_f(\frac{\partial p_f}{\partial t},v_f)
+\phi_m C_m(\frac{\partial p_m}{\partial t},v_m)
+\eta(\frac{\partial \vec{u}_c}{\partial t},\vec{v}_c),\\
&a_{\eta}(\boldsymbol{u},\boldsymbol{v})
:=a_F(p_F,v_F)+a_f(p_f,v_f)+a_m(p_m,v_m)+a_{c\eta}(\vec{u}_c,\vec{v}_c)
+a_{\Gamma \eta}(\mathbf{w},\mathbf{\Phi})\\
&~~~~~~~~~~~~~~~+a_{mf}(\mathbf{P},\mathbf{Q})
+a_{fF}(\mathbf{R},\mathbf{T}),\\
&a_F(p_F,v_F):=\frac{k_F}{\tilde{\mu}}(\nabla p_F, \nabla v_{F}),
~~a_f(p_f,v_f):=\frac{k_f}{\tilde{\mu}}(\nabla p_f, \nabla v_{f}),
~~a_m(p_m,v_m):=\frac{k_m}{\tilde{\mu}}(\nabla p_m, \nabla v_{m}),\\
&a_{c\eta}(\vec{u}_c,\vec{v}_c):=2\nu \eta (\mathbb{D}(\vec{u}_c), \mathbb{D}(\vec{v}_c))+\frac{\eta \nu \alpha \sqrt{d}}{\sqrt{\text{trace}(\Pi)}}\langle P_{\tau} \vec{u}_c, P_{\tau} \vec{v}_c \rangle,\\
&a_{\Gamma \eta}(\mathbf{w},\mathbf{\Phi}):= \frac{\eta}{\rho} \int_{\Gamma}p_F \vec{v}_c \cdot \vec{n}_c \mathrm{d}\Gamma - \int_{\Gamma} v_F \vec{u}_c \cdot \vec{n}_c \mathrm{d}\Gamma
+\frac{\eta \nu \alpha \sqrt{k_F}}{\tilde{\mu}}\int_{\Gamma}  \nabla_{\tau}p_F \cdot P_{\tau} \vec{v}_c \mathrm{d}\Gamma,\\
&a_{mf}(\mathbf{P},\mathbf{Q}):=\frac{\sigma k_m}{\tilde{\mu}}\int_{\Omega_p}(p_m - p_f) v_m \mathrm{d} \Omega
+ \frac{\sigma k_m}{\tilde{\mu}}\int_{\Omega_p} (p_f - p_m) v_f \mathrm{d} \Omega,\\
&a_{fF}(\mathbf{R},\mathbf{T}):=\frac{\sigma^{\ast} k_f}{\tilde{\mu}}\int_{\Omega_p} (p_f - p_F) v_f \mathrm{d}\Omega
+ \frac{\sigma^{\ast} k_f}{\tilde{\mu}}\int_{\Omega_p} (p_F - p_f) v_F \mathrm{d}\Omega,\\
&b_{N\eta}(\vec{u}_c,\vec{u}_c, \vec{v}_c):=\eta ((\vec{u}_c \cdot \nabla)\vec{u}_c,\vec{v}_c),
~~~~b_{\eta}(\vec{v}_c,p):= -\eta(p,\nabla \cdot \vec{v}_c),\\
&(\boldsymbol{f}, \boldsymbol{v})_{\eta}:=
(q_F, v_F) + (q_f, v_f) + (q_m, v_m) + \eta(\vec{f}_c, \vec{v}_c).
\end{align*}

\subsection{Mixed finite element spaces}
Let $T_h=\{K\}$ be a regular triangulation of $\overline{\Omega}=\overline{\Omega}_p \cup \overline{\Omega}_c$ with mesh size $0 <h <1$. If $d=2$, the element $K \in T_h$ will be a triangle; if $d=3$, it will be a tetrahedra. The regular partitions $T_h^{p}$~and~$T_h^c$ induced on the regions $\Omega_p$ and $\Omega_c$ are assumed to be compatible on the interface $\Gamma$. Let $W^h:=X_{pF}^h \times X_{pf}^h \times X_{pm}^h \times X_{c}^h \subset W$ and $Q^h \subset Q$ denote the finite element subspaces.
Here, $W^h$ is equipped with continuous piecewise polynomials of degree $r+1$, and $Q^h$ is equipped with continuous piecewise polynomials of degree $r(r \geq 1)$.
The pair $X_c^h$ and $Q^h$ is assumed to satisfy the discrete inf-sup condition, which is there exists a constant $\beta >0$ such that
\begin{align*}
\inf_{0 \neq q_h \in Q^h} \sup_{0 \neq \vec{v}_c^h \in X_{c}^h} \frac{b_{\eta}(\vec{v}_c^h, q_h)}{\eta \|q_h\|_0 \|\nabla \vec{v}_{c}^h\|_0} > \beta.
\end{align*}

Given $G \subset \subset \Omega_{\ast 0} \subset \subset \Omega_{\ast}$, where $G \subset \subset \Omega_{\ast}$ means that $\text{dist} (\partial G \backslash \partial \Omega_{\ast}, \partial \Omega_{\ast 0} \backslash \partial \Omega_{\ast}) >0 (\ast=c, p)$. Define $T_h^c(G), X_c(G), X_c^h(G), Q^h(G)$ to be the restriction of $T^c_h, X_c, X_c^h, Q^h$ to $G$, respectively. In addition,
\begin{align*}
X_{c0}^h &:=\{ \vec{v} \in X_c^h: \text{supp}~\vec{v} \subset \subset G \},
~~Q_0^h :=\{ q \in Q^h: \text{supp}~q \subset \subset G \}.
\end{align*}

Based on above mixed finite element spaces, some basic assumptions, inequalities and lemmas are introduced.\\
\noindent(A1. Poincar$\mathrm{\acute{e}}$-Friedriches~inequality)
For all $\vec{v} \in X_c(\Omega_{c0})$, there exists a positive constant~$C_P$~which only depends on the area of~$\Omega_{c0}$~such that
\begin{equation}\label{Poin}
\|\vec{v}\|_1 \leq C_P |\vec{v}|_1.
\end{equation}
(A2. Korn inequality) For all $\vec{v} \in X_c(\Omega_{c0})$, there exists a positive constant~$C_K$ such that
\begin{equation}\label{Korn}
(\mathbb{D}(\vec{v}),\mathbb{D}(\vec{v})) \geq C_K |\vec{v}|_1^2.
\end{equation}
(A3. Trace inequality) There exists a positive constant~$C_t$ to satisfy
\begin{equation}
\|\vec{v}\|_{L^2(\Gamma \Omega_{c0})} \leq C_t  \|\vec{v}\|_0^{1/2} \|\vec{v}\|_1^{1/2},~~\forall \vec{v} \in X_c(\Omega_{c0}).
\end{equation}
(A4. Inverse inequality) When~$1 \leq p,q \leq \infty, 0 \leq l \leq k$, it holds
\begin{equation}\label{inverse2}
\|\vec{v}_h\|_{W^{k,p}} \leq C_{I} h^{-\max\{0,\frac{d}{q}-\frac{d}{p}\}}h^{l-k} \|\vec{v}_h\|_{W^{l,q}},~~~~\forall \vec{v}_h \in X_c^h(\Omega_{c0}),
\end{equation}
with a positive constant~$C_{I}$~independent of~$h$.\\
(A6. Superapproximation)
For $G \subset \Omega_{c0}$, let $\omega \in C_0^{\infty}(\Omega)$ with $\mathrm{\text{supp}}~\omega \subset \subset G$. Then for any $(\vec{u}_c, p) \in X_c^h(G) \times Q^h(G)$, there is $(\vec{v},q) \in X_{c0}^h(G) \times Q_0^h(G)$ such that
\begin{align}\label{Superapproximation}
\| \omega \vec{u} -\vec{v} \|_{1,G} \leq c h \|\vec{u}\|_{1,G},~~~~
\| \omega p -q \|_{0,G} \leq c h \|p\|_{0,G}.
\end{align}
Here and after, $c$ is a generic positive constant which may represent different values at its different occurrences.\\
(A7. Stability)
There exists a constant $\beta >0$ such that
\begin{align*}
\inf_{0 \neq q_h \in Q^h(G)} \sup_{0 \neq \vec{v}_c^h \in X_{c}^h(G)} \frac{b_{\eta}(\vec{v}_c^h, q_h)}{\eta \|q_h\|_0 \|\nabla \vec{v}_{c}^h\|_0} > \beta.
\end{align*}
When $G=\Omega_c$, the pair $X_c^h$ and $Q^h$ is assumed to satisfy this condition.\\
(A8. Nonlinear property)
For $\vec{u}_c, \vec{v}_c, \vec{w}_c \in X_c$, we have the following nonlinear properties
\begin{align*}
b_{N\eta}(\vec{u}_c, \vec{v}_c, \vec{w}_c)=-b_{N\eta}(\vec{u}_c, \vec{w}_c, \vec{v}_c),
\end{align*}
and
\begin{align*}
b_{N\eta}(\vec{u}_c, \vec{v}_c, \vec{w}_c)
&\leq C_N \|\vec{u}_c\|_{0,p} \|\nabla \vec{v}_c\|_{0,q} \|\vec{w}_c\|_{0,r},
~~~~\frac{1}{p}+\frac{1}{q}+\frac{1}{r}=1,\\
b_{N\eta}(\vec{u}_c, \vec{v}_c, \vec{w}_c)
&\leq C_N \|\vec{u}_c\|_0^{1/2} \|\nabla \vec{u}_c\|_0^{1/2} \|\nabla \vec{v}_c\|_0 \|\nabla \vec{w}_c\|_0.
\end{align*}

\begin{lemma}[\cite{heywood1990finite}, Discrete Gronwall Inequality]\label{LemGronwall}
Assume that~$E \geq 0$, for any integer~$M \geq 0$, $\kappa_m, A_m, B_m, C_m \geq 0$~satisfying
\begin{equation*}
A_M + \Delta t \sum_{m=0}^M B_m \leq \Delta t \sum_{m=0}^M \kappa_m A_m + \Delta t \sum_{m=0}^M C_m +E.
\end{equation*}
For all~$m$, assume that
\begin{equation*}
\kappa_m \Delta t < 1,
\end{equation*}
and set~$g_m=(1-\kappa_m\Delta t)^{-1}$, then
\begin{equation*}
A_M + \Delta t\sum_{m=0}^M B_m \leq \exp(\Delta t \sum_{m=0}^M g_m \kappa_m)(\Delta t \sum_{m=0}^M C_m + E).
\end{equation*}
\end{lemma}

In the context of the notations and assumptions provided, the mixed finite element approximation for the problem described in equation \eqref{CoupledFormu} can be expressed as follows: Given $\boldsymbol{u}_h(0)=\mathbb{P}_h \boldsymbol{u}_0=[p_{F0},p_{f0},p_{m0},\vec{u}_{c0}]^{T}$, find $\boldsymbol{u}_h=[p_{Fh},p_{fh},p_{mh},\vec{u}_{ch}]^{T} \in W^h$ and $p_h \in Q^h$ for $t \in (0,T]$ such that for all $\boldsymbol{v}=[v_F,v_f,v_m,\vec{v}_c]^{T} \in X^h, q \in Q^h$,
\begin{equation}\label{CoupledSemiFormu}
\langle \frac{\partial \boldsymbol{u}_h}{\partial t},\boldsymbol{v} \rangle_{\eta}
+a_{\eta}(\boldsymbol{u}_h,\boldsymbol{v})
+b_{N\eta}(\vec{u}_{ch},\vec{u}_{ch}, \vec{v}_c)
+b_{\eta}(\vec{v}_c,p_h)
-b_{\eta}(\vec{u}_{ch},q)
=(\boldsymbol{f}, \boldsymbol{v})_{\eta},
\end{equation}
where projection operator $\mathbb{P}_h: (p_F(t),p_f(t),p_m(t),\vec{u}_c(t),p(t)) \in (W,Q) \rightarrow (\mathrm{P}_h^F p_F(t),\mathrm{P}_h^f p_f(t),$\\
$\mathrm{P}_h^m p_m(t), \mathrm{P}_h^c \vec{u}_c(t), \mathrm{P}_h^p p(t)) \in (W^h,Q^h)$ satisfying
\begin{align*}
a_{\eta}(\mathbb{P}_h \boldsymbol{u}-\boldsymbol{u},\boldsymbol{v}_h)
+b_{\eta}(\vec{v}_{ch},\mathrm{P}_h^p p- p)
&=0,~~~~\forall \boldsymbol{v}_h=[v_{Fh},v_{fh},v_{mh},\vec{v}_{ch}]^{T} \in W^h,\\
b_{\eta}(\mathrm{P}_h^c \vec{u}_c - \vec{u}_c,q_h)&=0,~~~~\forall q_h \in Q^h.
\end{align*}
Similar with the properties in literatures \cite{nasu2022new, hou2022modeling, cao2021decoupled},
the solution~$(\boldsymbol{u},p)=(p_F,p_f,p_m,\vec{u}_c,p)$~to the problem~\eqref{CoupledFormu}~was supposed to satisfy
\begin{align}
&\|\vec{u}_c\|_{L^{\infty}(0,T;H^{r+1})}
+\|\vec{u}_c\|_{L^{\infty}(0,T;W^{2,d^{\ast}})}
+\|\vec{u}_{c,t}\|_{L^2(0,T;H^{r+1})}
+\|p\|_{H^1(0,T;H^r)}
+\|p_{m,t}\|_{L^2(0,T;H^{r+1})}\nonumber\\
&+\|p_{m,tt}\|_{L^2(0,T;L^2)}
+\|p_{f,t}\|_{L^2(0,T;H^{r+1})}
+\|p_{f,tt}\|_{L^2(0,T;L^2)}
+\|p_{F,t}\|_{L^2(0,T;H^{r+1})}
+\|p_{F,tt}\|_{L^2(0,T;L^2)} \nonumber\\
&\leq C_B, \label{BDness}
\end{align}
where $0 < r \leq k$, the positive constant~$C_B$~is independent of~$h$~and~$\Delta t$,
$\vec{u}_{c,t}=\frac{\partial \vec{u}_c}{\partial t}$,
$p_{\ast,t}=\frac{\partial p_{\ast}}{\partial t}$, $p_{\ast,tt}=\frac{\partial^2 p_{\ast}}{\partial t^2}(\ast=F,f,m)$~and~$d^{\ast} > d$.
Therefore, the following error estimates are obtained.

\begin{lemma}\label{LemsemiConvergence}
Under the conditions of \eqref{BDness},
the problem \eqref{CoupledSemiFormu} has a unique solution $(\boldsymbol{u}_h, p_h)$ and the following properties hold:
\begin{align*}
\|\boldsymbol{u}(t)-\boldsymbol{u}_h(t)\|_0
+h(|\boldsymbol{u}(t)-\boldsymbol{u}_h(t)|_1
+\|p(t)-p_h(t)\|_{0})
&\leq c h^{r+1}(\|\boldsymbol{u}\|_{H^{r+1}} + \|p\|_{H^{r}}),\\
\Big\| \frac{\partial \boldsymbol{u}}{\partial t}-\frac{\partial \boldsymbol{u}_h}{\partial t} \Big\|_0
&\leq c h^{r+1} \|\boldsymbol{u}\|_{H^{r+1}}.
\end{align*}
Furthermore, we have
\begin{align*}
\int_{0}^T \Big\| \frac{\partial^2 \boldsymbol{u}}{\partial t^2} - \frac{\partial^2 \boldsymbol{u}_h}{\partial t^2} \Big\|_0^2 \mathrm{d}t \leq ch^{2(r+1)},
~~\int_{0}^T \Big | \frac{\partial \boldsymbol{u}}{\partial t } - \frac{\partial  \boldsymbol{u}_h}{\partial t } \Big |_1^2 \mathrm{d}t \leq ch^{2r},
~~\int_{0}^T \Big | \frac{\partial p}{\partial t} - \frac{\partial p_h}{\partial t} \Big |_0^2 \mathrm{d}t \leq ch^{2r}.
\end{align*}

\end{lemma}

\subsection{Fully discrete finite element scheme}
For the temporal discretization, the time interval $[0,T]$ is averagely divided into N segments
$[t_n,t_{n+1}] (n=0,1,...,N-1)$ satisfying
\begin{equation*}
0=t_0 < t_1 < \dots < t_{N-1} < t_N=T,
\end{equation*}
and the time step is~$\Delta t=\frac{T}{N}$. Using backward Euler scheme, the fully discretization of the problem \eqref{CoupledFormu} based on the partitioned time-stepping method reads as follows.
\begin{algorithm}{(Partitioned Time-Stepping Algorithm)}\label{Algorithm-1}

\textbf{Step 1}. Given $(p_{Fh}^n, p_{fh}^n, p_{mh}^n) \in X_p^h$, for all $(v_{F}^h, v_{f}^h, v_{m}^h) \in X_p^h$, find $(p_{Fh}^{n+1}, p_{fh}^{n+1}, p_{mh}^{n+1}) \in X_p^h$ such that
\begin{equation}\label{fullypF}
\begin{split}
&\phi_F C_F \Big( \frac{p_{Fh}^{n+1} - p_{Fh}^n}{\Delta t},v_{F}^h\Big)
+\frac{k_F}{\tilde{\mu}}(\nabla p_{Fh}^{n+1}, \nabla v_{F}^h)
+\frac{\sigma^{\ast} k_f}{\tilde{\mu}} \int_{\Omega_p} (p_{Fh}^{n+1} - p_{fh}^n) v_{F}^h \mathrm{d}\Omega ~~~~~~\\
&- \int_{\Gamma} v_{F}^h \vec{u}_{ch}^n \cdot \vec{n}_c \mathrm{d}\Gamma
=(q_F(t_{n+1}),v_{F}^h),
\end{split}
\end{equation}
\begin{equation}\label{fullypf}
\begin{split}
&\phi_f C_f \Big( \frac{p_{fh}^{n+1} - p_{fh}^n}{\Delta t}, v_{f}^h\Big)
+ \frac{k_f}{\tilde{\mu}}(\nabla p_{fh}^{n+1}, \nabla v_{f}^h)
+ \frac{\sigma k_m}{\tilde{\mu}} \int_{\Omega_p}  (p_{fh}^{n+1} - p_{mh}^n)v_{f}^h \mathrm{d}\Omega ~~~~~~~\\
&+ \frac{\sigma^{\ast} k_f}{\tilde{\mu}} \int_{\Omega_p} (p_{fh}^{n+1} - p_{Fh}^n) v_{f}^h \mathrm{d}\Omega
=(q_f(t_{n+1}), v_{f}^h),
\end{split}
\end{equation}
\begin{equation}\label{fullypm}
\begin{split}
&\phi_m C_m \Big( \frac{p_{mh}^{n+1} - p_{mh}^n}{\Delta t}, v_{m}^h\Big)
+\frac{k_m}{\tilde{\mu}}(\nabla p_{mh}^{n+1},\nabla v_{m}^h)
+\frac{\sigma k_m}{\tilde{\mu}} \int_{\Omega_p} (p_{mh}^{n+1} - p_{fh}^n) v_{m}^h \mathrm{d}\Omega ~~~\\
&=(q_m(t_{n+1}),v_{m}^h).
\end{split}
\end{equation}

\textbf{Step 2}. Given $(\vec{u}_{ch}^n, p_h^n) \in X_c^h \times Q^h$, for all $(\vec{v}_{c}^h, q_h) \in X_c^h \times Q^h$, seek $(\vec{u}_{ch}^{n+1}, p_h^{n+1}) \in X_c^h \times Q^h$ such that
\begin{align}
&\eta \Big( \frac{\vec{u}_{ch}^{n+1} - \vec{u}_{ch}^n}{\Delta t}, \vec{v}_{ch} \Big)
+a_{c\eta}(\vec{u}_{ch}^{n+1},\vec{v}_{ch})
+b_{\eta}(\vec{v}_{ch},p_h^{n+1})-b_{\eta}(\vec{u}_{ch}^{n+1},q_h)
+b_{N\eta}(\vec{u}_{ch}^{n+1},\vec{u}_{ch}^{n+1},\vec{v}_{ch}) \nonumber\\
&+\frac{\eta}{\rho}\int_{\Gamma} p_{Fh}^{n} \vec{v}_{ch} \cdot \vec{n}_c \mathrm{d}\Gamma
+\frac{\eta \nu \alpha \sqrt{k_F}}{\tilde{\mu}}\int_{\Gamma}  \nabla_{\tau}p_{Fh}^{n} \cdot P_{\tau} \vec{v}_{ch} \mathrm{d}\Gamma
=\eta (\vec{f}_c(t_{n+1}) ,\vec{v}_{ch}). \label{fullyuc}
\end{align}
\end{algorithm}

For the scheme \eqref{fullypF}-\eqref{fullyuc}, we have the following results.
\begin{theorem}\label{PTS-Results}
Under the conditions of \eqref{BDness}, we have
\begin{align*}
\|\boldsymbol{u}(t_{n+1}) - \boldsymbol{u}_h^{n+1}\|_{0}^2
+\Delta t |\boldsymbol{u}(t_{n+1}) - \boldsymbol{u}_h^{n+1}|_1^2
\leq c(\Delta t^2 + h^{2(r+1)}).
\end{align*}

\end{theorem}

\section{Local finite element algorithms}
In this section, a local and parallel finite element algorithm based on two-grid discretizations is proposed. The underlying idea behind employing local and parallel method is to combine the decoupled method and two-grid method.
This strategy allows us to initially capture low-frequency data across the decoupled entire domain using a coarse grid.
Subsequently, we tackle high-frequency components by solving residual equations within overlapping subdomains, employing finer grids and local parallel procedures at each time step.
To provide an error analysis between the true solution $\boldsymbol{u}(t_{m+1})$ and the local numerical solution $\boldsymbol{u}_{m+1}^h$ in this method, the estimation approach has been designed as follows:
\begin{align*}
\|\boldsymbol{u}(t_{m+1})-\boldsymbol{u}_{m+1}^h\|_{\Omega_{\ast 0}}
&=\|\boldsymbol{u}(t_{m+1})-\boldsymbol{u}_h^{m+1}+\boldsymbol{u}_h^{m+1}-\boldsymbol{u}_{m+1}^h\|_{\Omega_{\ast 0}}\\
&\leq \|\boldsymbol{u}(t_{m+1})-\boldsymbol{u}_h^{m+1}\|_{\Omega_{\ast 0}}
+\|\boldsymbol{u}_h^{m+1}-\boldsymbol{u}_{H}^{m+1}\|_{\Omega_{\ast 0}}
+\|(\boldsymbol{u}_{H}-\boldsymbol{u}^h)(t_{m+1})\|_{\Omega_{\ast 0}}\\
&~~~~+\|\boldsymbol{u}_{H}^{m+1}-\boldsymbol{u}_{m+1}^h-(\boldsymbol{u}_H-\boldsymbol{u}^h)(t_{m+1})\|_{\Omega_{\ast 0}}.
\end{align*}
Therefore, two local algorithms are proposed based on this method. The first algorithm involves a semi-discrete approach utilizing local finite elements for spatial
discretization. The second algorithm adopts a fully discrete local finite element algorithm.
Furthermore, error estimates for these algorithms will be derived.

\begin{figure}[H]
\begin{centering}
\begin{subfigure}[t]{0.31\textwidth}
\centering
\includegraphics[width=1.32\textwidth]{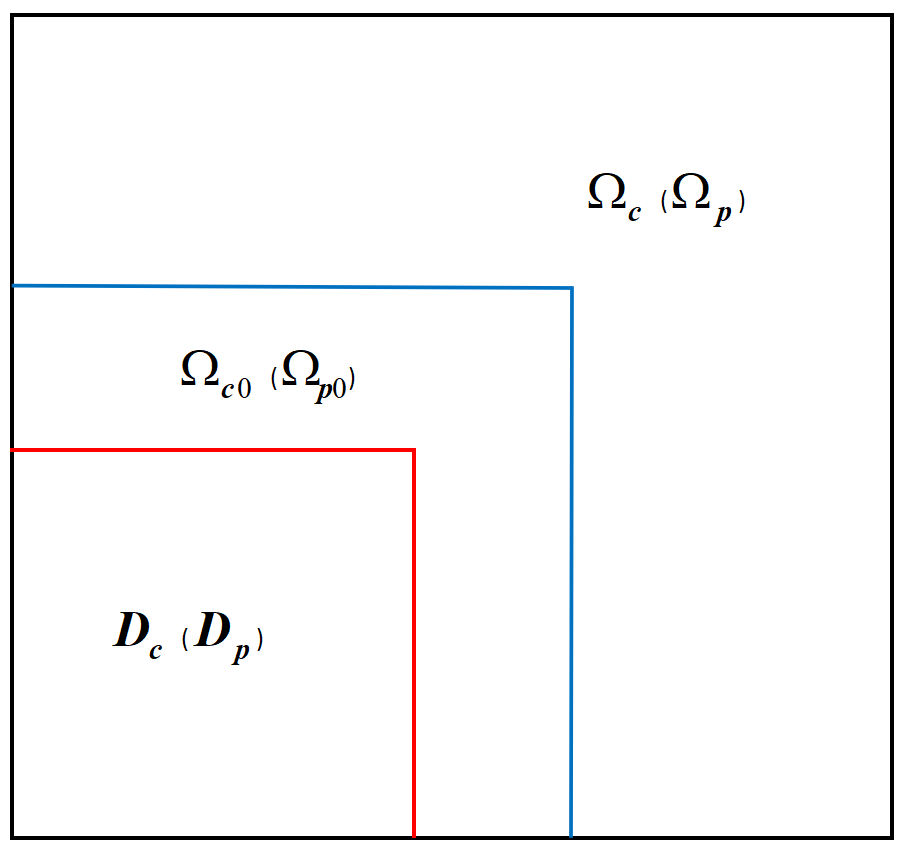}
\end{subfigure}
\hspace{32mm}
\begin{subfigure}[t]{0.31\textwidth}
\centering
\includegraphics[width=1.32\textwidth]{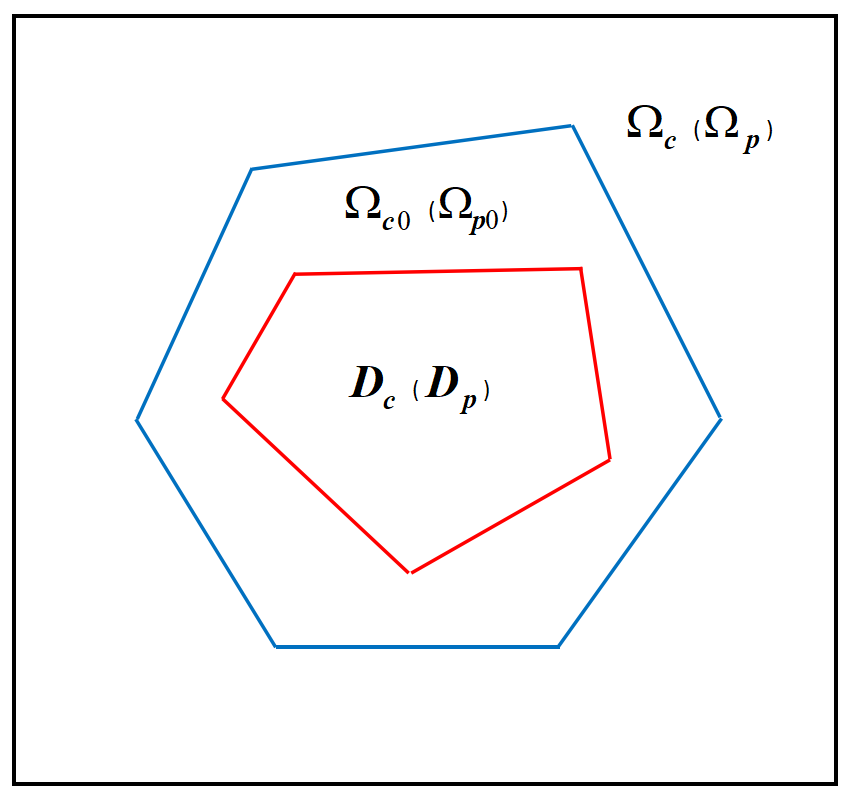}
\end{subfigure}
\end{centering}
\caption{\small{\label{Figdomain12} Subdomains $D_c \subset \subset \Omega_{c0} \subset \Omega_c$ and $D_p \subset \subset \Omega_{p0} \subset \Omega_p$.}}
\end{figure}

Let $D_c \subset \subset \Omega_{c0} \subset \Omega_c, D_p \subset \subset \Omega_{p0} \subset \Omega_p, \Gamma \Omega_{c0}=\Gamma \cap \partial \Omega_{c0}, \Gamma \Omega_{p0}=\Gamma \cap \partial \Omega_{p0}$ and $\Gamma \Omega_{0}=\Gamma \cap \partial \Omega_{c0} \cap \partial \Omega_{p0}$(see Figure \ref{Figdomain12}).
Consider $T_H(\Omega_{\ast})$ and $T_h(\Omega_{\ast})$ as the uniform discretizations with space sizes of $H$ and $h$, respectively, where $h << H <1$. Here, $\Omega_{\ast}$ can be $\Omega_c, \Omega_p, \Omega_{c0}$ and $\Omega_{p0}$.
The local and uniform refined grids $T_h(\Omega_{c0})$ and $T_h(\Omega_{p0})$ are obtained from $T_H(\Omega_c)$ and $T_H(\Omega_p)$, respectively. Moreover, $T_h(\Omega_{c0})$ and $T_h(\Omega_{p0})$ coincide with $T_h(\Omega_c)$ and $T_h(\Omega_p)$.

\subsection{Semi-discrete local finite element algorithm}\label{Sec-semi-discrete}
Set $\boldsymbol{u}_H(0)=\mathbb{P}_H \boldsymbol{u}_0=[\mathrm{P}^F_H p_{F}^0, \mathrm{P}^f_H p_{f}^0, \mathrm{P}^m_H p_{m}^0, \mathrm{P}^c_H \vec{u}_{c}^0]^T$
and $\boldsymbol{e}^h(0)=[(\mathbb{P}_h - \mathbb{P}_H) \boldsymbol{u}_0 ] |_{\Omega_0}=[(\mathrm{P}^F_h-\mathrm{P}^F_H)p_{F}^0|_{\Omega_{p}^0}, (\mathrm{P}^f_h-\mathrm{P}^f_H)p_{f}^0|_{\Omega_{p0}}, (\mathrm{P}^m_h-\mathrm{P}^m_H)p_{m}^0|_{\Omega_{p0}},
(\mathrm{P}^c_h-\mathrm{P}^c_H)\vec{u}_{c}^0|_{\Omega_{c0}}]^T$.
For $0 < t \leq T$, the local semi-discrete solutions $\boldsymbol{u}^h=[p^{Fh}, p^{fh}, p^{mh}, \vec{u}^{ch}]$ and $p^h$ are obtained by the following steps.\\
\noindent \textbf{Step 1.} Find global coarse grid solutions $\boldsymbol{u}_H=[p_{FH}, p_{fH}, p_{mH}, \vec{u}_{cH}]^T \in W^H$ and $p_H \in Q^H$, such that for all $\boldsymbol{v}=[v_F, v_f, v_m, \vec{v}_c]^T \in W^H$ and $q \in Q^H$ to satisfy
\begin{equation}\label{Semi-LFEM}
\langle \frac{\partial \boldsymbol{u}_H}{\partial t}, \boldsymbol{v} \rangle_{\eta}
+ a_{\eta}(\boldsymbol{u}_H, \boldsymbol{v})
+ b_{N\eta}(\vec{u}_{cH}, \vec{u}_{cH}, \vec{v}_c)
+ b_{\eta}(\vec{v}_c, p_H)
- b_{\eta}(\vec{u}_{cH}, q)
=(\boldsymbol{f},\boldsymbol{v})_{\eta}.
\end{equation}
\textbf{Step 2.} Find local fine grid corrections $\boldsymbol{e}^h=[e^{Fh}, e^{fh}, e^{mh}, e^{ch}]^T \in W^h(\Omega_0)$ and $\xi^h=Q^h(\Omega_{c0})$, such that the following equations hold for all $\boldsymbol{v}=[v_F, v_f, v_m, \vec{v}_c]^T \in W^h(\Omega_0)$ and $q \in Q^h(\Omega_{c0})$.\\
(2.1) In the local triple-porosity media subdomain $\Omega_{p0}$, the solutions $[e^{Fh}, e^{fh}, e^{mh}]$ are satisfied
\begin{align}
&\phi_F C_F \Big( \frac{\partial e^{Fh}}{\partial t}, v_F \Big)
+ a_F(e^{Fh}, v_F)
+ \frac{\sigma^{\ast} k_f}{\tilde{\mu}}(e^{Fh}-e^{fh},v_F)
= (q_F, v_F)\nonumber \\
&- \Big[ \phi_F C_F \Big( \frac{\partial p_{FH}}{\partial t}, v_F \Big)
+ a_F(p_{FH}, v_F)
+ \frac{\sigma^{\ast} k_f}{\tilde{\mu}}(p_{FH}-p_{fH}, v_F)
\Big]
+\langle \vec{u}_{cH} \cdot \vec{n}_c,v_F \rangle_{\Gamma \Omega_{p0}},~~ \label{Semi-pF}
\end{align}
\begin{align}
&\phi_f C_f \Big( \frac{\partial e^{fh}}{\partial t}, v_f \Big)
+ a_f(e^{fh}, v_f)
+ \frac{\sigma k_m}{\tilde{\mu}}(e^{fh}-e^{mh},v_f)
+ \frac{\sigma^{\ast} k_f}{\tilde{\mu}}(e^{fh}-e^{Fh},v_F)
= (q_F, v_F) \nonumber \\
&- \Big[ \phi_f C_f \Big( \frac{\partial p_{fH}}{\partial t}, v_f \Big)
+ a_f(p_{fH}, v_f)
+ \frac{\sigma k_m}{\tilde{\mu}}(p_{fH}-p_{mH},v_f)
+ \frac{\sigma^{\ast} k_f}{\tilde{\mu}}(p_{fH}-p_{FH}, v_F)
\Big],\label{Semi-pf}
\end{align}
\begin{align}
&\phi_m C_m \Big( \frac{\partial e^{mh}}{\partial t}, v_m \Big)
+ a_m(e^{mh}, v_m)
+ \frac{\sigma k_m}{\tilde{\mu}}(e^{mh}-e^{fh},v_m)
= (q_m, v_m) \nonumber\\
&~~~~~~~~~~~~~~~~~~~~~~~~~- \Big[ \phi_m C_m \Big( \frac{\partial p_{mH}}{\partial t}, v_m \Big)
+ a_m(p_{mH}, v_m)
+ \frac{\sigma k_m}{\tilde{\mu}}(p_{mH}-p_{fH},v_m)
\Big].\label{Semi-pF}
\end{align}
(2.2) In the local conduit subdomain $\Omega_{c0}$, the solutions $[e^{ch}, \xi^h]$ are satisfied
\begin{align}
&\eta \Big( \frac{\partial e^{ch}}{\partial t}, \vec{v}_c \Big)
+ a_{c\eta}(e^{ch}, \vec{v}_c )
+ b_{\eta}(\vec{v}_c, \xi^h)
- b_{\eta}(e^{ch}, q)
+ b_{N\eta}(e^{ch}, \vec{u}_{cH}, \vec{v}_c)
+ b_{N\eta}(\vec{u}_{cH}, e^{ch}, \vec{v}_c) \nonumber\\
&=\eta (\vec{f}_c,\vec{v}_c)
- \Big[
\eta \Big( \frac{\partial \vec{u}_{cH}}{\partial t}, \vec{v}_c \Big)
+ a_{c\eta}(\vec{u}_{cH}, \vec{v}_c)
+ b_{\eta}(\vec{v}_c, p_{H})
- b_{\eta}(\vec{u}_{cH}, q)
+ b_{N\eta}(\vec{u}_{cH},\vec{u}_{cH},\vec{v}_c)
\Big] \nonumber\\
&~~~~-\frac{\eta}{\rho}\langle p_{FH}, \vec{v}_c \cdot \vec{n}_c \rangle_{\Gamma \Omega_{c0}}
-\frac{\eta \nu \alpha \sqrt{k_F}}{\tilde{\mu}} \langle \nabla_{\tau} p_{FH}, P_{\tau} \vec{v}_c  \rangle_{\Gamma \Omega_{c0}}. \label{semi-localC}
\end{align}

\noindent \textbf{Step 3.} Correction:
\begin{align*}
&p^{Fh}|_{D_p}=p_{FH} + e^{Fh}|_{D_p},~~
p^{fh}|_{D_p}=p_{fH} + e^{fh}|_{D_p},~~
p^{mh}|_{D_p}=p_{mH} + e^{mh}|_{D_p},\\
&\vec{u}^{ch}|_{D_c}=\vec{u}_{cH} + e^{ch}|_{D_c},~~
p^h|_{D_c}=p_H + \xi^h|_{D_c}.
\end{align*}

Following the duality argument in literature \cite{heywood1988finite}, an auxiliary problem \eqref{dualityProblem} is introduced and some results are obtained. For $\ell \in (0,T]$ and $\phi \in L^2(0,\ell;[L^2(\Omega_{c0})]^d)$, find $(\boldsymbol{\Phi}(t),\psi(t)) \in X_c(\Omega_{c0}) \times L^2(\Omega_{c0})$ for $t \in [0,\ell)$ and $\forall (\boldsymbol{v}, q) \in X_c(\Omega_{c0}) \times L^2(\Omega_{c0})$ such that
\begin{equation}\label{dualityProblem}
\begin{split}
&\eta \Big(\boldsymbol{v},-\frac{\partial \boldsymbol{\Phi}}{\partial t} \Big)
+ a_{c\eta}(\boldsymbol{v},\boldsymbol{\Phi})
+ b_{\eta}(\boldsymbol{\Phi},q)
- b_{\eta}(\boldsymbol{v},\psi)
+ b_{N\eta}(\vec{u}_{cH},\boldsymbol{v},\boldsymbol{\Phi})\\
&~~~~~~~~~~~~~~~~~~~~~~~~~~~~~~~~~~~~~~~~~~~~~~~~~~~~~~~~~~~
+ b_{N\eta}(\boldsymbol{v},\vec{u}_{cH},\boldsymbol{\Phi})
=\eta (\boldsymbol{v},\phi),\\
&\boldsymbol{\Phi}(\ell)=0.
\end{split}
\end{equation}
Furthermore, the spatial semi-discrete scheme corresponding to the problem \eqref{dualityProblem} on the mesh grid $T_h(\Omega_{c0})$ reads: find $(\boldsymbol{\Phi}_h, \psi_h) \in L^{\infty}(0,\ell;X_c^h(\Omega_{c0})) \times L^{\infty}(0,\ell;Q^h(\Omega_{c0}))$ for $t \in [0,\ell)$ such that
\begin{equation}\label{semi-spacedualityProblem}
\begin{split}
&\eta \Big(\boldsymbol{v},-\frac{\partial \boldsymbol{\Phi}_h}{\partial t} \Big)
+ a_{c\eta}(\boldsymbol{v},\boldsymbol{\Phi}_h)
+ b_{\eta}(\boldsymbol{\Phi}_h,q)
- b_{\eta}(\boldsymbol{v},\psi_h)
+ b_{N\eta}(\vec{u}_{cH},\boldsymbol{v},\boldsymbol{\Phi}_h)\\
&~~~~~~~~~~~~~~~~~~~~~~~~~~~~~~~~~~~~~~~~~~~~~~~~~~~~~~~~~~~
+ b_{N\eta}(\boldsymbol{v},\vec{u}_{cH},\boldsymbol{\Phi}_h)
=\eta (\boldsymbol{v},\phi),\\
&\boldsymbol{\Phi}_h(\ell)=0.
\end{split}
\end{equation}
The stability results for the solution $(\boldsymbol{\Phi},\psi)$ are obtained as follows:
\begin{equation}\label{StabResult}
\sup_{0 \leq t \leq T} \|\boldsymbol{\Phi}(t)\|_{1,\Omega_{c0}}^2
+ \int_{0}^T \big(
\|\boldsymbol{\Phi}\|_{2,\Omega_{c0}}
+ \|\psi\|_{1,\Omega_{c0}}
+ \Big\|\frac{\partial \boldsymbol{\Phi}}{\partial t} \Big\|_{0,\Omega_{c0}}
\big)
\mathrm{d}t
\leq c \int_{0}^T \|\phi\|_{0,\Omega_{c0}}^2 \mathrm{d}t.
\end{equation}
Moreover, the following results hold
\begin{equation}\label{ErrorResult}
\|\boldsymbol{\Phi}-\boldsymbol{\Phi}_h\|_{0,\Omega_{c0}}
+ h(|\boldsymbol{\Phi}-\boldsymbol{\Phi}_h|_{1,\Omega_{c0}}
+  \|\psi - \psi_h\|_{0,\Omega_{c0}} )
\leq ch^2(\|\boldsymbol{\Phi}\|_{2,\Omega_{c0}}+\|\psi\|_{1,\Omega_{c0}}).
\end{equation}

\begin{lemma}\label{Lem1}
Under the assumptions of Theorem \ref{PTS-Results}, there hold
\begin{equation}\label{Lem1eqs}
\eta \int_{0}^T \|e^{ch}(t)\|_{0,\Omega_{c0}}^2 \mathrm{d}t
+ \eta \int_{0}^T \Big \|\frac{\partial e^{ch} }{\partial t}(t) \Big\|_{0,\Omega_{c0}}^2 \mathrm{d}t
+ \eta \int_{0}^T \Big \|\frac{\partial^2 e^{ch} }{\partial t^2}(t) \Big\|_{0,\Omega_{c0}}^2 \mathrm{d}t
\leq cH^{2(r+1)}.
\end{equation}
\end{lemma}

\begin{proof}
Taking $(\boldsymbol{v},q)=(e^{ch},\xi^h)$ in \eqref{semi-spacedualityProblem} with $\phi=e^{ch}$, we have
\begin{equation}\label{aim1}
\begin{split}
\eta \|e^{ch}\|_{0,\Omega_{c0}}^2
&= -\eta \frac{\mathrm{d} }{\mathrm{d} t} (e^{ch},\boldsymbol{\Phi}_h)
+ \eta \Big( \frac{\partial e^{ch}}{\partial t}, \boldsymbol{\Phi}_h \Big)
+ a_{c\eta}(e^{ch},\boldsymbol{\Phi}_h)
+ b_{\eta}(\boldsymbol{\Phi}_h,\xi^h)\\
&~~~~- b_{\eta}(e^{ch},\psi_h)
+ b_{N\eta}(\vec{u}_{cH},e^{ch},\boldsymbol{\Phi}_h)
+ b_{N\eta}(e^{ch},\vec{u}_{cH},\boldsymbol{\Phi}_h).
\end{split}
\end{equation}
Thanks to the assumption on the auxiliary grid $T_h(\Omega_c)$ that coincides with $T_h(\Omega_{c0})$ on $\Omega_{c0}$, \eqref{semi-localC} can be written as
\begin{align}
&\eta \Big( \frac{\partial e^{ch}}{\partial t}, \vec{v}_c \Big)
+ a_{c\eta}(e^{ch},\vec{v}_c)
+ b_{\eta}(\vec{v}_c,\xi^h)
- b_{\eta}(e^{ch},q)
+ b_{N\eta}(e^{ch},\vec{u}_{cH},\vec{v}_c)
+ b_{N\eta}(\vec{u}_{cH},e^{ch},\vec{v}_c) \nonumber\\
&=\eta \Big( \frac{\partial \vec{u}_{ch}}{\partial t} - \frac{\partial \vec{u}_{cH} }{\partial t},\vec{v}_c \Big)
+ a_{c\eta}(\vec{u}_{ch}-\vec{u}_{cH},\vec{v}_c)
+ b_{\eta}(\vec{v}_c,p_{h}-p_{H})
- b_{\eta}(\vec{u}_{ch}-\vec{u}_{cH},q) \nonumber\\
&~~~+ b_{N\eta}(\vec{u}_{ch}-\vec{u}_{cH},\vec{u}_{cH},\vec{v}_c)
+ b_{N\eta}(\vec{u}_{cH},\vec{u}_{ch}-\vec{u}_{cH},\vec{v}_c)
+ b_{N\eta}(\vec{u}_{ch}-\vec{u}_{cH},\vec{u}_{ch}-\vec{u}_{cH},\vec{v}_c) \nonumber\\
&~~~+\frac{\eta}{\rho}\langle p_{Fh} - p_{FH}, \vec{v}_c \cdot \vec{n}_c \rangle_{\Gamma \Omega_{c0}}
+ \frac{\eta \nu \alpha \sqrt{k_F}}{\tilde{\mu}} \langle \nabla_{\tau} (p_{Fh} - p_{FH}),P_{\tau} \vec{v}_c \rangle_{\Gamma \Omega_{c0}}. \label{AlgorithmB}
\end{align}
Subtracting \eqref{Semi-LFEM} from \eqref{CoupledSemiFormu} and taking $(v_F,v_f,v_m)=\boldsymbol{0},(\vec{v}_c,q)=(\boldsymbol{\Phi}_H,\psi_H)$, we get
\begin{align}
&\eta \Big( \frac{\partial \vec{u}_{ch}}{\partial t} - \frac{\partial \vec{u}_{cH}}{\partial t},\boldsymbol{\Phi}_H \Big)
+a_{c\eta}(\vec{u}_{ch}-\vec{u}_{cH},\boldsymbol{\Phi}_H)
+b_{\eta}(\boldsymbol{\Phi}_H,p_h-p_H)
-b_{\eta}(\vec{u}_{ch}-\vec{u}_{cH},\psi_H) \nonumber\\
&+b_{N\eta}(\vec{u}_{ch}-\vec{u}_{cH},\vec{u}_{cH},\boldsymbol{\Phi}_H)
+b_{N\eta}(\vec{u}_{cH},\vec{u}_{ch}-\vec{u}_{cH},\boldsymbol{\Phi}_H)
+b_{N\eta}(\vec{u}_{ch}-\vec{u}_{cH},\vec{u}_{ch}-\vec{u}_{cH},\boldsymbol{\Phi}_H)\nonumber\\
&+\frac{\eta}{\rho}\langle p_{Fh}-p_{FH},\boldsymbol{\Phi}_H \cdot \vec{n}_c  \rangle_{\Gamma \Omega_{c0}}
+\frac{\eta \nu \alpha \sqrt{k_F}}{\tilde{\mu}}\langle \nabla_{\tau}(p_{Fh}-p_{FH}),
P_{\tau} \boldsymbol{\Phi}_H \rangle_{\Gamma \Omega_{c0}}=0. \label{uh-uH}
\end{align}
Combining \eqref{AlgorithmB} and \eqref{uh-uH}, the formula \eqref{aim1} can be written as
\begin{align}
&\eta \|e^{ch}\|_{0,\Omega_{c0}}^2
= -\eta \frac{\mathrm{d} }{\mathrm{d} t} (e^{ch},\boldsymbol{\Phi}_h)
+ \eta \Big( \frac{\partial \vec{u}_{ch}}{\partial t} - \frac{\partial \vec{u}_{cH} }{\partial t},\boldsymbol{\Phi}_h-\boldsymbol{\Phi} \Big)
+ a_{c\eta}(\vec{u}_{ch}-\vec{u}_{cH},\boldsymbol{\Phi}_h-\boldsymbol{\Phi})\nonumber\\
&~~~+ b_{\eta}(\boldsymbol{\Phi}_h-\boldsymbol{\Phi},p_{h}-p_{H})
- b_{\eta}(\vec{u}_{ch}-\vec{u}_{cH},\psi_h-\psi)
+ b_{N\eta}(\vec{u}_{ch}-\vec{u}_{cH},\vec{u}_{cH},\boldsymbol{\Phi}_h-\boldsymbol{\Phi})\nonumber\\
&~~~+ b_{N\eta}(\vec{u}_{cH},\vec{u}_{ch}-\vec{u}_{cH},\boldsymbol{\Phi}_h-\boldsymbol{\Phi})
+ b_{N\eta}(\vec{u}_{ch}-\vec{u}_{cH},\vec{u}_{ch}-\vec{u}_{cH},\boldsymbol{\Phi}_h-\boldsymbol{\Phi}) \nonumber\\
&~~~+\frac{\eta}{\rho}\langle p_{Fh} - p_{FH}, (\boldsymbol{\Phi}_h-\boldsymbol{\Phi}) \cdot \vec{n}_c \rangle_{\Gamma \Omega_{c0}}
+ \frac{\eta \nu \alpha \sqrt{k_F}}{\tilde{\mu}} \langle \nabla_{\tau} (p_{Fh} - p_{FH}),P_{\tau} (\boldsymbol{\Phi}_h-\boldsymbol{\Phi}) \rangle_{\Gamma \Omega_{c0}}\nonumber\\
&~~~+\eta \Big( \frac{\partial \vec{u}_{ch}}{\partial t} - \frac{\partial \vec{u}_{cH}}{\partial t},\boldsymbol{\Phi}-\boldsymbol{\Phi}_H \Big)
+a_{c\eta}(\vec{u}_{ch}-\vec{u}_{cH},\boldsymbol{\Phi}-\boldsymbol{\Phi}_H)
+b_{\eta}(\boldsymbol{\Phi}-\boldsymbol{\Phi}_H,p_h-p_H)\nonumber\\
&~~~-b_{\eta}(\vec{u}_{ch}-\vec{u}_{cH},\psi-\psi_H)
+b_{N\eta}(\vec{u}_{ch}-\vec{u}_{cH},\vec{u}_{cH},\boldsymbol{\Phi}-\boldsymbol{\Phi}_H)
+b_{N\eta}(\vec{u}_{cH},\vec{u}_{ch}-\vec{u}_{cH},\boldsymbol{\Phi}-\boldsymbol{\Phi}_H)\nonumber\\
&~~~+b_{N\eta}(\vec{u}_{ch}-\vec{u}_{cH},\vec{u}_{ch}-\vec{u}_{cH},\boldsymbol{\Phi}-\boldsymbol{\Phi}_H)
+\frac{\eta}{\rho}\langle p_{Fh}-p_{FH},(\boldsymbol{\Phi}-\boldsymbol{\Phi}_H) \cdot \vec{n}_c  \rangle_{\Gamma \Omega_{c0}}\nonumber\\
&~~~+\frac{\eta \nu \alpha \sqrt{k_F}}{\mu}\langle \nabla_{\tau}(p_{Fh}-p_{FH}),
P_{\tau} (\boldsymbol{\Phi}-\boldsymbol{\Phi}_H) \rangle_{\Gamma \Omega_{c0}}. \label{aim2}
\end{align}
Using the H$\ddot{\mathrm{o}}$lder inequality, Young inequality and trace inequality, the formula \eqref{aim2} is bounded by
\begin{align}
&\eta \|e^{ch}\|_{0,\Omega_{c0}}^2
\leq -\eta \frac{\mathrm{d} }{\mathrm{d} t} (e^{ch},\boldsymbol{\Phi}_h)
+ c\eta \Big\| \frac{\partial \vec{u}_{ch}}{\partial t} - \frac{\partial \vec{u}_{cH} }{\partial t} \Big\|_{0,\Omega_{c0}} \|\boldsymbol{\Phi}_h - \boldsymbol{\Phi}\|_{0,\Omega_{c0}} \nonumber\\
&~~~+ c\eta( \|\vec{u}_{ch}-\vec{u}_{cH}\|_{1,\Omega_{c0}}
+ \|p_h-p_H\|_{0,\Omega_{c0}} )
(\|\boldsymbol{\Phi}_h-\boldsymbol{\Phi}\|_{1,\Omega_{c0}}
+\|\psi_h - \psi\|_{0,\Omega_{c0}})\nonumber\\
&~~~+c\eta \|p_{Fh}-p_{FH}\|_{1,\Omega_{p0}}
\|\boldsymbol{\Phi}_h-\boldsymbol{\Phi}\|_{1,\Omega_{c0}}
+ c\eta \Big\| \frac{\partial \vec{u}_{ch}}{\partial t} - \frac{\partial \vec{u}_{cH} }{\partial t} \Big\|_{0,\Omega_{c0}} \|\boldsymbol{\Phi} - \boldsymbol{\Phi}_H\|_{0,\Omega_{c0}}\nonumber\\
&~~~+ c\eta( \|\vec{u}_{ch}-\vec{u}_{cH}\|_{1,\Omega_{c0}}
+ \|p_h-p_H\|_{0,\Omega_{c0}} )
(\|\boldsymbol{\Phi}-\boldsymbol{\Phi}_H\|_{1,\Omega_{c0}}
+\|\psi - \psi_H\|_{0,\Omega_{c0}})\nonumber\\
&~~~+ c\eta \|p_{Fh}-p_{FH}\|_{1,\Omega_{p0}}
\|\boldsymbol{\Phi}-\boldsymbol{\Phi}_H\|_{1,\Omega_{c0}}. \label{eh1}
\end{align}
Utilizing the same idea employed to handle the interface term in references \cite{cao2010coupled, shan2013partitioned}, we have
\begin{align*}
&\frac{\eta \nu \alpha \sqrt{k_F}}{\mu}\langle \nabla_{\tau}(p_{Fh}-p_{FH}),
P_{\tau} (\boldsymbol{\Phi}_h-\boldsymbol{\Phi}) \rangle_{\Gamma \Omega_{c0}}\\
&\leq \frac{\eta \nu \alpha \sqrt{k_F}}{\mu} \|\nabla_{\tau}(p_{Fh}-p_{FH})\|_{H^{-1/2}(\Gamma \Omega_{p0})}
\|P_{\tau} (\boldsymbol{\Phi}_h-\boldsymbol{\Phi})\|_{H_{00}^{1/2}(\Gamma \Omega_{c0})}\\
&\leq \frac{\eta \nu \alpha \sqrt{k_F}}{\mu} \|p_{Fh}-p_{FH}\|_{H^{1/2}(\Gamma \Omega_{p0})}
\|\boldsymbol{\Phi}_h-\boldsymbol{\Phi}\|_{H_{00}^{1/2}(\Gamma \Omega_{c0})}\\
&\leq c\eta \|p_{Fh}-p_{FH}\|_{1,\Omega_{p0}}
\|\boldsymbol{\Phi}_h-\boldsymbol{\Phi}\|_{1,\Omega_{c0}},
\end{align*}
where $H_{00}^{1/2}(\Gamma \Omega_{c0})=X_c(\Omega_{c0})|_{\Gamma }$ and $(H_{00}^{1/2}(\Gamma \Omega_{c0}))'=H^{-1/2}(\Gamma \Omega_{p0})$ is a dual space of $H_{00}^{1/2}(\Gamma \Omega_{c0})$.
Using similar arguments, we arrive at
\begin{align*}
&\frac{\eta}{\rho}\langle p_{Fh} - p_{FH}, (\boldsymbol{\Phi}_h-\boldsymbol{\Phi}) \cdot \vec{n}_c \rangle_{\Gamma \Omega_{c0}}
\leq c \eta \|p_{Fh} - p_{FH}\|_{1,\Omega_{p0}} \|\boldsymbol{\Phi}_h-\boldsymbol{\Phi}\|_{1,\Omega_{c0}},\nonumber\\
&\frac{\eta \nu \alpha \sqrt{k_F}}{\mu}\langle \nabla_{\tau}(p_{Fh}-p_{FH}),
P_{\tau} (\boldsymbol{\Phi}-\boldsymbol{\Phi}_H) \rangle_{\Gamma \Omega_{c0}}
\leq c\eta \|p_{Fh}-p_{FH}\|_{1,\Omega_{p0}}
\|\boldsymbol{\Phi}-\boldsymbol{\Phi}_H\|_{1,\Omega_{c0}},\nonumber\\
&\frac{\eta}{\rho}\langle p_{Fh}-p_{FH},(\boldsymbol{\Phi}-\boldsymbol{\Phi}_H) \cdot \vec{n}_c  \rangle_{\Gamma \Omega_{c0}}
\leq c\eta \|p_{Fh}-p_{FH}\|_{1,\Omega_{p0}}
\|\boldsymbol{\Phi}-\boldsymbol{\Phi}_H\|_{1,\Omega_{c0}}.
\end{align*}
Integrating \eqref{eh1} from 0 to $T$ and using the H$\ddot{\mathrm{o}}$lder inequality and the results in \eqref{StabResult} \eqref{ErrorResult}, we get
\begin{align*}
&\eta \int_{0}^T \|e^{ch}(t)\|_{0,\Omega_{c0}}^2 \mathrm{d}t
\leq \eta (e^{ch}(0),\boldsymbol{\Phi}_h(0))\\
&~~~+c\eta H^2 \Big( \int_{0}^T \Big\|\frac{\partial \vec{u}_{ch}}{\partial t} - \frac{\partial \vec{u}_{cH} }{\partial t} \Big\|_{0,\Omega_{c0}}^2 \mathrm{d}t \Big)^{1/2}
\Big(\int_{0}^T \|\boldsymbol{\Phi}\|_{2,\Omega_{c0}}^2 + \|\psi\|_{1,\Omega_{c0}}^2 \mathrm{d}t \Big)^{1/2}\\
&~~~+c\eta H ( \int_{0}^T \|\vec{u}_{ch}-\vec{u}_{cH}\|_{1,\Omega_{c0}}^2 \mathrm{d}t )^{1/2} (\int_{0}^T \|\boldsymbol{\Phi}\|_{2,\Omega_{c0}}^2+ \|\psi\|_{1,\Omega_{c0}}^2 \mathrm{d}t )^{1/2}\\
&~~~+c\eta H ( \int_{0}^T \|p_h-p_H\|_{0,\Omega_{c0}}^2 \mathrm{d}t )^{1/2} (\int_{0}^T \|\boldsymbol{\Phi}\|_{2,\Omega_{c0}}^2+ \|\psi\|_{1,\Omega_{c0}}^2 \mathrm{d}t )^{1/2}\\
&~~~+c\eta H ( \int_{0}^T \|p_{Fh}-p_{FH}\|_{1,\Omega_{p0}}^2 \mathrm{d}t )^{1/2} (\int_{0}^T \|\boldsymbol{\Phi}\|_{2,\Omega_{c0}}^2+ \|\psi\|_{1,\Omega_{c0}}^2 \mathrm{d}t )^{1/2}\\
&\leq \eta(e^{ch}(0),\boldsymbol{\Phi}_h(0))
+\Big[ c\eta H^2 \Big( \int_{0}^T \Big\|\frac{\partial \vec{u}_{ch}}{\partial t} - \frac{\partial \vec{u}_{cH} }{\partial t} \Big\|_{0,\Omega_{c0}}^2 \mathrm{d}t \Big)^{1/2}\\
&~~~+c\eta H ( \int_{0}^T \|\vec{u}_{ch}-\vec{u}_{cH}\|_{1,\Omega_{c0}}^2 \mathrm{d}t )^{1/2}
+c\eta H ( \int_{0}^T \|p_h-p_H\|_{0,\Omega_{c0}}^2 \mathrm{d}t )^{1/2} \\
&~~~+c\eta H ( \int_{0}^T \|p_{Fh}-p_{FH}\|_{1,\Omega_{p0}}^2 \mathrm{d}t )^{1/2}
\Big]  \Big( \int_{0}^T \|e^{ch}(t)\|_{0,\Omega_{c0}}^2 \mathrm{d}t \Big)^{1/2}.
\end{align*}
Applying the stability results \eqref{StabResult}, we have
\begin{equation}
\eta \int_{0}^T \|e^{ch}(t)\|_{0,\Omega_{c0}}^2 \mathrm{d}t
\leq cH^{2(r+1)}.
\end{equation}
Similar to this way, we can obtain the other results of $\frac{\partial e^{ch}}{\partial t}$ and $\frac{\partial^2 e^{ch}}{\partial t^2}$ in \eqref{Lem1eqs}.
\end{proof}

\begin{lemma}\label{echConvergnece}
Under the assumptions of Theorem \ref{PTS-Results}, there holds
\begin{equation}\label{Lemma2}
\|e^{ch}(t)\|_{0,\Omega_{c0}}^2 \leq cH^{2(r+1)}.
\end{equation}
\end{lemma}

\begin{proof}
Differentiating \eqref{semi-localC} with respect to $t$ and repeat the same procedure
used in \eqref{AlgorithmB}, we obtain
\begin{align*}
&\eta \Big( \frac{\partial^2 e^{ch}}{\partial t^2 }, \vec{v}_c \Big)
+ a_{c\eta}(\frac{\partial e^{ch}}{\partial t},\vec{v}_c)
+ b_{\eta}(\vec{v}_c,\frac{\partial \xi^h}{\partial t})
- b_{\eta}(\frac{\partial e^{ch}}{\partial t},q)
+ b_{N\eta}(\frac{\partial e^{ch}}{\partial t},\vec{u}_{cH},\vec{v}_c)\\
&+ b_{N\eta}( e^{ch},\frac{\partial \vec{u}_{cH}}{\partial t},\vec{v}_c)
+ b_{N\eta}(\frac{\partial \vec{u}_{cH}}{\partial t},e^{ch},\vec{v}_c)
+ b_{N\eta}(\vec{u}_{cH},\frac{\partial e^{ch}}{\partial t},\vec{v}_c) \nonumber\\
&=\eta \Big( \frac{\partial^2 \vec{u}_{ch}}{\partial t^2} - \frac{\partial^2 \vec{u}_{cH} }{\partial t^2},\vec{v}_c \Big)
+ a_{c\eta}(\frac{\partial \vec{u}_{ch}}{\partial t}-\frac{\partial \vec{u}_{cH}}{\partial t},\vec{v}_c)
+ b_{\eta}(\vec{v}_c,\frac{\partial p_{h}}{\partial t}-\frac{\partial p_{H}}{\partial t})\nonumber\\
&~~~- b_{\eta}(\frac{\partial \vec{u}_{ch}}{\partial t}-\frac{\partial \vec{u}_{cH}}{\partial t},q)
+ b_{N\eta}(\frac{\partial \vec{u}_{ch}}{\partial t}-\frac{\partial \vec{u}_{cH}}{\partial t},\vec{u}_{cH},\vec{v}_c)
+ b_{N\eta}(\vec{u}_{ch}-\vec{u}_{cH},\frac{\partial \vec{u}_{cH}}{\partial t},\vec{v}_c)\\
&~~~+ b_{N\eta}(\frac{\partial \vec{u}_{cH}}{\partial t},\vec{u}_{ch}-\vec{u}_{cH},\vec{v}_c)
+ b_{N\eta}(\vec{u}_{cH},\frac{\partial \vec{u}_{ch}}{\partial t}-\frac{\vec{u}_{cH}}{\partial t},\vec{v}_c)\nonumber\\
&~~~+ b_{N\eta}(\frac{\partial \vec{u}_{ch}}{\partial t}-\frac{\partial \vec{u}_{cH}}{\partial t},\vec{u}_{ch}-\vec{u}_{cH},\vec{v}_c)
+ b_{N\eta}(\vec{u}_{ch}-\vec{u}_{cH},\frac{\partial \vec{u}_{ch}}{\partial t}-\frac{\partial \vec{u}_{cH}}{\partial t},\vec{v}_c) \nonumber\\
&~~~+\frac{\eta}{\rho}\langle \frac{\partial p_{Fh}}{\partial t} - \frac{\partial p_{FH}}{\partial t}, \vec{v}_c \cdot \vec{n}_c \rangle_{\Gamma \Omega_{c0}}
+ \frac{\eta \nu \alpha \sqrt{k_F}}{\tilde{\mu}} \langle \nabla_{\tau} (\frac{\partial p_{Fh}}{\partial t} - \frac{\partial p_{FH}}{\partial t}),P_{\tau} \vec{v}_c \rangle_{\Gamma \Omega_{c0}}.
\end{align*}
Taking $(\vec{v}_c,q)=(\frac{\partial e^{ch}}{\partial t},\frac{\partial \xi^h}{\partial t})$ in \eqref{semi-spacedualityProblem} with $\phi=e^{ch}$ and using the similar proof in Lemma \ref{Lem1}, we have
\begin{align}
&\frac{\eta}{2}\|e^{ch}(t)\|_{0,\Omega_{c0}}^2
\leq
\eta \Big( \int_{0}^T \|e^{ch}(t)\|_{0,\Omega_{c0}}^2 \mathrm{d}t \Big)^{1/2}\Big[ c H^2 \Big( \int_{0}^T \Big\|\frac{\partial^2 \vec{u}_{ch}}{\partial t^2} - \frac{\partial^2 \vec{u}_{cH} }{\partial t^2} \Big\|_{0,\Omega_{c0}}^2 \mathrm{d}t \Big)^{1/2} \nonumber\\
&+c H \Big( \int_{0}^T \Big\| \frac{\partial \vec{u}_{ch}}{\partial t}-\frac{\partial \vec{u}_{cH}}{\partial t} \Big\|_{1,\Omega_{c0}}^2 \mathrm{d}t \Big)^{1/2}
+c H \Big( \int_0^T \|\vec{u}_{ch}-\vec{u}_{cH}\|_{1,\Omega_{c0}}^2 \mathrm{d}t \Big)^{1/2} \nonumber\\
&+c H \Big( \int_{0}^T \Big\| \frac{\partial p_h}{\partial t}-\frac{\partial p_H}{\partial t} \Big\|_{0,\Omega_{c0}}^2 \mathrm{d}t \Big)^{1/2}
+c H \Big( \int_{0}^T \Big\| \frac{\partial p_{Fh}}{\partial t}-\frac{\partial p_{FH}}{\partial t} \Big\|_{1,\Omega_{p0}}^2 \mathrm{d}t \Big)^{1/2}
\Big]. \label{lemma2proof}
\end{align}
Using Lemma \ref{Lem1}, the inequality \eqref{lemma2proof} is bounded by \eqref{Lemma2}.
\end{proof}

\subsection{Fully discrete local finite element algorithm}\label{Sec-fully-discrete}
For $n=0,1,...,N-1(N >1)$, take $\boldsymbol{u}_H^0=\boldsymbol{u}_H(0),\boldsymbol{e}_0^h=\boldsymbol{e}^h(0)$ and find $\boldsymbol{u}_{n+1}^h=[p_{n+1}^{Fh}, p_{n+1}^{fh}, p_{n+1}^{mh}, \vec{u}_{n+1}^{ch}]^T$, $p_{n+1}^h$ by the following steps.\\
\noindent \textbf{Step 1.} Find global coarse grid solutions $\boldsymbol{u}_H^{n+1}=[p_{FH}^{n+1}, p_{fH}^{n+1}, p_{mH}^{n+1}, \vec{u}_{cH}^{n+1}]^T \in W^H$ and $p_H^{n+1} \in Q^H$, such that for all $\boldsymbol{v}=[v_F, v_f, v_m, \vec{v}_c]^T \in W^H$ and $q \in Q^H$ to satisfy:\\
(1.1) The fully discrete triple-porosity system in $\Omega_p$\\
\begin{equation}\label{Fully_pF}
\begin{split}
&\phi_F C_F \Big( \frac{p_{FH}^{n+1} - p_{FH}^n}{\Delta t},v_{F}\Big)
+\frac{k_F}{\tilde{\mu}}(\nabla p_{FH}^{n+1}, \nabla v_{F})
+\frac{\sigma^{\ast} k_f}{\tilde{\mu}}(p_{FH}^{n+1} - p_{fH}^n, v_{F})~~~~~~ \\
&- \langle v_{F}, \vec{u}_{cH}^n \cdot \vec{n}_c \rangle_{\Gamma}
=(q_F(t_{n+1}),v_{F}),
\end{split}
\end{equation}
\begin{equation}\label{Fully_pf}
\begin{split}
&\phi_f C_f \Big( \frac{p_{fH}^{n+1} - p_{fH}^n}{\Delta t}, v_{f}\Big)
+ \frac{k_f}{\tilde{\mu}}(\nabla p_{fH}^{n+1}, \nabla v_{f})
+ \frac{\sigma k_m}{\tilde{\mu}} (p_{fH}^{n+1} - p_{mh}^n,v_{f})~~~~~~~~~\\
&+ \frac{\sigma^{\ast} k_f}{\tilde{\mu}}(p_{fH}^{n+1} - p_{FH}^n, v_{f})
=(q_f(t_{n+1}), v_{f}),
\end{split}
\end{equation}
\begin{equation}\label{Fully_pm}
\begin{split}
&\phi_m C_m \Big( \frac{p_{mH}^{n+1} - p_{mH}^n}{\Delta t}, v_{m}\Big)
+\frac{k_m}{\tilde{\mu}}(\nabla p_{mH}^{n+1},\nabla v_{m})
+ \frac{\sigma k_m}{\tilde{\mu}} (p_{mH}^{n+1} - p_{fH}^n, v_{m})~~~~\\
&=(q_m(t_{n+1}),v_{m}).
\end{split}
\end{equation}
(1.2) The fully discrete conduit system in $\Omega_c$
\begin{align}
&\eta \Big( \frac{\vec{u}_{cH}^{n+1} - \vec{u}_{cH}^n}{\Delta t}, \vec{v}_{c} \Big)
+a_{c\eta}(\vec{u}_{cH}^{n+1},\vec{v}_{c})
+b_{\eta}(\vec{v}_{c},p_H^{n+1})-b_{\eta}(\vec{u}_{cH}^{n+1},q)
+b_{N\eta}(\vec{u}_{cH}^{n+1},\vec{u}_{cH}^{n+1},\vec{v}_{c}) \nonumber\\
&+\frac{\eta}{\rho} \langle p_{FH}^{n}, \vec{v}_{c} \cdot \vec{n}_c \rangle_{\Gamma}
+\frac{\eta \nu \alpha \sqrt{k_F}}{\tilde{\mu}} \langle \nabla_{\tau} p_{FH}^{n} ,P_{\tau} \vec{v}_{c} \rangle_{\Gamma}
=\eta (\vec{f}_c(t_{n+1}) ,\vec{v}_{c}). \label{Fully_uc}
\end{align}
\textbf{Step 2.} Find local fine grid corrections $\boldsymbol{e}_{n+1}^h=[e_{n+1}^{Fh}, e_{n+1}^{fh}, e_{n+1}^{mh}, e_{n+1}^{ch}]^T \in W^h(\Omega_0)$ and $\xi_{n+1}^h=Q^h(\Omega_{c0})$, such that the following equations hold for all $\boldsymbol{v}=[v_F, v_f, v_m, \vec{v}_c]^T \in W^h(\Omega_0)$ and $q \in Q^h(\Omega_{c0})$.\\
(2.1) The fully discrete triple-porosity system in local subdomain $\Omega_{p0}$
\begin{align}
&\phi_F C_F \Big( \frac{\partial e_{n+1}^{Fh}}{\partial t}, v_F \Big)
+ a_F(e_{n+1}^{Fh}, v_F)
+ \frac{\sigma^{\ast} k_f}{\tilde{\mu}}(e_{n+1}^{Fh}-e_{n}^{fh},v_F) \nonumber\\
&= (q_F(t_{n+1}), v_F)
- \Big[ \phi_F C_F \Big( \frac{\partial p_{FH}^{n+1}}{\partial t}, v_F \Big)
+ a_F(p_{FH}^{n+1}, v_F)
+ \frac{\sigma^{\ast} k_f}{\tilde{\mu}}(p_{FH}^{n+1}-p_{fH}^{n}, v_F)
\Big] \nonumber\\
&~~~~+\langle \vec{u}_{cH}^n \cdot \vec{n}_c,v_F \rangle_{\Gamma \Omega_{p0}}, ~~~~~~~~~~~~~~~~~~~~~~~~~~~~~~~~~~~~~~~~~~~~~~~~~~~~~~~~~~~~~~~~~~~~ \label{LocalFully_pF}
\end{align}
\begin{align}
&\phi_f C_f \Big( \frac{\partial e_{n+1}^{fh}}{\partial t}, v_f \Big)
+ a_f(e_{n+1}^{fh}, v_f)
+ \frac{\sigma k_m}{\tilde{\mu}}(e_{n+1}^{fh}-e_n^{mh},v_f)
+ \frac{\sigma^{\ast} k_f}{\tilde{\mu}}(e_{n+1}^{fh}-e_n^{Fh},v_F) \nonumber\\
&= (q_F(t_{n+1}), v_F)
- \Big[ \phi_f C_f \Big( \frac{\partial p_{fH}^{n+1}}{\partial t}, v_f \Big)
+ a_f(p_{fH}^{n+1}, v_f)
+ \frac{\sigma k_m}{\tilde{\mu}}(p_{fH}^{n+1}-p_{mH}^n,v_f) \nonumber\\
&~~~~+ \frac{\sigma^{\ast} k_f}{\tilde{\mu}}(p_{fH}^{n+1}-p_{FH}^n, v_F)
\Big],~~~~~~~~~~~~~~~~~~~~~~~~~~~~~~~~~~~~~~~~~~~~~~~~~~~~~~~~~~~~ \label{LocalFully_pf}
\end{align}
\begin{align}
&\phi_m C_m \Big( \frac{\partial e_{n+1}^{mh}}{\partial t}, v_m \Big)
+ a_m(e_{n+1}^{mh}, v_m)
+ \frac{\sigma k_m}{\tilde{\mu}}(e_{n+1}^{mh}-e_n^{fh},v_m)
= (q_m(t_{n+1}), v_m)\nonumber\\
&~~~~- \Big[ \phi_m C_m \Big( \frac{\partial p_{mH}^{n+1}}{\partial t}, v_m \Big)
+ a_m(p_{mH}^{n+1}, v_m)
+ \frac{\sigma k_m}{\tilde{\mu}}(p_{mH}^{n+1}-p_{fH}^n,v_m)
\Big].~~~~~~~~~~~~~~\label{LocalFully_pm}
\end{align}
(2.2) The fully discrete conduit system in local subdomain $\Omega_{c0}$
\begin{align}
&\eta \Big( \frac{e_{n+1}^{ch}-e_{n}^{ch}}{\Delta t}, \vec{v}_c \Big)
+ a_{c\eta}(e_{n+1}^{ch}, \vec{v}_c )
+ b_{\eta}(\vec{v}_c, \xi_{n+1}^h)
- b_{\eta}(e_{n+1}^{ch}, q)
+ b_{N\eta}(e_{n+1}^{ch}, \vec{u}_{cH}^{n+1}, \vec{v}_c) \nonumber\\
&+ b_{N\eta}(\vec{u}_{cH}^{n+1}, e_{n+1}^{ch}, \vec{v}_c) \nonumber\\
&=\eta (\vec{f}_c(t_{n+1}),\vec{v}_c)
- \Big[
\eta \Big( \frac{\vec{u}_{cH}^{n+1}-\vec{u}_{cH}^{n}}{\Delta t}, \vec{v}_c \Big)
+ a_{c\eta}(\vec{u}_{cH}^{n+1}, \vec{v}_c)
+ b_{\eta}(\vec{v}_c, p_{H}^{n+1})
- b_{\eta}(\vec{u}_{cH}^{n+1}, q) \nonumber\\
&+ b_{N\eta}(\vec{u}_{cH}^{n+1},\vec{u}_{cH}^{n+1},\vec{v}_c)
\Big]
-\frac{\eta}{\rho}\langle p_{FH}^{n}, \vec{v}_c \cdot \vec{n}_c \rangle_{\Gamma \Omega_{c0}}
-\frac{\eta \nu \alpha \sqrt{k_F}}{\tilde{\mu}} \langle \nabla_{\tau} p_{FH}^{n}, P_{\tau} \vec{v}_c  \rangle_{\Gamma \Omega_{c0}}. \label{LocalFully_uc}
\end{align}

\noindent \textbf{Step 3.} Correction:
\begin{align*}
&p_{n+1}^{Fh}|_{D_p}=p_{FH}^{n+1} + e_{n+1}^{Fh}|_{D_p},~~
p_{n+1}^{fh}|_{D_p}=p_{fH}^{n+1} + e_{n+1}^{fh}|_{D_p},~~
p_{n+1}^{mh}|_{D_p}=p_{mH}^{n+1} + e_{n+1}^{mh}|_{D_p},\\
&\vec{u}^{ch}_{n+1}|_{D_c}=\vec{u}_{cH}^{n+1} + e_{n+1}^{ch}|_{D_c},~~
p_{n+1}^h|_{D_c}=p_H^{n+1} + \xi_{n+1}^h|_{D_c}.
\end{align*}

In the following, we will present some error estimates based on the fully discrete local finite element algorithm. Some results in whole domain are first given. Assume that $\boldsymbol{u}_{\mu}^{n+1}=[p_{F\mu}^{n+1}, p_{f\mu}^{n+1}, p_{m\mu}^{n+1}, \vec{u}_{c\mu}^{n+1}]^T, p_{\mu}^{n+1}$ and
$\boldsymbol{u}_{\mu}(t_{n+1})=[p_{F\mu}(t_{n+1}), p_{f\mu}(t_{n+1}), p_{m\mu}(t_{n+1}), \vec{u}_{c\mu}(t_{n+1})]^T,p_{\mu}$\\
$(t_{n+1}) (\mu=h, H)$ are obtained from
\eqref{fullypF}-\eqref{fullyuc} and \eqref{CoupledSemiFormu}, respectively.
Defining
\begin{align*}
&E_{F\mu}^{n+1}=p_{F\mu}^{n+1}-p_{F\mu}(t_{n+1}),
~~~~E_{f\mu}^{n+1}=p_{f\mu}^{n+1}-p_{f\mu}(t_{n+1}),
~~~~E_{m\mu}^{n+1}=p_{m\mu}^{n+1}-p_{m\mu}(t_{n+1}),\\
&E_{c\mu}^{n+1}=\vec{u}_{c\mu}^{n+1} - \vec{u}_{c\mu}(t_{n+1}),
~~~~\delta_{\mu}^{n+1}=p_{\mu}^{n+1}-p_{\mu}(t_{n+1}),
\end{align*}
the following bound holds.

\begin{lemma}\label{LemSemitime}
Under the boundedness of \eqref{BDness} and the rescaling factor $\eta$ satisfies the condition
$\eta \leq \frac{C_K \tilde{\mu}}{6\nu \alpha^2}$,
it is valid for $k=1,2,...,N$ that
\begin{align}
&\eta \|E_{c\mu}^{k}\|_0^2
+\phi_F C_F \|E_{F\mu}^{k}\|_0^2
+\phi_f C_f \|E_{f\mu}^{k}\|_0^2
+\phi_m C_m \|E_{m\mu}^{k}\|_0^2
+ \eta \sum_{n=0}^{k-1}\|E_{c\mu}^{n+1}-E_{c\mu}^{n}\|_0^2 \nonumber\\
&+\phi_F C_F \sum_{n=0}^{k-1}\|E_{F\mu}^{n+1}-E_{F\mu}^{n}\|_0^2
+\phi_f C_f \sum_{n=0}^{k-1}\|E_{f\mu}^{n+1}-E_{f\mu}^n\|_0^2
+\sum_{n=0}^{k-1}\phi_m C_m \|E_{m\mu}^{n+1}-E_{m\mu}^{n}\|_0^2 \nonumber\\
&+2\nu \eta C_K \Delta t  |E_{c\mu}^{k}|_1^2
+\frac{k_F \Delta t}{\tilde{\mu}}|E_{F\mu}^{k}|_1^2
+\frac{2k_f\Delta t}{\tilde{\mu}}\sum_{n=0}^{k-1} |E_{f\mu}^{n+1}|_1^2
+\frac{2k_m \Delta t}{\tilde{\mu}} \sum_{n=0}^{k-1}|E_{m\mu}^{n+1}|_1^2 \nonumber\\
&\leq c \Delta t^2. \label{LemSemitimeformu}
\end{align}
\end{lemma}

\begin{proof}
For $(\vec{v}_c,q )\in X_c^{\mu} \times Q^{\mu}$, taking $[v_F, v_f, v_m]=\boldsymbol{0}$ in \eqref{CoupledSemiFormu} and using the Taylor expansion with the integral remainder, we have
\begin{align}
&\eta \Big( \frac{\vec{u}_{c\mu}(t_{n+1})-\vec{u}_{c\mu}(t_n)}{\Delta t},\vec{v}_c \Big)
+a_{c\eta}(\vec{u}_{c\mu}(t_{n+1}),\vec{v}_c)
+b_{\eta}(\vec{v}_c,p_{\mu}(t_{n+1}))
-b_{\eta}(\vec{u}_{c\mu}(t_{n+1}),q) \nonumber\\
&+b_{N\eta}(\vec{u}_{c\mu}(t_{n+1}),\vec{u}_{c\mu}(t_{n+1}),\vec{v}_c)
+\frac{\eta}{\rho}\langle p_{F\mu}(t_{n+1}),\vec{v}_c \cdot \vec{n}_c \rangle_{\Gamma}
+\frac{\eta \nu \alpha \sqrt{k_F}}{\tilde{\mu}}\langle \nabla_{\tau} p_{F\mu}(t_{n+1}),P_{\tau} \vec{v}_c \rangle_{\Gamma} \nonumber\\
&=\eta(\vec{f}_c(t_{n+1}),\vec{v}_c)
+\frac{1}{\Delta t}\int_{t_n}^{t_{n+1}} (t_{n+1} -t )\Big( \frac{\partial^2 \vec{u}_{c\mu}}{\partial t^2}(t),\vec{v}_c\Big)\mathrm{d}t. \label{ucmu}
\end{align}
Subtracting \eqref{ucmu} from
\eqref{fullyuc}, we obtain
\begin{align}
&\eta \Big( \frac{E_{c\mu}^{n+1} - E_{c\mu}^n}{\Delta t},\vec{v}_c\Big)
+a_{c\eta}(E_{c\mu}^{n+1},\vec{v}_c)
+b_{\eta}(\vec{v}_c,\delta_{\mu}^{n+1})
-b_{\eta}(E_{c\mu}^{n+1},q)
+b_{N\eta}(E_{c\mu}^{n+1},\vec{u}_{c\mu}(t_{n+1}),\vec{v}_c)\nonumber\\
&+b_{N\eta}(\vec{u}_{c\mu}(t_{n+1}),E_{c\mu}^{n+1},\vec{v}_c)
+b_{N\eta}(E_{c\mu}^{n+1},E_{c\mu}^{n+1},\vec{v}_c)
+\frac{\eta}{\rho}\langle E_{F\mu}^n + p_{F\mu}(t_n)-p_{F\mu}(t_{n+1}),\vec{v}_c \cdot \vec{n}_c \rangle_{\Gamma}\nonumber\\
&+\frac{\eta \nu \alpha \sqrt{k_F}}{\tilde{\mu}}
\langle  \nabla_{\tau}(E_{F\mu}^n + p_{F\mu}(t_n)-p_{F\mu}(t_{n+1})),P_{\tau} \vec{v}_c \rangle_{\Gamma}\nonumber\\
&=\frac{1}{\Delta t}\int_{t_n}^{t_{n+1}} (t-t_{n+1})\Big( \frac{\partial^2 \vec{u}_{c\mu}}{\partial t^2}(t),\vec{v}_c\Big)\mathrm{d}t. \label{UsePropety1}
\end{align}
Taking $(\vec{v}_c,q)=(2\Delta t E_{c\mu}^{n+1},2\Delta t \delta_{\mu}^{n+1})$ in \eqref{UsePropety1} and using $2(a-b,a)=|a|^2-|b|^2+|a-b|^2$ and the skew-symmetrized property $b_{N\eta}(\vec{u}_{c\mu},\vec{v}_c,\vec{v}_c)=0$, we get
\begin{align}
&\eta \|E_{c\mu}^{n+1}\|_0^2 - \eta \|E_{c\mu}^{n}\|_0^2 + \eta \|E_{c\mu}^{n+1}-E_{c\mu}^{n}\|_0^2
+ 4\nu \eta \Delta t \|\mathbb{D}(E_{c\mu}^{n+1})\|_0^2 \nonumber\\
&=2\int_{t_n}^{t_{n+1}} (t-t_{n+1})\Big( \frac{\partial^2 \vec{u}_{c\mu}}{\partial t^2}(t),E_{c\mu}^{n+1} \Big)\mathrm{d}t
-b_{N\eta}(E_{c\mu}^{n+1},\vec{u}_{c\mu}(t_{n+1}),2\Delta t E_{c\mu}^{n+1}) \nonumber\\
&~~~-\frac{2 \eta \Delta t}{\rho} \langle E_{F\mu}^n + p_{F\mu}(t_n)-p_{F\mu}(t_{n+1}),E_{c\mu}^{n+1} \cdot \vec{n}_c \rangle_{\Gamma} \nonumber\\
&~~~-\frac{2\Delta t \eta \nu \alpha \sqrt{k_F}}{\tilde{\mu}}\langle \nabla_{\tau}(E_{F\mu}^n+p_{F\mu}(t_n)-p_{F\mu}(t_{n+1})),P_{\tau} E_{c\mu}^{n+1} \rangle_{\Gamma} \nonumber\\
&:=T_1+T_2+T_3+T_4. \label{TFour}
\end{align}
Utilizing the H$\ddot{\mathrm{o}}$lder inequality, $T_1$ is bounded by
\begin{align*}
T_1 &\leq 2\Big[\int_{t_n}^{t_{n+1}} (t-t_{n+1})^2\mathrm{d}t \Big]^{1/2}
\Big[\int_{t_n}^{t_{n+1}} \Big|\Big( \frac{\partial^2 \vec{u}_{c\mu}}{\partial t^2}(t),E_{c\mu}^{n+1} \Big) \Big|^2 \mathrm{d}t \Big]^{1/2}\\
&\leq \frac{2\Delta t^{3/2}}{3} \Big( \int_{t_n}^{t_{n+1}} \Big\| \frac{\partial^2 \vec{u}_{c\mu}}{\partial t^2}(t)\Big\|_0^2 \mathrm{d}t \Big)^{1/2} \|E_{c\mu}^{n+1}\|_0\\
&\leq c\Delta t^2 \int_{t_n}^{t_{n+1}} \Big\| \frac{\partial^2 \vec{u}_{c\mu}}{\partial t^2}(t) \Big\|_0^2 \mathrm{d}t
+ \frac{\eta C_P^2 C_B^2 \Delta t}{\nu C_K} \|E_{c\mu}^{n+1}\|_0^2.
~~~~~~~~~~~~~~~~~~~~~~~~~~
\end{align*}
By applying the H$\ddot{\mathrm{o}}$lder inequality, Poinc$\acute{\mathrm{a}}$re inequality, Sobolev's imbedding property $W^{1,2}(\Omega_c)$\\
$\hookrightarrow L^6(\Omega_c)$ and the boundedness assumption in \eqref{BDness}, we have
\begin{align*}
T_2 &\leq 2 \eta \Delta t \|E_{c\mu}^{n+1}\|_{L^6} \|\vec{u}_{c\mu}(t_{n+1})\|_{W^{1,3}}
\|E_{c\mu}^{n+1}\|_0\\
&\leq 2 C_P C_B \eta \Delta t |E_{c\mu}^{n+1}|_1 \|E_{c\mu}^{n+1}\|_0\\
&\leq \nu \eta C_K \Delta t |E_{c\mu}^{n+1}|_1^2
+\frac{\eta C_P^2 C_B^2 \Delta t}{\nu C_K}\|E_{c\mu}^{n+1}\|_0^2.~~~~~~~~~~~~~~~~~~~~~~~~~~~~~~~~~~~~~
\end{align*}
For the interface terms, repeating the same proof in Lemma \ref{Lem1} and using the trace inequality, we derive that
\begin{align*}
T_3
&\leq  \frac{2\eta C_t^2 \Delta t}{\rho} \|E_{F\mu}^n\|_0^{1/2} |E_{F\mu}^n|_1^{1/2} \|E_{c\mu}^{n+1}\|_0^{1/2} |E_{c\mu}^{n+1}|_1^{1/2}\\
&~~~+ \frac{2\eta C_t^2 \Delta t}{\rho}\|p_{F\mu}(t_n)-p_{F\mu}(t_{n+1})\|_1 |E_{c\mu}^{n+1}|_1\\
&\leq \frac{\eta C_t^4 \Delta t}{\nu C_K \rho^2}\|E_{F\mu}^n\|_0^2
+\frac{2\eta \nu \alpha^2 k_F\Delta t}{C_K \tilde{\mu}^2}|E_{F\mu}^n|_1^2
+\frac{\eta C_K \tilde{\mu}^2 C_t^4 \Delta t}{8\nu k_F\rho^2 \alpha^2}\|E_{c\mu}^{n+1}\|_0^2 \\
&~~~+ c\Delta t^2 \int_{t_n}^{t_{n+1}}\Big\| \frac{\partial p_{F\mu}}{\partial t}(t) \Big\|_1^2 \mathrm{d}t
+ \frac{\nu \eta C_K \Delta t}{2} |E_{c\mu}^{n+1}|_1^2,~~~~~~~~~~~~~~~~~~~~~~~~~
\end{align*}
and
\begin{align*}
T_4
&\leq \frac{2\Delta t \eta \nu \alpha \sqrt{k_F}}{\tilde{\mu}} \|\nabla_{\tau} E_{F\mu}^n\|_{H^{-1/2}(\Gamma)}\|P_{\tau} E_{c\mu}^{n+1}\|_{H_{00}^{1/2}(\Gamma)}\\
&~~~+\frac{2\Delta t \eta \nu \alpha \sqrt{k_F}}{\tilde{\mu}} \|\nabla_{\tau}(p_{F\mu}(t_n)-p_{F\mu}(t_{n+1}))\|_{H^{-1/2}(\Gamma)}\|P_{\tau} E_{c\mu}^{n+1}\|_{H_{00}^{1/2}(\Gamma)}\\
&\leq \frac{2\Delta t \eta \nu \alpha \sqrt{k_F}}{\tilde{\mu}} \|E_{F\mu}^n\|_{H^{1/2}(\Gamma)}\| E_{c\mu}^{n+1}\|_{H_{00}^{1/2}(\Gamma)}\\
&~~~+\frac{2\Delta t \eta \nu \alpha \sqrt{k_F}}{\tilde{\mu}} \|p_{F\mu}(t_n)-p_{F\mu}(t_{n+1})\|_{H^{1/2}(\Gamma)}\| E_{c\mu}^{n+1}\|_{H_{00}^{1/2}(\Gamma)}\\
&\leq \frac{4\eta \nu \alpha^2 k_F \Delta t }{C_K \tilde{\mu}^2}|E_{F\mu}^{n}|_1^2
+c\Delta t^2 \int_{t_n}^{t_{n+1}}\Big\| \frac{\partial p_{F\mu}}{\partial t}(t) \Big\|_1^2 \mathrm{d}t
+\frac{\nu \eta C_K \Delta t}{2} |E_{c\mu}^{n+1}|_1^2.
\end{align*}
Collecting the estimates of $T_1 - T_4$ and using the Korn inequality, we deduce that
\begin{align}
&\eta \|E_{c\mu}^{n+1}\|_0^2 - \eta \|E_{c\mu}^{n}\|_0^2 + \eta \|E_{c\mu}^{n+1}-E_{c\mu}^{n}\|_0^2
+ 2\nu \eta C_K \Delta t |E_{c\mu}^{n+1}|_1^2 \nonumber\\
&\leq \Big( \frac{2\eta C_P^2 C_B^2 \Delta t}{\nu C_K}+
\frac{\eta C_K \tilde{\mu}^2 C_t^4 \Delta t}{8\nu k_F\rho^2 \alpha^2} \Big)\|E_{c\mu}^{n+1}\|_0^2
+ c\Delta t^2 \int_{t_n}^{t_{n+1}} \Big\| \frac{\partial^2 \vec{u}_{c\mu}}{\partial t^2}(t) \Big\|_0^2 \mathrm{d}t \nonumber\\
&~~~+ c\Delta t^2 \int_{t_n}^{t_{n+1}} \Big\| \frac{\partial p_{F\mu}}{\partial t}(t) \Big\|_1^2 \mathrm{d}t
+ \frac{6\eta \nu \alpha^2 k_F\Delta t}{C_K \tilde{\mu}^2} |E_{F\mu}^n|_1^2
+\frac{\eta C_t^4 \Delta t}{\nu C_K \rho^2}\|E_{F\mu}^n\|_0^2. \label{Ecmu}
\end{align}
Similar to obtain the estimation of $E_{c\mu}^{n+1}$, we estimate $E_{F\mu}^{n+1}, E_{f\mu}^{n+1}$ and $E_{m\mu}^{n+1}$ as follows
\begin{align}
&\phi_F C_F \|E_{F\mu}^{n+1}\|_0^2
-\phi_F C_F \|E_{F\mu}^{n}\|_0^2
+\phi_F C_F \|E_{F\mu}^{n+1}-E_{F\mu}^{n}\|_0^2
+\frac{k_F \Delta t}{\tilde{\mu}} |E_{F\mu}^{n+1}|_1^2 \nonumber\\
&=2\int_{t_n}^{t_{n+1}} (t-t_{n+1})\Big( \frac{\partial^2 p_{F\mu}}{\partial t^2}(t),E_{F\mu}^{n+1} \Big)\mathrm{d}t
+ 2\Delta t \langle (E_{c\mu}^n + \vec{u}_{c\mu}(t_n) - \vec{u}_{c\mu}(t_{n+1}))\cdot \vec{n}_c , E_{F\mu}^{n+1} \rangle_{\Gamma} \nonumber\\
&~~~-\frac{2\sigma^{\ast}k_f \Delta t}{\mu}(E_{F\mu}^{n+1}-E_{f\mu}^n+p_{f\mu}(t_{n+1})-p_{f\mu}(t_n),E_{F\mu}^{n+1}) \nonumber\\
&\leq \frac{4\sigma^{\ast}k_f \Delta t}{\mu}\|E_{F\mu}^{n+1}\|_0^2
+ c\Delta t^2 \int_{t_n}^{t_{n+1}} \Big\| \frac{\partial^2 p_{F\mu}}{\partial t^2}(t) \Big\|_0^2 \mathrm{d}t
+ 2\nu \eta C_K \Delta t |E_{c\mu}^n|_1^2
+ \frac{\tilde{\mu}^2 C_t^8 \Delta t}{2 \nu \eta C_K k_F^2}\|E_{c\mu}^n\|_0^2 \nonumber\\
&+c\Delta t^2 \int_{t_n}^{t_{n+1}}\Big\| \frac{\partial \vec{u}_{c\mu}}{\partial t}(t) \Big\|_1^2 \mathrm{d}t
+c\Delta t^2 \int_{t_n}^{t_{n+1}}\Big\| \frac{\partial^2 p_{f\mu}}{\partial t^2}(t) \Big\|_0^2 \mathrm{d}t
+\frac{2\sigma^{\ast}k_f \Delta t}{\mu}\|E_{f\mu}^{n}\|_0^2, \label{EFmu}
\end{align}
\begin{align}
&\phi_f C_f \|E_{f\mu}^{n+1}\|_0^2
-\phi_f C_f \|E_{f\mu}^{n}\|_0^2
+\phi_f C_f \|E_{f\mu}^{n+1}-E_{f\mu}^n\|_0^2
+\frac{2k_f\Delta t}{\tilde{\mu}} |E_{f\mu}^{n+1}|_1^2 \nonumber\\
&=2\int_{t_n}^{t_{n+1}} (t-t_{n+1})\Big( \frac{\partial^2 p_{f\mu}}{\partial t^2}(t),E_{f\mu}^{n+1} \Big)\mathrm{d}t \nonumber\\
&~~~-\frac{2\sigma k_m \Delta t}{\tilde{\mu}}(E_{f\mu}^{n+1}-E_{m\mu}^n
+p_{m\mu}(t_{n+1})-p_{m\mu}(t_n),E_{f\mu}^{n+1}) \nonumber\\
&~~~-\frac{2\sigma^{\ast}k_f \Delta t}{\tilde{\mu}}(E_{f\mu}^{n+1}-E_{F\mu}^n+p_{F\mu}(t_{n+1})-p_{F\mu}(t_n),E_{f\mu}^{n+1}) \nonumber\\
&\leq \frac{4(\sigma k_m+\sigma^{\ast} k_f) \Delta t}{\tilde{\mu}}\|E_{f\mu}^{n+1}\|_0^2
+ c\Delta t^2 \int_{t_n}^{t_{n+1}} \Big\| \frac{\partial^2 p_{f\mu}}{\partial t^2}(t) \Big\|_0^2 \mathrm{d}t
+\frac{2\sigma k_m \Delta t}{\tilde{\mu}} \|E_{m\mu}^n\|_0^2 \nonumber\\
&+\frac{2\sigma^{\ast} k_f \Delta t}{\tilde{\mu}} \|E_{F\mu}^n\|_0^2
+c\Delta t^2 \int_{t_n}^{t_{n+1}} \Big\| \frac{\partial^2 p_{m\mu}}{\partial t^2}(t) \Big\|_0^2 \mathrm{d}t
+c\Delta t^2 \int_{t_n}^{t_{n+1}} \Big\| \frac{\partial^2 p_{F\mu}}{\partial t^2}(t) \Big\|_0^2 \mathrm{d}t, \label{Efmu}~~~
\end{align}
and
\begin{align}
&\phi_m C_m \|E_{m\mu}^{n+1}\|_0^2
-\phi_m C_m \|E_{m\mu}^{n}\|_0^2
+\phi_m C_m \|E_{m\mu}^{n+1}-E_{m\mu}^{n}\|_0^2
+\frac{2k_m \Delta t}{\tilde{\mu}} |E_{m\mu}^{n+1}|_1^2 \nonumber\\
&=2\int_{t_n}^{t_{n+1}} (t-t_{n+1})\Big( \frac{\partial^2 p_{m\mu}}{\partial t^2}(t),E_{m\mu}^{n+1} \Big)\mathrm{d}t \nonumber\\
&~~~-\frac{2\sigma k_m \Delta t}{\tilde{\mu}}(E_{m\mu}^{n+1}-E_{f\mu}^n+p_{f\mu}(t_{n+1})-p_{f\mu}(t_n),E_{m\mu}^{n+1})\nonumber\\
&\leq \frac{4\sigma k_m \Delta t}{\tilde{\mu}} \|E_{m\mu}^{n+1}\|_0^2
+c\Delta t^2 \int_{t_n}^{t_{n+1}} \Big\| \frac{\partial^2 p_{m\mu}}{\partial t^2}(t) \Big\|_0^2 \mathrm{d}t
+\frac{2\sigma k_m \Delta t}{\tilde{\mu}}\|E_{f\mu}^n\|_0^2 \nonumber\\
&+c\Delta t^2 \int_{t_n}^{t_{n+1}} \Big\| \frac{\partial^2 p_{f\mu}}{\partial t^2}(t) \Big\|_0^2 \mathrm{d}t. \label{Emmu}
~~~~~~~~~~~~~~~~~~~~~~~~~~~~~~~~~~~~~~~~~~~~~~~~~~~~~~~~~~~~~
\end{align}
Combining \eqref{Ecmu}-\eqref{Emmu} and summing it from $n=0$ to $n=k-1(k=1,2,...,N)$, we obtain
\begin{align*}
&\eta \|E_{c\mu}^{k}\|_0^2 - \eta \|E_{c\mu}^{0}\|_0^2 + \eta \sum_{n=0}^{k-1}\|E_{c\mu}^{n+1}-E_{c\mu}^{n}\|_0^2
+\phi_F C_F \|E_{F\mu}^{k}\|_0^2
-\phi_F C_F \|E_{F\mu}^{0}\|_0^2 \\
&+\phi_F C_F \sum_{n=0}^{k-1}\|E_{F\mu}^{n+1}-E_{F\mu}^{n}\|_0^2
+\phi_f C_f \|E_{f\mu}^{k}\|_0^2
-\phi_f C_f \|E_{f\mu}^{0}\|_0^2
+\phi_f C_f \sum_{n=0}^{k-1}\|E_{f\mu}^{n+1}-E_{f\mu}^n\|_0^2\\
&+\phi_m C_m \|E_{m\mu}^{k}\|_0^2
-\phi_m C_m \|E_{m\mu}^{0}\|_0^2
+\sum_{n=0}^{k-1}\phi_m C_m \|E_{m\mu}^{n+1}-E_{m\mu}^{n}\|_0^2
+2\nu \eta C_K \Delta t ( |E_{c\mu}^{k}|_1^2-|E_{c\mu}^{0}|_1^2 )\\
&+\frac{k_F \Delta t}{\tilde{\mu}} (|E_{F\mu}^{k}|_1^2
-|E_{F\mu}^0|_1^2 )
+\frac{2k_f\Delta t}{\tilde{\mu}}\sum_{n=0}^{k-1} |E_{f\mu}^{n+1}|_1^2
+\frac{2k_m \Delta t}{\tilde{\mu}} \sum_{n=0}^{k-1}|E_{m\mu}^{n+1}|_1^2\\
&\leq \Big( \frac{2C_P^2 C_B^2 \Delta t}{\nu C_K}+
\frac{C_K \tilde{\mu}^2 C_t^4 \Delta t}{8\nu k_F\rho^2 \alpha^2} +
\frac{\tilde{\mu}^2 C_t^8 \Delta t}{2 \nu \eta^2 C_K k_F^2} \Big)
\sum_{n=0}^{k-1}\eta(\|E_{c\mu}^{n+1}\|_0^2 + \|E_{c\mu}^n\|_0^2) \\
&~~~+\Big(\frac{6\sigma^{\ast}k_f \Delta t}{\tilde{\mu} \phi_F C_F}
+\frac{\eta C_t^4 \Delta t}{\nu C_K \phi_F C_F \rho^2}\Big)
\sum_{n=0}^{k-1}\phi_F C_F(\|E_{F\mu}^{n+1}\|_0^2+\|E_{F\mu}^n\|_0^2)\\
&~~~+\frac{6(\sigma k_m+\sigma^{\ast} k_f) \Delta t}{\tilde{\mu}\phi_f C_f}
\sum_{n=0}^{k-1}\phi_f C_f(\|E_{f\mu}^{n+1}\|_0^2+\|E_{f\mu}^{n}\|_0^2)
+\frac{6\sigma k_m \Delta t}{\tilde{\mu}\phi_m C_m} \sum_{n=0}^{k-1}\phi_m C_m(\|E_{m\mu}^{n+1}\|_0^2+\|E_{m\mu}^{n}\|_0^2)\\
&~~~+ c\Delta t^2  \Big\| \frac{\partial^2 \vec{u}_{c\mu}}{\partial t^2}(t) \Big\|_{L^2(0,T;L^2)}^2
+c\Delta t^2 \Big\| \frac{\partial^2 p_{F\mu}}{\partial t^2}(t) \Big\|_{L^2(0,T;L^2)}^2
+c\Delta t^2 \Big\| \frac{\partial^2 p_{f\mu}}{\partial t^2}(t) \Big\|_{L^2(0,T;L^2)}^2 \\
&~~~+c\Delta t^2 \Big\| \frac{\partial^2 p_{m\mu}}{\partial t^2}(t) \Big\|_{L^2(0,T;L^2)}^2
+c\Delta t^2 \Big\| \frac{\partial \vec{u}_{c\mu}}{\partial t}(t) \Big\|_{L^2(0,T;H^1)}^2
+ c\Delta t^2 \Big\| \frac{\partial p_{F\mu}}{\partial t}(t) \Big\|_{L^2(0,T;H^1)}^2.
\end{align*}
Using the discrete Gronwall Lemma \ref{LemGronwall},
there exists a positive constant $\eta_0=\frac{C_K \tilde{\mu}}{6\nu \alpha^2}$ such that when $\eta \leq \eta_0$, we show that
\begin{align*}
&\eta \|E_{c\mu}^{k}\|_0^2
+\phi_F C_F \|E_{F\mu}^{k}\|_0^2
+\phi_f C_f \|E_{f\mu}^{k}\|_0^2
+\phi_m C_m \|E_{m\mu}^{k}\|_0^2
+ \eta \sum_{n=0}^{k-1}\|E_{c\mu}^{n+1}-E_{c\mu}^{n}\|_0^2\\
&+\phi_F C_F \sum_{n=0}^{k-1}\|E_{F\mu}^{n+1}-E_{F\mu}^{n}\|_0^2
+\phi_f C_f \sum_{n=0}^{k-1}\|E_{f\mu}^{n+1}-E_{f\mu}^n\|_0^2
+\sum_{n=0}^{k-1}\phi_m C_m \|E_{m\mu}^{n+1}-E_{m\mu}^{n}\|_0^2\\
&+2\nu \eta C_K \Delta t  |E_{c\mu}^{k}|_1^2
+\frac{k_F \Delta t}{\tilde{\mu}}|E_{F\mu}^{k}|_1^2
+\frac{2k_f\Delta t}{\tilde{\mu}}\sum_{n=0}^{k-1} |E_{f\mu}^{n+1}|_1^2
+\frac{2k_m \Delta t}{\tilde{\mu}} \sum_{n=0}^{k-1}|E_{m\mu}^{n+1}|_1^2\\
&\leq \exp{ \Big( \frac{C T}{1-C \Delta t}  \Big) }
\Big\{ c\Delta t^2  \Big\| \frac{\partial^2 \vec{u}_{c\mu}}{\partial t^2}(t) \Big\|_{L^2(0,T;L^2)}^2
+ c\Delta t^2 \Big\| \frac{\partial^2 p_{F\mu}}{\partial t^2}(t) \Big\|_{L^2(0,T;L^2)}^2\\
&~~~~+ c\Delta t^2 \Big\| \frac{\partial^2 p_{f\mu}}{\partial t^2}(t) \Big\|_{L^2(0,T;L^2)}^2
+c\Delta t^2 \Big\| \frac{\partial^2 p_{m\mu}}{\partial t^2}(t) \Big\|_{L^2(0,T;L^2)}^2
+c\Delta t^2 \Big\| \frac{\partial \vec{u}_{c\mu}}{\partial t}(t) \Big\|_{L^2(0,T;H^1)}^2\\
&~~~+ c\Delta t^2 \Big\| \frac{\partial p_{F\mu}}{\partial t}(t) \Big\|_{L^2(0,T;H^1)}^2
\Big\}
\end{align*}
where
\begin{align*}
C&=\max \Big\{ \frac{2C_P^2 C_B^2 \Delta t}{\nu C_K}+
\frac{C_K \tilde{\mu}^2 C_t^4 \Delta t}{8\nu k_F\rho^2 \alpha^2} +
\frac{\tilde{\mu}^2 C_t^8 \Delta t}{2 \nu \eta^2 C_K k_F^2},
\frac{6\sigma^{\ast}k_f \Delta t}{\tilde{\mu} \phi_F C_F}
+\frac{\eta C_t^4 \Delta t}{\nu C_K \phi_F C_F \rho^2},\\
&~~~~~~~~~~~~~\frac{6(\sigma k_m+\sigma^{\ast} k_f) \Delta t}{\tilde{\mu}\phi_f C_f},
\frac{6\sigma k_m \Delta t}{\tilde{\mu}\phi_m C_m} \Big\}.
\end{align*}
Under the boundedness of \eqref{BDness} and the convergence results in Lemma \ref{LemsemiConvergence}, using the triangle inequality, we conclude \eqref{LemSemitimeformu}.
\end{proof}
\\

There are some notations denoted by
\begin{align*}
&E_{n+1}^{Fh}=e_{n+1}^{Fh}-e^{Fh}(t_{n+1}),
~~~~E_{n+1}^{fh}=e_{n+1}^{fh}-e^{fh}(t_{n+1}),
~~~~E_{n+1}^{mh}=e_{n+1}^{mh}-e^{mh}(t_{n+1}),\\
&E_{n+1}^{ch}=e_{n+1}^{ch}-e^{ch}(t_{n+1}),
~~~~\delta_{n+1}^h=\xi_{n+1}^h-\xi^h(t_{n+1}),
\end{align*}
where $e^{ih}(t_{n+1}), e_{n+1}^{ih} (i=F,f,m,c), \xi^h(t_{n+1})$ and $\xi_{n+1}^h$
are defined in Section \ref{Sec-semi-discrete} and Section \ref{Sec-fully-discrete}.

\begin{lemma}\label{ech-ect}
For $k=1,2,...,N$, there is the following inequality
\begin{align}\label{lem54equ}
\eta \|E_{k}^{ch}\|_{0,\Omega_{c0}}^2 + \eta \sum_{n=0}^{k-1} \|E_{n+1}^{ch}-E_n^{ch}\|_{0,\Omega_{c0}}^2
+\nu C_K \eta \Delta t \sum_{n=0}^{k-1} |E_{n+1}^{ch}|_{1,\Omega_{c0}}^2
\leq c \Delta t^2.
\end{align}
\end{lemma}
\begin{proof}
Subtracting \eqref{semi-localC} from \eqref{LocalFully_uc}, we have
\begin{align}
&\eta\Big( \frac{E_{n+1}^{ch}-E_n^{ch}}{\Delta t},\vec{v}_c \Big)
+a_{c\eta}(E_{n+1}^{ch},\vec{v}_c)
+b_{\eta}(\vec{v}_c,\delta_{n+1}^h)
-b_{\eta}(E_{n+1}^{ch},q)
+b_{N\eta}(E_{n+1}^{ch},\vec{u}_{cH}^{n+1},\vec{v}_c) \nonumber\\
&+b_{N\eta}(\vec{u}_{cH}^{n+1},E_{n+1}^{ch},\vec{v}_c)
+b_{N\eta}(E_{cH}^{n+1},e^{ch}(t_{n+1}),\vec{v}_c)
+b_{N\eta}(e^{ch}(t_{n+1}),E_{cH}^{n+1},\vec{v}_c) \nonumber\\
&=-\Big[ \eta \Big( \frac{E_{cH}^{n+1} - E_{cH}^n}{\Delta t},\vec{v}_c \Big)
+a_{c\eta}(E_{cH}^{n+1},\vec{v}_c)
+b_{\eta}(\vec{v}_c,\delta_{H}^{n+1})
-b_{\eta}(E_{cH}^{n+1},q) \nonumber\\
&~~~+b_{N\eta}(E_{cH}^{n+1},\vec{u}_{cH}(t_{n+1}),\vec{v}_c)
+b_{N\eta}(\vec{u}_{cH}(t_{n+1}),E_{cH}^{n+1},\vec{v}_c)
+b_{N\eta}(E_{cH}^{n+1},E_{cH}^{n+1},\vec{v}_c) \nonumber\\
&~~~+\frac{\eta}{\rho}\langle E_{FH}^{n}+p_{FH}(t_n)-p_{FH}(t_{n+1}),\vec{v}_c \cdot \vec{n}_c \rangle_{\Gamma\Omega_{c0}} \nonumber\\
&~~~+\frac{\eta \nu \alpha \sqrt{k_F}}{\tilde{\mu}}\langle \nabla_{\tau}(E_{FH}^n+p_{FH}(t_n)-p_{FH}(t_{n+1})),P_{\tau} \vec{v}_c \rangle_{\Gamma \Omega_{c0}} \Big] \nonumber\\
&~~~-\frac{1}{\Delta t} \int_{t_n}^{t_{n+1}} (t_{n+1}-t)\Big( \frac{\partial^2 \vec{u}_{cH}}{\partial t^2}(t),\vec{v}_c \Big) \mathrm{d}t
-\frac{1}{\Delta t}\int_{t_n}^{t_{n+1}}(t_{n+1}-t)\Big( \frac{\partial^2 e^{ch}}{\partial t^2}(t),\vec{v}_c \Big) \mathrm{d}t. \label{Lem54First}
\end{align}
Taking $(\vec{v}_c,q)=(2\Delta t E_{n+1}^{ch},2\Delta t \delta_{n+1}^h)$ and using  \eqref{UsePropety1} in \eqref{Lem54First}, we get
\begin{align*}
&\eta \|E_{n+1}^{ch}\|_{0,\Omega_{c0}}^2 - \eta \|E_n^{ch}\|_{0,\Omega_{c0}}^2 + \eta \|E_{n+1}^{ch}-E_n^{ch}\|_{0,\Omega_{c0}}^2
+4\nu \eta \Delta t \|\mathbb{D}(E_{n+1}^{ch})\|_{0,\Omega_{c0}}^2 \\
&=
-\frac{1}{\Delta t}\int_{t_n}^{t_{n+1}}(t_{n+1}-t)\Big( \frac{\partial^2 e^{ch}}{\partial t^2}(t),2\Delta t E_{n+1}^{ch} \Big) \mathrm{d}t
-b_{N\eta}(E_{n+1}^{ch},\vec{u}_{cH}^{n+1},2\Delta t E_{n+1}^{ch})\\
&~~~~-b_{N\eta}(E_{cH}^{n+1},e^{ch}(t_{n+1}),2\Delta t E_{n+1}^{ch})
-b_{N\eta}(e^{ch}(t_{n+1}),E_{cH}^{n+1},2\Delta t E_{n+1}^{ch}).
\end{align*}
By the H$\ddot{\mathrm{o}}$lder inequality, the Young inequality,
the Poincar$\mathrm{\acute{e}}$-Friedriches inequality, the inverse inequality, the nonlinear properties and Lemma \ref{LemSemitime}, we show that
\begin{align*}
&-\frac{1}{\Delta t}\int_{t_n}^{t_{n+1}}(t_{n+1}-t)\Big( \frac{\partial^2 e^{ch}}{\partial t^2}(t),2\Delta t E_{n+1}^{ch} \Big) \mathrm{d}t \\
&\leq c \Delta t^{3/2} \Big( \int_{t_n}^{t_{n+1}} \Big\| \frac{\partial^2 e^{ch}}{\partial t^2}(t) \Big\|_{0,\Omega_{c0}}^2 \mathrm{d}t \Big)^{1/2} \|E_{n+1}^{ch}\|_{0,\Omega_{c0}}\\
&\leq c \Delta t^2 \int_{t_n}^{t_{n+1}} \Big\| \frac{\partial^2 e^{ch}}{\partial t^2}(t) \Big\|_{0,\Omega_{c0}}^2 \mathrm{d}t
+ c \Delta t \|E_{n+1}^{ch}\|_{0,\Omega_{c0}}^2,
\end{align*}
\begin{align*}
&-b_{N\eta}(E_{n+1}^{ch},\vec{u}_{cH}^{n+1}-\vec{u}_c(t_{n+1})+\vec{u}_c(t_{n+1}),2\Delta t E_{n+1}^{ch})\\
&=-b_{N\eta}(E_{n+1}^{ch},\vec{u}_{cH}^{n+1}-\vec{u}_c(t_{n+1}),2\Delta t E_{n+1}^{ch})
-b_{N\eta}(E_{n+1}^{ch},\vec{u}_c(t_{n+1}),2\Delta t E_{n+1}^{ch})\\
&\leq 2C_N \Delta t \|E_{n+1}^{ch}\|_{0,\Omega_{c0}}^{1/2} |E_{n+1}^{ch}|_{1,\Omega_{c0}}^{1/2} |\vec{u}_{cH}^{n+1} - \vec{u}_c(t_{n+1})|_{1,\Omega_{c0}} |E_{n+1}^{ch}|_{1,\Omega_{c0}}
+2C_N C_B \Delta t \|E_{n+1}^{ch}\|_{0,\Omega_{c0}}^2\\
&\leq \nu C_K \eta \Delta t |E_{n+1}^{ch}|_{1,\Omega_{c0}}^2
+ \Big(\frac{C_N^2 C(T)^2 \Delta t}{\nu C_K \eta}
+ 2C_N C_B \Delta t \Big) \|E_{n+1}^{ch}\|_{0,\Omega_{c0}}^2,
~~~~~~~~~~~~~~~~~~~~~~~~~~~~~~~
\end{align*}
\begin{align*}
&-b_{N\eta}(E_{cH}^{n+1},e^{ch}(t_{n+1}),2\Delta t E_{n+1}^{ch})\\
&\leq
2C_N \Delta t \|E_{cH}^{n+1}\|_{0,\Omega_{c0}}^{1/2} |E_{cH}^{n+1}|_{1,\Omega_{c0}}^{1/2} h^{-1} \|e^{ch}(t_{n+1})\|_{0,\Omega_{c0}} |E_{n+1}^{ch}|_{1,\Omega_{c0}} \\
&\leq
\nu C_K \eta \Delta t |E_{n+1}^{ch}|_{1,\Omega_{c0}}^2
+ c \Delta t^3,
~~~~~~~~~~~~~~~~~~~~~~~~~~~~~~~~~~~~~~~~~~~~~~~~~~~~~~~~~~~~~~~~~~~~~~`
\end{align*}
and
\begin{align*}
&-b_{N\eta}(e^{ch}(t_{n+1}),E_{cH}^{n+1},2\Delta t E_{n+1}^{ch})\\
&\leq 2 C_N \Delta t \|e^{ch}(t_{n+1})\|_{0,\Omega_{c0}}^{1/2} |e^{ch}(t_{n+1})|_{1,\Omega_{c0}}^{1/2} |E_{cH}^{n+1}|_{1,\Omega_{c0}} |E_{n+1}^{ch}|_{1,\Omega_{c0}}\\
&\leq \nu C_K \eta \Delta t |E_{n+1}^{ch}|_{1,\Omega_{c0}}^2
+ c \Delta t^3.
~~~~~~~~~~~~~~~~~~~~~~~~~~~~~~~~~~~~~~~~~~~~~~~~~~~~~~~~~~~~~~~~~~~~~~~
\end{align*}
Summing it from $n=0$ to $n=k-1$ and using the Korn inequality, we obtain
\begin{align*}
&\eta \|E_{k}^{ch}\|_{0,\Omega_{c0}}^2 + \eta \sum_{n=0}^{k-1} \|E_{n+1}^{ch}-E_n^{ch}\|_{0,\Omega_{c0}}^2
+\nu C_K \eta \Delta t \sum_{n=0}^{k-1} |E_{n+1}^{ch}|_{1,\Omega_{c0}}^2 \\
&\leq
\Big(\frac{C_N^2 C(T)^2 \Delta t}{\nu C_K \eta}
+ 2C_N C_B \Delta t \Big) \sum_{n=0}^{k-1} \|E_{n+1}^{ch}\|_{0,\Omega_{c0}}^2
+c \Delta t^2 \Big\| \frac{\partial^2 e^{ch}}{\partial t^2}(t) \Big\|_{L^2(0,T;L^2)}^2
+c \Delta t^2.
\end{align*}
By the discrete Gronwall Lemma \ref{LemGronwall}, we derive \eqref{lem54equ}.
\end{proof}

Similar to the proof of Lemma \ref{echConvergnece} and Lemma \ref{ech-ect}, we repeat the above process and have the following lemmas.
\begin{lemma}\label{tripleeih}
Under the assumptions of Theorem \ref{PTS-Results}, there holds
\begin{align*}
\|e^{ih}(t)\|_{0,\Omega_{p0}}^2 \leq cH^{2(r+1)},~~~~(i=F,f,m).
\end{align*}
\end{lemma}

\begin{lemma}\label{tripleEkih}
For $k=1,2,...,N$, there is the following inequality
\begin{align*}
\|E_{k}^{ih}\|_{0,\Omega_{p0}}^2 \leq c \Delta t^2,~~~~(i=F,f,m).
\end{align*}
\end{lemma}

\section{Fully discrete local parallel finite element algorithm}
In this section, a fully discrete scheme of local parallel finite element algorithm is proposed as follows.
\begin{algorithm}\label{algorithm3}{(Local Parallel Finite Element Algorithm)}

\textbf{Step 1}(Decoupled marching schemes for the low frequency solution).\\
(1)In the triple porous media region $\Omega_p$, given $(p_{FH}^0,p_{fH}^0,p_{mH}^0 )=(P_H^F p_F^0, P_H^f p_f^0, P_H^m p_m^0)$,
find a global coarse grid solution $(p_{FH}^{n+1}, p_{fH}^{n+1}, p_{mH}^{n+1}) \in X_p^H$, such that for all $(v_F, v_f, v_m) \in X_p^H$,
\begin{align}
&\phi_F C_F \Big( \frac{p_{FH}^{n+1} - p_{FH}^n}{\Delta t},v_{F}\Big)
+\frac{k_F}{\tilde{\mu}}(\nabla p_{FH}^{n+1}, \nabla v_{F})
+\frac{\sigma^{\ast} k_f}{\tilde{\mu}}(p_{FH}^{n+1} - p_{fH}^n, v_{F})
- \int_{\Gamma} v_{F} \vec{u}_{cH}^n \cdot \vec{n}_c \mathrm{d}\Gamma \nonumber\\
&=(q_F(t_{n+1}),v_{F}),\label{pFFully}
\end{align}
\begin{align}
&\phi_f C_f \Big( \frac{p_{fH}^{n+1} - p_{fH}^n}{\Delta t}, v_{f}\Big)
+ \frac{k_f}{\tilde{\mu}}(\nabla p_{fH}^{n+1}, \nabla v_{f})
+ \frac{\sigma k_m}{\tilde{\mu}} (p_{fH}^{n+1} - p_{mH}^n,v_{f} ) \nonumber\\
&~~~~~~~~~~~~~~~~~~~~~~~~~~~~~~~~~~~~~~~~~~~~~~~~
+ \frac{\sigma^{\ast} k_f}{\tilde{\mu}}(p_{fH}^{n+1} - p_{FH}^n, v_{f})
=(q_f(t_{n+1}), v_{f}),\label{pfFully}
\end{align}
\begin{align}
&\phi_m C_m \Big( \frac{p_{mH}^{n+1} - p_{mH}^n}{\Delta t}, v_{m}\Big)
+\frac{k_m}{\tilde{\mu}}(\nabla p_{mH}^{n+1},\nabla v_{m})
+ \frac{\sigma k_m}{\tilde{\mu}} (p_{mH}^{n+1} - p_{fH}^n, v_{m} )~~~~~~~~~~~~~~~~~~~~~\nonumber\\
&=(q_m(t_{n+1}),v_{m}).\label{pmFully}
\end{align}
(2)In the conduit region $\Omega_c$, given $\vec{u}_{cH}^0=P_H^c \vec{u}_c^0$, find a global coarse grid solution $(\vec{u}_{cH}^{n+1}, p_{H}^{n+1}) \in X_c^H \times Q^H$, such that for all $(\vec{v}_{c}, q) \in X_c^H \times Q^H$,
\begin{align}
&\eta \Big( \frac{\vec{u}_{cH}^{n+1} - \vec{u}_{cH}^n}{\Delta t}, \vec{v}_{c} \Big)
+a_{c\eta}(\vec{u}_{cH}^{n+1},\vec{v}_{c})
+b_{\eta}(\vec{v}_{c},p_{H}^{n+1})-b_{\eta}(\vec{u}_{cH}^{n+1},q)
+b_{N\eta}(\vec{u}_{cH}^{n+1},\vec{u}_{cH}^{n+1},\vec{v}_{c}) \nonumber\\
&+\frac{\eta}{\rho}\int_{\Gamma} p_{FH}^{n} \vec{v}_{c} \cdot \vec{n}_c \mathrm{d}\Gamma
+\frac{\eta \nu \alpha \sqrt{k_F}}{\tilde{\mu}}\int_{\Gamma}  \nabla_{\tau}p_{FH}^{n} \cdot P_{\tau} \vec{v}_{c} \mathrm{d}\Gamma
=\eta (\vec{f}_c(t_{n+1}) ,\vec{v}_{c}). ~~~~~~~~~~~~~\label{pcFully}
\end{align}

\textbf{Step 2}(Partition overlapping subdomains).
Divide $\Omega_p, \Omega_c$ into a series of disjoint subdomains
$\{D_{pj}\}_1^M, \{D_{cj'}\}_1^{M^{'}}$, then enlarge these subdomains to obtain
$\{\Omega_{pj}\}_{1}^M, \{\Omega_{cj'}\}_{1}^{M^{'}}$ which align with $T_h^p$ and $T_h^c$.

\textbf{Step 3}(Decoupled marching schemes for the high frequency solution).\\
(1)In the triple porous media region $\Omega_p$, find a local fine correction $(e_{FH,j}^{n+1}, e_{fH,j}^{n+1}, e_{mH,j}^{n+1}) \in X_p^h(\Omega_{pj })$, such that for all $(v_F, v_f, v_m) \in X_p^h(\Omega_{pj})$,
\begin{align}
&\phi_F C_F \Big( \frac{e_{Fh,j}^{n+1} - e_{Fh,j}^n}{\Delta t},v_{F}\Big)
+\frac{k_F}{\tilde{\mu}}(\nabla e_{Fh,j}^{n+1}, \nabla v_{F})
+\frac{\sigma^{\ast} k_f}{\tilde{\mu}}(e_{Fh,j}^{n+1} - e_{fh,j}^n, v_{F} )\nonumber\\
&=(q_F(t_{n+1}),v_{F})
-\Big[ \phi_F C_F \Big( \frac{p_{FH}^{n+1} - p_{FH}^n}{\Delta t},v_{F}\Big)
+ \frac{k_F}{\tilde{\mu}}(\nabla p_{FH}^{n+1}, \nabla v_{F})
+ \frac{\sigma^{\ast} k_f}{\tilde{\mu}}(p_{Fh}^{n+1} - p_{fH}^n, v_{F}) \Big] \nonumber\\
&+ \langle  \vec{u}_{cH}^n \cdot \vec{n}_c, v_{F} \rangle_{\Gamma \Omega_{pj}},\label{pFlocalFully}
\end{align}
\begin{align}
&\phi_f C_f \Big( \frac{e_{fh,j}^{n+1} - e_{fh,j}^n}{\Delta t}, v_{f}\Big)
+ \frac{k_f}{\tilde{\mu}}(\nabla e_{fh,j}^{n+1}, \nabla v_{f})
+ \frac{\sigma k_m}{\tilde{\mu}} (e_{fh,j}^{n+1} - e_{mh,j}^n,v_{f}) \nonumber\\
&+ \frac{\sigma^{\ast} k_f}{\tilde{\mu}} (e_{fh,j}^{n+1} - e_{Fh,j}^n, v_{f}) \nonumber\\
&=(q_f(t_{n+1}), v_{f})
-\Big[ \phi_f C_f \Big( \frac{p_{fH}^{n+1} - p_{fH}^n}{\Delta t}, v_{f}\Big)
+ \frac{k_f}{\tilde{\mu}}(\nabla p_{fH}^{n+1}, \nabla v_{f})
+ \frac{\sigma k_m}{\tilde{\mu}} (p_{fH}^{n+1} - p_{mH}^n,v_{f}) \nonumber\\
&~~+\frac{\sigma^{\ast} k_f}{\tilde{\mu}} (p_{fH}^{n+1} - p_{FH}^n, v_{f})
\Big],
~~~~~~~~~~~~~~~~~~~~~~~~~~~~~~~~~~~~~~~~~~~~~~~~~~~~~~~~~~~~~~~~~~~~
\label{pflocalFully}
\end{align}
\begin{align}
&\phi_m C_m \Big( \frac{e_{mh,j}^{n+1} - e_{mh,j}^n}{\Delta t}, v_{m}\Big)
+\frac{k_m}{\tilde{\mu}}(\nabla e_{mh,j}^{n+1},\nabla v_{m})
+\frac{\sigma k_m}{\tilde{\mu}} (e_{mh,j}^{n+1} - \xi_{fh,j}^n, v_{m})
=(q_m(t_{n+1}),v_{m}) \nonumber\\
&-\Big[ \phi_m C_m \Big( \frac{p_{mH}^{n+1} - p_{mH}^n}{\Delta t}, v_{m}\Big)
+\frac{k_m}{\tilde{\mu}}(\nabla p_{mH}^{n+1},\nabla v_{m})
+\frac{\sigma k_m}{\tilde{\mu}} (p_{mH}^{n+1} - p_{fH}^n, v_{m})
\Big].\label{pmlocalFully}
\end{align}
(2)In the conduit region $\Omega_c$, find a local fine correction  $(e_{ch,j'}^{n+1}, \delta_{h,j'}^{n+1}) \in X_c^h(\Omega_{cj'}) \times Q^h(\Omega_{cj'})$, such that for all $(\vec{v}_{c}, q) \in X_c^h(\Omega_{cj'}) \times Q^h(\Omega_{cj'})$,
\begin{align}
&\eta \Big( \frac{e_{ch,j'}^{n+1} - e_{ch,j'}^n}{\Delta t}, \vec{v}_{c} \Big)
+a_{c\eta}(e_{ch,j'}^{n+1},\vec{v}_{c})
+b_{\eta}(\vec{v}_{c},\delta_{h,j'}^{n+1})-b_{\eta}(e_{ch,j'}^{n+1},q)
+b_{N\eta}(e_{ch,j'}^{n+1},\vec{u}_{cH}^{n+1},\vec{v}_{c}) \nonumber\\
&+b_{N\eta}(\vec{u}_{cH}^{n+1},e_{ch,j'}^{n+1},\vec{v}_{c}) \nonumber\\
&=\eta (\vec{f}_c(t_{n+1}) ,\vec{v}_{c})
-\Big[ \eta \Big( \frac{\vec{u}_{cH}^{n+1} - \vec{u}_{cH}^n}{\Delta t}, \vec{v}_{c} \Big)
+ a_{c\eta}(\vec{u}_{cH}^{n+1},\vec{v}_{c})
+b_{\eta}(\vec{v}_{c},p_{cH}^{n+1})-b_{\eta}(\vec{u}_{cH}^{n+1},q) \nonumber\\
&+b_{N\eta}(\vec{u}_{cH}^{n+1},\vec{u}_{cH}^{n+1},\vec{v}_{c})\Big]
- \frac{\eta}{\rho} \langle p_{FH}^{n}, \vec{v}_{c} \cdot \vec{n}_c \rangle_{\Gamma \Omega_{cj'}}
- \frac{\eta \nu \alpha \sqrt{k_F}}{\tilde{\mu}} \langle \nabla_{\tau}p_{FH}^{n}, P_{\tau} \vec{v}_{c} \rangle_{\Gamma \Omega_{cj'}}.~~~~~~~~~~ \label{uclocalFully}
\end{align}

\textbf{Step 4}(Correction of data).
For $n=0,1,2,...,N-1, j=1,2,3,...,M, j^{'}=1,2,3,...M^{'}$,
\begin{align*}
p^{Fh,j}_{n+1}|_{D_{pj }}&= p_{FH}^{n+1} + e_{Fh,j}^{n+1}|_{D_{pj }},\\
p^{fh,j}_{n+1}|_{D_{pj }}&= p_{fH}^{n+1} + e_{fh,j}^{n+1}|_{D_{pj }},\\
p^{mh,j}_{n+1}|_{D_{pj }}&= p_{mH}^{n+1} + e_{mh,j}^{n+1}|_{D_{pj }},\\
\vec{u}^{ch,j'}_{n+1}|_{D_{cj' }} &= \vec{u}_{cH}^{n+1} + e_{ch,j'}^{n+1}|_{D_{cj' }},\\
p^{h,j'}_{n+1}|_{D_{cj' }} &= p_{H}^{n+1} + \delta_{h,j'}^{n+1}|_{D_{cj' }}.
\end{align*}
\end{algorithm}
The series of steps in the above algorithm is shown in Figure \ref{Pic3}.
\begin{figure}[htbp]
  \centering
  \includegraphics[width=15cm]{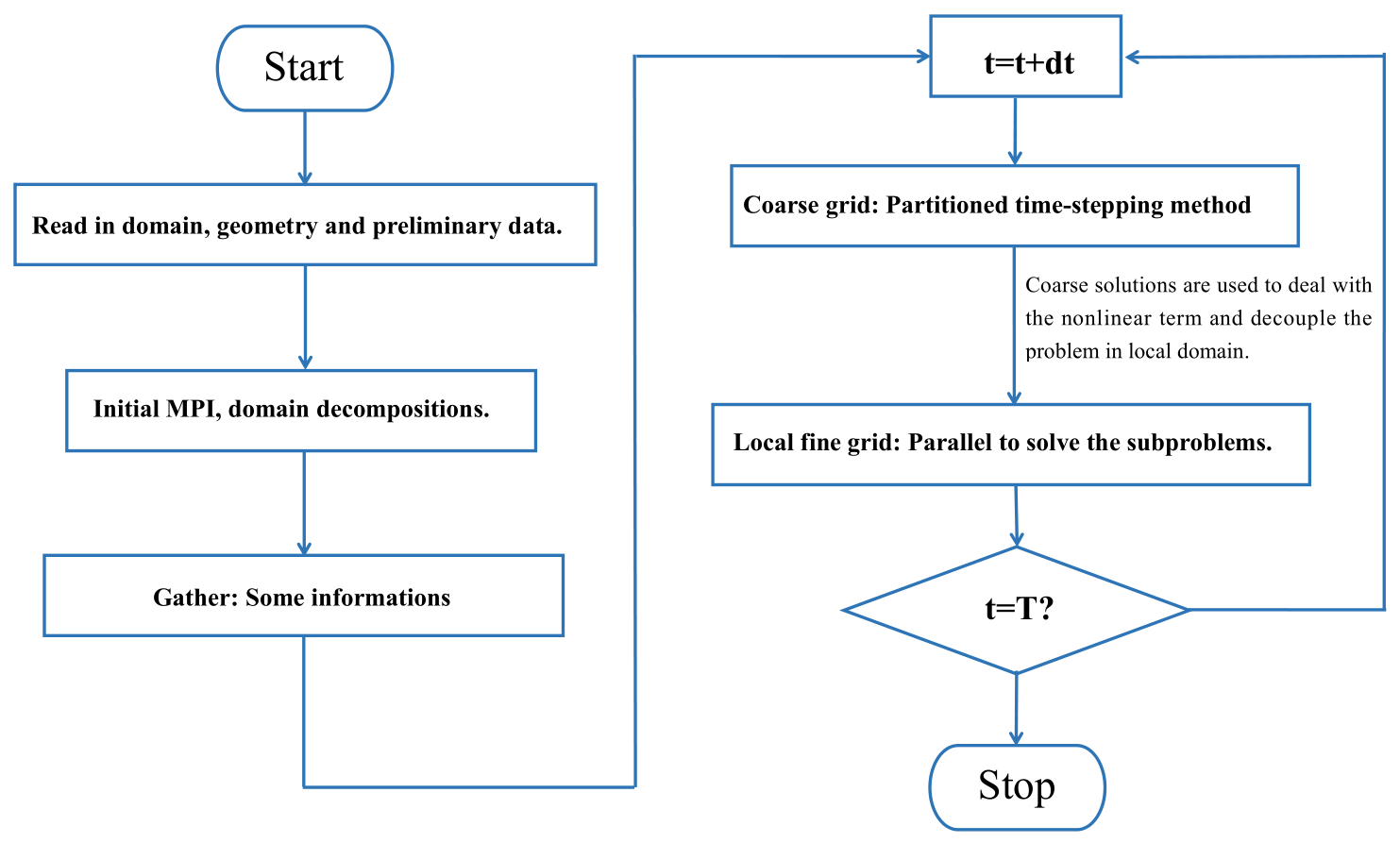}\\
  \caption{\label{Pic3}\small{Flowchart of local and parallel algorithm.}}
\end{figure}


To establish convergence results for the local and parallel finite element algorithm, we initially present local priori error estimates.

\subsection{Local a priori error estimate}
In this section, we initially introduce the following Lemma \ref{LemmaXu} which plays a crucial role in proof of local priori error estimate.
According to the literature \cite{li2021local, ding2021local}, Lemma \ref{LemmaXu} and using the property in \eqref{Superapproximation}, we have local a priori error estimate in Lemma
\ref{prioriLem}.

\begin{lemma}[\cite{xu2000local}]\label{LemmaXu}
Let $\Omega_{c0} \subset \Omega_c$ and $\omega \in C_0^{\infty}(\Omega_c)$ such that
$\mathrm{supp}~\omega \subset \subset \Omega_{c0}$. Then
\begin{equation*}
\|\omega \vec{w}\|_1^2 \leq c a_{c\eta}(\vec{w},\omega^2 \vec{w})
+ c \|\vec{w}\|_{0,\Omega_{c0}}^2,~~~~\forall \vec{w} \in X_c .
\end{equation*}
\end{lemma}

\begin{lemma}\label{prioriLem}
Suppose that $D_c \subset \subset \Omega_{c0} \subset \subset \Omega_c$ and given
$\vec{w}_h^0=\vec{w}_h(0)$.
If $(\vec{w}_h^{n+1},r_h^{n+1}) \in X_c^h(\Omega_c) \times Q^h(\Omega_c)(n=0,1,...,N-1)$, for all $(\vec{v},q) \in X_{c0}^h(\Omega_{c0}) \times Q_0^h(\Omega_{c0})$ satisfies
\begin{align*}
&\eta(\frac{\vec{w}_h^{n+1} - \vec{w}_h^n}{\Delta t}, \vec{v})
+a_{c\eta}(\vec{w}_h^{n+1},\vec{v})
+b_{\eta}(\vec{v},r_h^{n+1})
-b_{\eta}(\vec{w}_h^{n+1},q) \nonumber\\
&+b_{N \eta}(\vec{w}_h^{n+1},\vec{u}_{\mu},\vec{v})
+b_{N \eta}(\vec{u}_{\mu}, \vec{w}_h^{n+1},\vec{v}) \nonumber\\
&=\eta(\vec{f}_c^{n+1},\vec{v})-b_{N \eta}(\vec{u}_{\mu}, \vec{u}_{\mu}, \vec{v})
-\frac{\eta}{\rho}\langle p_{FH}^n ,\vec{v} \cdot \hat{n}_c \rangle_{\Gamma \Omega_{c0} }
-\frac{\eta \nu \alpha \sqrt{k_{F}}}{\tilde{\mu}}
\langle \nabla_{\tau} p_{FH}^n, P_{\tau} \vec{v}
\rangle_{\Gamma \Omega_{c0}}, \label{PrioriLemEqu}
\end{align*}
where $\mu=h, H$. Then the following local error estimate holds:
\begin{equation}
\|\vec{w}_h^{n+1}\|_{1,D_c} \leq c \Big(
(1+\Delta t^{-1/2})\|\vec{w}_h^{n+1}\|_{0,\Omega_{c0}}
+\|f\|_{L^2(\Gamma \Omega_{c0})}
+\Delta t^{-1/2} \|\vec{w}_h^{n}\|_{0,\Omega_{c0}}
\Big).\label{PrioriResult}
\end{equation}
\end{lemma}

\begin{theorem}
Suppose that $(p_{Fh}^{n+1},p_{fh}^{n+1},p_{mh}^{n+1},\vec{u}_{ch}^{n+1})$
and $(p_{n+1}^{Fh},p_{n+1}^{fh},p_{n+1}^{mh},\vec{u}_{n+1}^{ch})$ are obtained
from \eqref{fullypF}-\eqref{fullyuc} and \eqref{LocalFully_pF}-\eqref{LocalFully_uc},
respectively. For $0 \leq n \leq N-1$, the following inequalities hold:
\begin{align*}
&| \vec{u}_{ch}^{n+1}-\vec{u}_{n+1}^{ch} |_{1,D_c} \leq c (1+\Delta t^{-1/2})(\Delta t + H^{r+1}),\\
&| p_{ih}^{n+1}-p_{n+1}^{ih} |_{1,D_p} \leq c(1+\Delta t^{-1/2})(\Delta t + H^{r+1}),~~~~(i=F,f,m),\\
&\|p_{ih}^{n+1}-p_{n+1}^{ih}\|_{0,D_p} \leq c(\Delta t + H^{r+1}),~~~~(i=f,m).
\end{align*}
Furthermore,
\begin{equation}\label{FinialResult}
\begin{split}
&| \vec{u}_{c}(t_{n+1})-\vec{u}_{n+1}^{ch} |_{1,D_c} \leq c (1+\Delta t^{-1/2})(\Delta t +h^r + H^{r+1}),\\
&| p_i(t_{n+1})-p_{n+1}^{ih} |_{1,D_p} \leq c(1+\Delta t^{-1/2})(\Delta t + h^r+H^{r+1}),~~~~(i=F,f,m),\\
&\|p_i(t_{n+1})-p_{n+1}^{ih}\|_{0,D_p} \leq c(\Delta t + h^{r+1}),~~~~(i=f,m).
\end{split}
\end{equation}
\end{theorem}

\begin{proof}
Subtracting \eqref{LocalFully_uc} from \eqref{fullyuc}, we obtain
\begin{align}
&\eta \Big( \frac{\vec{u}_{ch}^{n+1} - \vec{u}_{n+1}^{ch}-(\vec{u}_{ch}^n-\vec{u}_n^{ch})}{\Delta t},\vec{v}_c \Big)
+a_{c\eta}(\vec{u}_{ch}^{n+1}-\vec{u}_{n+1}^{ch},\vec{v}_c)
+b_{\eta}(\vec{v}_c,p_h^{n+1}-p_{n+1}^h) \nonumber\\
&-b_{\eta}(\vec{u}_{ch}^{n+1}-\vec{u}_{n+1}^{ch},q)
+b_{N\eta}(\vec{u}_{ch}^{n+1}-\vec{u}_{n+1}^{ch},\vec{u}_{cH}^{n+1},\vec{v}_c)
+b_{N\eta}(\vec{u}_{cH}^{n+1},\vec{u}_{ch}^{n+1}-\vec{u}_{cH}^{n+1},\vec{v}_c) \nonumber\\
&+b_{N\eta}(\vec{u}_{ch}^{n+1}-\vec{u}_{cH}^{n+1},\vec{u}_{ch}^{n+1}-\vec{u}_{cH}^{n+1},\vec{v}_c) \nonumber\\
&= -\frac{\eta}{\rho} \langle p_{Fh}^n-p_{FH}^n,\vec{v}_c \cdot \vec{n}_c \rangle_{\Gamma \Omega_{c0}}
-\frac{\eta \nu \alpha \sqrt{k_F}}{\tilde{\mu}} \langle \nabla_{\tau}(p_{Fh}^n-p_{FH}^n),P_{\tau} \vec{v}_c \rangle_{\Gamma \Omega_{c0}}. \label{errorCha}
\end{align}
Using the Lemma \ref{prioriLem}, Lemma \ref{echConvergnece}, Lemma \ref{ech-ect} in \eqref{errorCha}, we deduce that
\begin{align*}
&|\vec{u}_{ch}^{n+1} - \vec{u}_{n+1}^{ch}|_{1,D_c}\\
&\leq c(1+\Delta t^{-1/2})(\|\vec{u}_{ch}^{n+1}-\vec{u}_{cH}^{n+1}\|_{0,\Omega_{c0}}
+\|e_{n+1}^{ch}-e^{ch}(t_{n+1})\|_{0,\Omega_{c0}})
+\|e^{ch}(t_{n+1})\|_{0,\Omega_{c0}}) \\
&~~~+c\Delta t^{-1/2}(\|\vec{u}_{ch}^n-\vec{u}_{cH}^n\|_{0,\Omega_{c0}} + \|e_n^{ch}-e^{ch}(t_n)\|_{0,\Omega_{c0}}
+\|e^{ch}(t_n)\|_{0,\Omega_{c0}} )
+c(\Delta t^{1/2} + H^{r+1})\\
&\leq c(1+\Delta t^{-1/2})(\Delta t + H^{r+1}).
\end{align*}
Similarly, applying Lemma \ref{tripleeih} and Lemma \ref{tripleEkih},
we can get
\begin{align*}
|p_{ih}^{n+1}-p_{n+1}^{ih}|_{1,D_p} \leq c(1+\Delta t^{-1/2})(\Delta t + H^{r+1}),~~~~(i=F,f,m).
\end{align*}
In addition, we have
\begin{align*}
\|p_{ih}^{n+1}-p_{n+1}^{ih}\|_{0,D_p}
&\leq \|p_{ih}^{n+1}-p_{iH}^{n+1}\|_{0,\Omega_{p0}}
+ \|e_{n+1}^{ih} - e^{ih}(t_{n+1})\|_{0,\Omega_{p0}}
+ \|e^{ih}(t_{n+1})\|_{0,\Omega_{p0}}\\
&\leq c (\Delta t + H^{r+1}).
\end{align*}
Using triangle inequality, \eqref{FinialResult} can be achieved.

\end{proof}

\subsection{Convergence results}\label{FinalResult}
Defining the piecewise norms
\begin{align*}
||| \vec{u}_{ch}^{n+1} - \vec{u}_{n+1}^{ch} |||_{1,\Omega_c}
&=\Big(  \sum_{j'=1}^{M'} | \vec{u}_{ch,j'}^{n+1}-\vec{u}_{n+1}^{ch,j'}  |_{1,D_{cj'}}^2  \Big)^{1/2},\\
||| p_{ih}^{n+1} - p_{n+1}^{ih} |||_{1,\Omega_p}
&=\Big(  \sum_{j=1}^M | p_{ih,j}^{n+1}-p_{n+1}^{ih,j}  |_{1,D_{pj}}^2  \Big)^{1/2},~~~~(i=F,f,m),\\
||| p_{ih}^{n+1} - p_{n+1}^{ih} |||_{0,\Omega_p}
&=\Big(  \sum_{j=1}^M \| p_{ih,j}^{n+1}-p_{n+1}^{ih,j} \|_{0,D_{pj}}^2  \Big)^{1/2},~~~~(i=f,m),
\end{align*}
we have the following convergence results.

\begin{theorem}\label{TheoremFinal}
Assume that $(p_{Fh}^{n+1},p_{fh}^{n+1},p_{mh}^{n+1},\vec{u}_{ch}^{n+1})$
and $(p_{n+1}^{Fh},p_{n+1}^{fh},p_{n+1}^{mh},\vec{u}_{n+1}^{ch})$ are obtained
from Algorithm \ref{Algorithm-1} and Algorithm \ref{algorithm3}, respectively.
For $0 \leq n \leq N-1$, the following inequalities hold:
\begin{align*}
||| \vec{u}_{ch}^{n+1} - \vec{u}_{n+1}^{ch} |||_{1,\Omega_c} &\leq c(1+\Delta t^{-1/2})(\Delta t +h^r+ H^{r+1}),\\
||| p_{ih}^{n+1} - p_{n+1}^{ih} |||_{1,\Omega_p} &\leq c(1+\Delta t^{-1/2})(\Delta t +h^r+ H^{r+1}),~~(i=F,f,m),\\
||| p_{ih}^{n+1} - p_{n+1}^{ih} |||_{0,\Omega_p} &\leq c(\Delta t +H^{r+1}),~~(i=f,m).
\end{align*}
Furthermore,
\begin{align*}
||| \vec{u}_{c}(t_{n+1}) - \vec{u}_{n+1}^{ch} |||_{1,\Omega_c} &\leq c(1+\Delta t^{-1/2})(\Delta t + h^r + H^{r+1}),\\
||| p_{i}(t_{n+1}) - p_{n+1}^{ih} |||_{1,\Omega_p} &\leq c(1+\Delta t^{-1/2})(\Delta t +h^r+ H^{r+1}),~~(i=F,f,m),\\
||| p_{i}(t_{n+1}) - p_{n+1}^{ih} |||_{0,\Omega_p} &\leq c(\Delta t +h^{r+1}),~~(i=f,m).
\end{align*}
\end{theorem}

\begin{remark}
The conclusions in this theorem
are the same as the results of parallel methods for a simpler Navier-Stokes-Darcy model or even Navier-Stokes model \cite{li2022local, ding2021local, li2021local}.
\end{remark}

\section{Numerical results}
In this section, the first two numerical examples are presented to validate the accuracy and efficiency of the proposed algorithm. The last one is provided to illustrate the features of the application to flow problems around multistage fractured horizontal wellbore completions with super-hydrophobic proppant.
The well-known MINI elements (P1b-P1) are employed in the conduit region, while P1-elements are used in the triple-porosity region to evaluate the algorithm's convergence rate, as previously conducted.
All simulations reported in this work are
carried out on a same cluster,
and the message-passing is supported by MPI of FreeFEM++ package \cite{Hechet2010}.

\subsection{Example 1: Experimental rate of convergence in 2D}
Let the computational domain $\Omega$ be composed of $\Omega_p = (0,1) \times (0,1)$ and $\Omega_c = (0,1)\times (1,2)$ with the interface $\Gamma = (0,1) \times \{1\}$. The analytical solutions satisfying the transient triple-porosity Navier-Stokes model are given by
\begin{equation*}
\begin{split}
&p_m=(2-\pi \sin(\pi x)) \sin(0.5 \pi (3y^3-2y^2)) \cos(t),\\
&p_f=(2-\pi \sin(\pi x)) \cos(\pi (1-y)) \cos(t),\\
&p_F=(2-\pi \sin(\pi x)) (1-y-\cos(\pi y)) \cos(t),\\
&\vec{u}_c=\Big[ \big(x^2(y-1)^2+y \big)\cos(t), ~~-\frac{2}{3}x (y-1)^3 \cos(t)+  \big( 2-\pi \sin(\pi x) \big) \cos(t) \Big]^{T},~~~~~~~~~~~~\\
&p=(2-\pi \sin(\pi x)) \sin(0.5\pi y) \cos(t).
\end{split}
\end{equation*}
In addition, the initial conditions, boundary conditions and forcing terms can be derived from the analytical solutions. For simplicity of calculation, all the parameters $\phi_i, C_i, k_i(i=F,f,m), \sigma, \sigma^{\ast}, \tilde{\mu}, \rho, \eta, \nu, \alpha$ and $T$ are supposed to be 1. To test the proposed local and parallel algorithm, $\Omega_c$ and $\Omega_p$ are partitioned into $2 \times 2$
subdomains respectively as follows:
\begin{align*}
&D_{c1}=[0,\frac{1}{2}] \times [1,\frac{3}{2}],~D_{c2}=[\frac{1}{2},1] \times [1,\frac{3}{2}],
~D_{c3}=[\frac{1}{2},1] \times [\frac{3}{2},2],~D_{c4}=[0,\frac{1}{2}] \times [\frac{3}{2},2],\\
&D_{p1}=[0,\frac{1}{2}] \times [0,\frac{1}{2}],~D_{p2}=[\frac{1}{2},1] \times [0,\frac{1}{2}],
~D_{p3}=[\frac{1}{2},1] \times [\frac{1}{2},1],~D_{p4}=[0,\frac{1}{2}] \times [\frac{1}{2},1],
\end{align*}
in which each subdomain computed by one MPI process (see Figure \ref{Fig2DDarcyNS}).
Then extend each $D_{cj'}$ and $D_{pj} (j',j=1,2,3,4)$ to $\Omega_{cj'}$ and $\Omega_{pj}$ as follows:
\begin{align*}
&\Omega_{c1}=[0,\frac{3}{4}] \times [1,\frac{7}{4}],~\Omega_{c2}=[\frac{1}{4},1] \times [1,\frac{7}{4}],
~\Omega_{c3}=[\frac{1}{4},1] \times [\frac{5}{4},2],~\Omega_{c4}=[0,\frac{3}{4}] \times [\frac{5}{4},2],\\
&\Omega_{p1}=[0,\frac{3}{4}] \times [0,\frac{3}{4}],~\Omega_{p2}=[\frac{1}{4},1] \times [0,\frac{3}{4}],
~\Omega_{p3}=[\frac{1}{4},1] \times [\frac{1}{4},1],~\Omega_{p4}=[0,\frac{3}{4}] \times [\frac{1}{4},1].
\end{align*}

\begin{figure}[H]
\begin{centering}
\begin{subfigure}[t]{0.31\textwidth}
\centering
\includegraphics[width=1.5\textwidth]{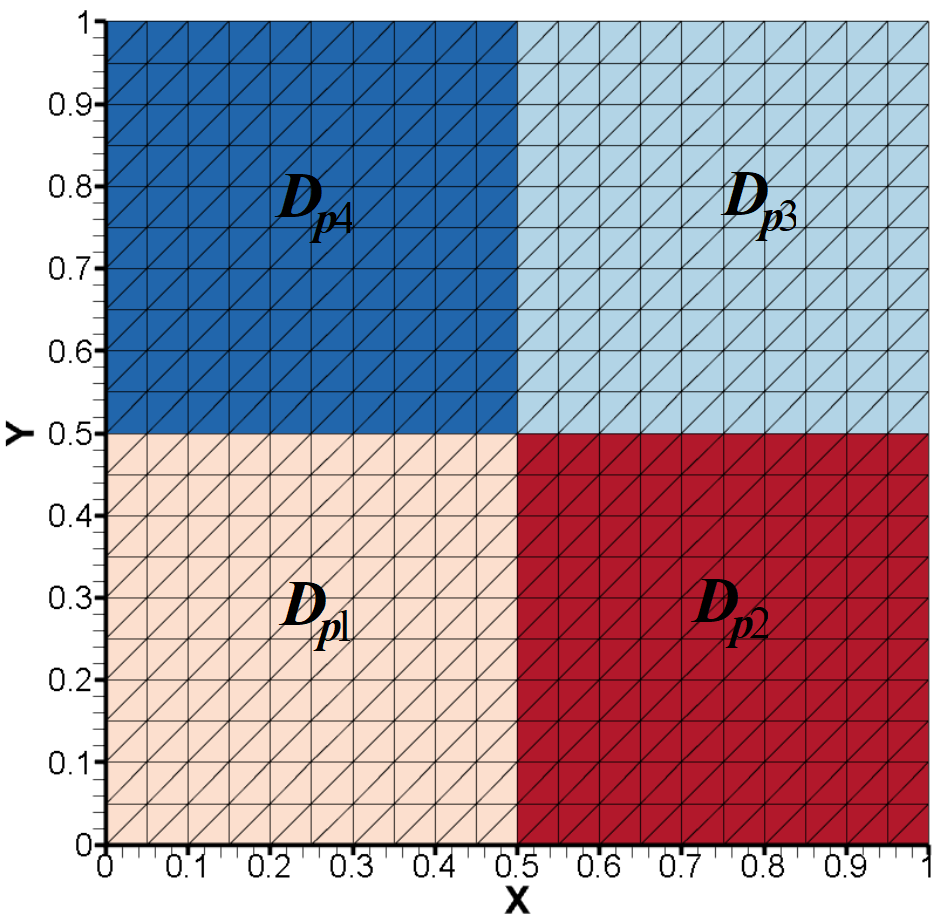}
\end{subfigure}
\hspace{30mm}
\begin{subfigure}[t]{0.31\textwidth}
\centering
\includegraphics[width=1.5\textwidth]{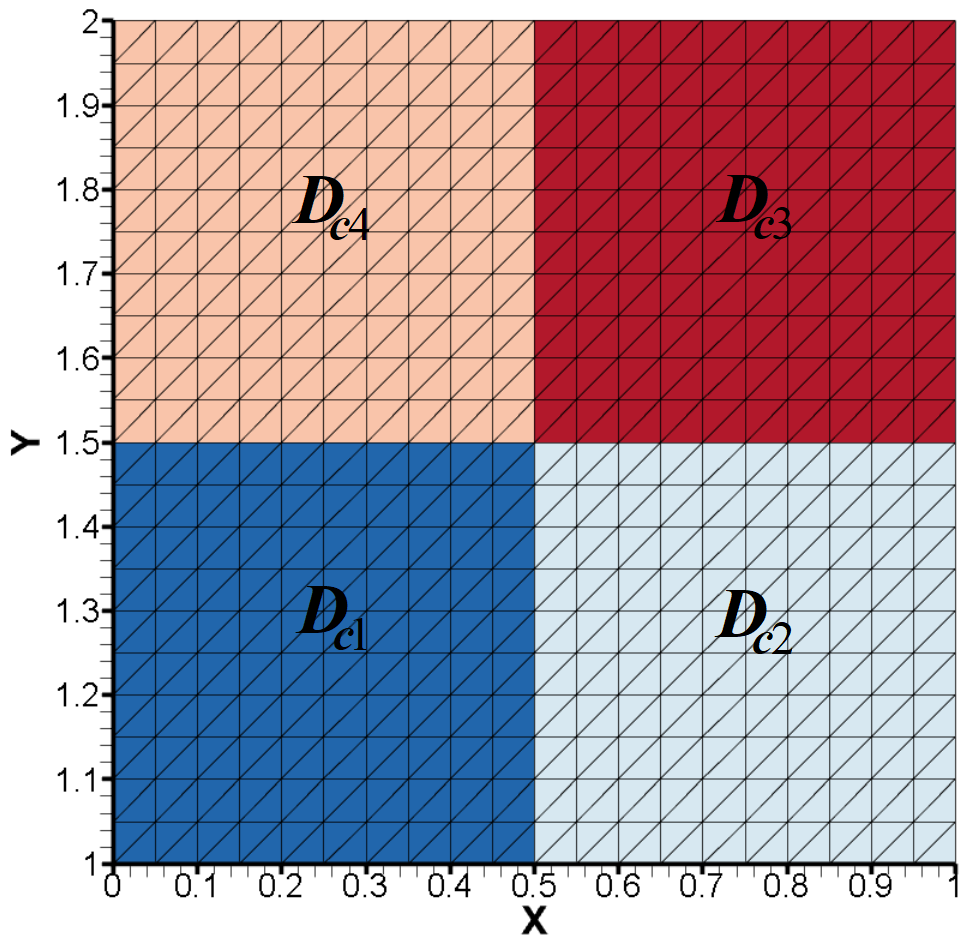}
\end{subfigure}
\end{centering}
\caption{\small{\label{Fig2DDarcyNS}2D partitions of triple-porosity domain and conduit domain.}}
\end{figure}

We solve the model with time size $\Delta t=h^2$, space mesh sizes~$h=1/4, 1/16, 1/64, 1/256$
and coarse grid mesh size $H$ satisfy $h=H^2$.
The numerical results are shown in Table \ref{T1}, which are consistent with the theoretical results in Theorem \ref{TheoremFinal}.
Furthermore, we solve this model with Algorithm \ref{Algorithm-1}, which is partitioned time stepping method.
From Table \ref{T2}, we can see that our parallel algorithm saves a large amount of computational time compared with the Algorithm \ref{Algorithm-1}.


\begin{table}[H]
\caption{\label{T1}The convergence performance and computational cost of Algorithm \ref{algorithm3}(Local Parallel Algorithm) in 2D}
\centering
\resizebox{\textwidth}{!}{
\begin{tabular}{cccccccc}
\hline
$h ~\&~ H$ & $|||\vec{u}_c -\vec{u}_{n+1}^{ch}|||_1$ &
Rate &
 $|||p_F -p_{n+1}^{Fh}|||_1$&
Rate &
$|||p_f -p_{n+1}^{fh}|||_0$&
Rate \\
\hline
~$\frac{1}{4}~~~~~\frac{1}{2}$  & 0.786682    & --       & 0.933324    &--   &0.094526  &--\\
$\frac{1}{16}~~~~~\frac{1}{4}$  & 0.209930    & 0.95     & 0.244991    &0.96 &0.006536  &1.93\\
$\frac{1}{64}~~~~~\frac{1}{8}$  & 0.057940    & 0.93     & 0.061445    &1.00 &0.000384  &2.05\\
$\frac{1}{256}~~~~\frac{1}{16}$ & 0.015161    & 0.97     & 0.015351    &1.00 &0.000023  &2.03\\
\hline
$|||p_f-p_{n+1}^{fh}|||_1$ &
Rate &
$|||p_m-p_{n+1}^{mh}|||_0$ &
Rate &
$|||p_m-p_{n+1}^{mh}|||_1$ &
Rate &
CPU(s)\\
\hline
1.327330 &  --  & 0.070969  &--     &1.038320 & --   &4.18\\
0.343407 & 0.98 & 0.005808  & 1.81  &0.299781 & 0.90 &25.75\\
0.086055 & 1.00 & 0.000368  & 1.99  &0.075829 & 0.99 &320.68\\
0.021517 & 1.00 & 0.000021  & 2.06  &0.018971 & 1.00 &7972.32\\
\hline
\end{tabular}
}
\end{table}

\begin{table}[H]
\caption{\label{T2}The convergence performance and computational cost of Algorithm \ref{Algorithm-1}(Traditional Algorithm) in 2D}
\centering
\resizebox{\textwidth}{!}{
\footnotesize{
\begin{tabular}{cccccccc}
\hline
$h $ & $|\vec{u}_c -\vec{u}_{ch}^{n+1}|_1$ &
Rate &
$|p_F -p_{Fh}^{n+1}|_1$&
Rate &
$\|p_f -p_{fh}^{n+1}\|_0$&
Rate \\
\hline
$\frac{1}{4} $  & 0.783562    & --      & 0.922156    &--   &0.094526 & --   \\
$\frac{1}{16} $  & 0.208532    & 0.95   & 0.245630    &1.95 &0.006536 & 1.93 \\
$\frac{1}{64} $  & 0.057523    & 0.93   & 0.056354    &1.06 &0.000384 & 2.05  \\
$\frac{1}{256} $ & 0.015151    & 0.96   & 0.015264    &0.94 &0.000023 & 2.03  \\
\hline
$|p_f-p_{fh}^{n+1}|_1$ &
Rate &
$\|p_m-p_{mh}^{n+1}\|_0$ &
Rate &
$|p_m-p_{mh}^{n+1}|_1$ &
Rate &
CPU(s)\\
\hline
1.327330 &  --  & 0.070969  &--     &1.038320 & --   &4.22\\
0.343407 & 0.98 & 0.005808  & 1.81  &0.299781 & 0.90 &32.78 \\
0.086055 & 1.00 & 0.000368  & 1.99  &0.075829 & 0.99 &570.62
\\
0.021517 & 1.00 & 0.000021  & 2.06  &0.018971 & 1.00 &11958.50
\\
\hline
\end{tabular}
}
}
\end{table}

\vspace{-0.5cm}
\begin{figure}[H]
\begin{centering}
\begin{subfigure}[t]{0.31\textwidth}
\centering
\includegraphics[width=1.05\textwidth]{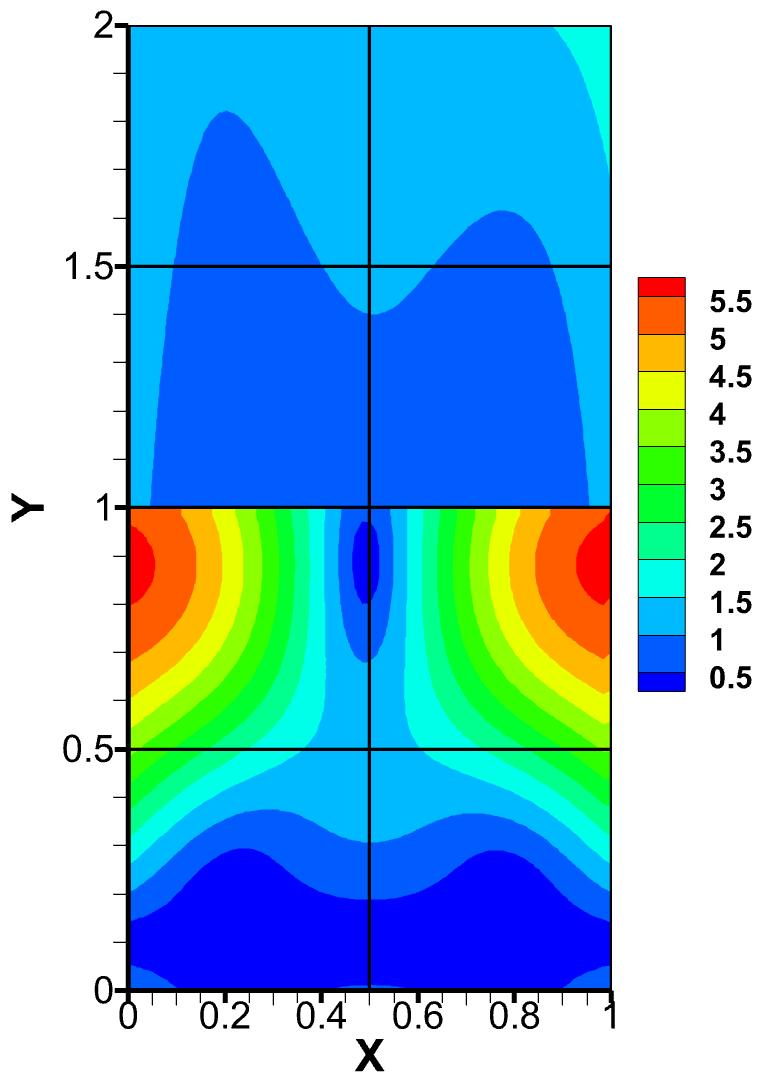}
\end{subfigure}
\quad
\begin{subfigure}[t]{0.31\textwidth}
\centering
\includegraphics[width=1.05\textwidth]{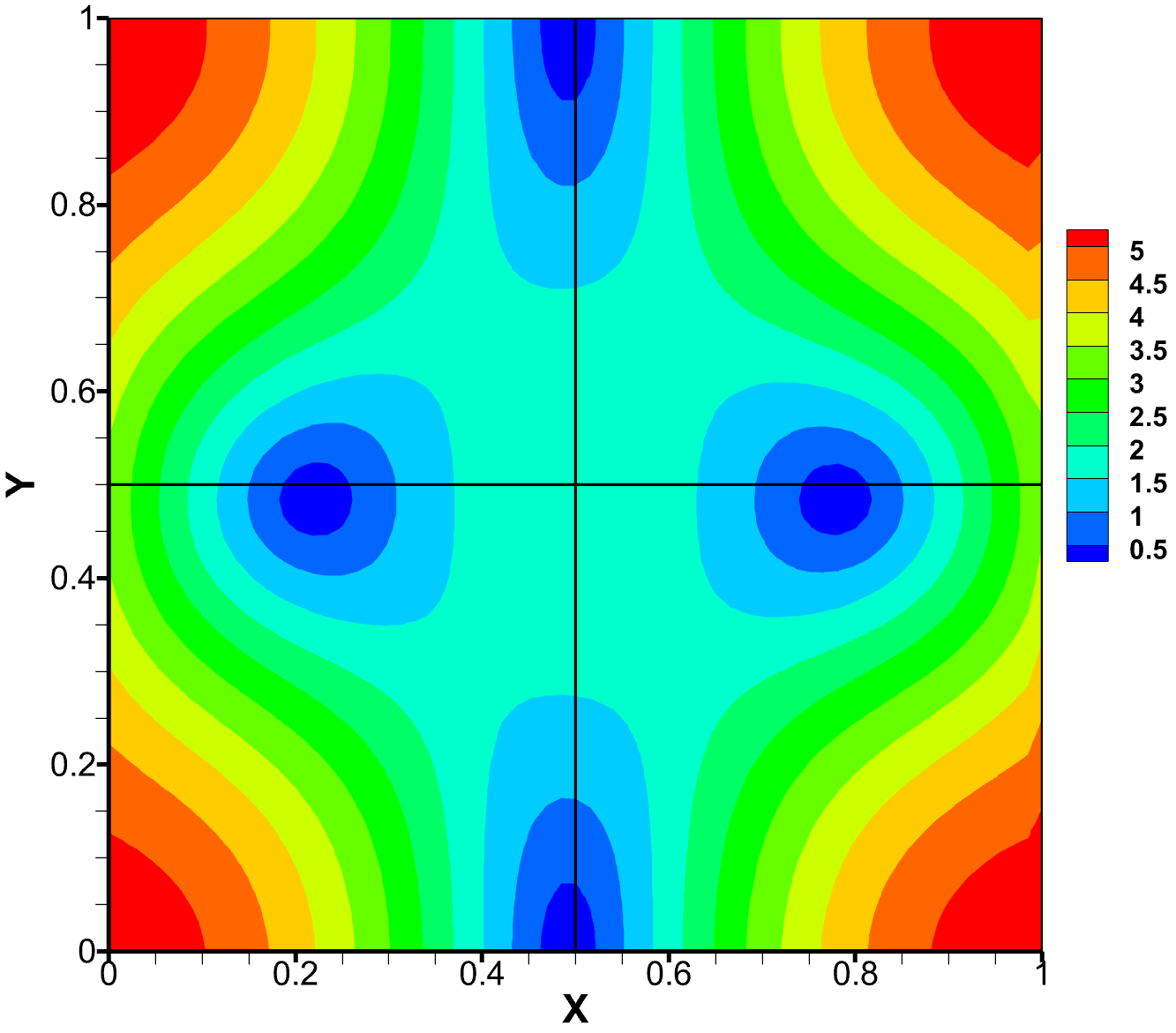}
\end{subfigure}
\quad
\begin{subfigure}[t]{0.31\textwidth}
\centering
\includegraphics[width=1.05\textwidth]{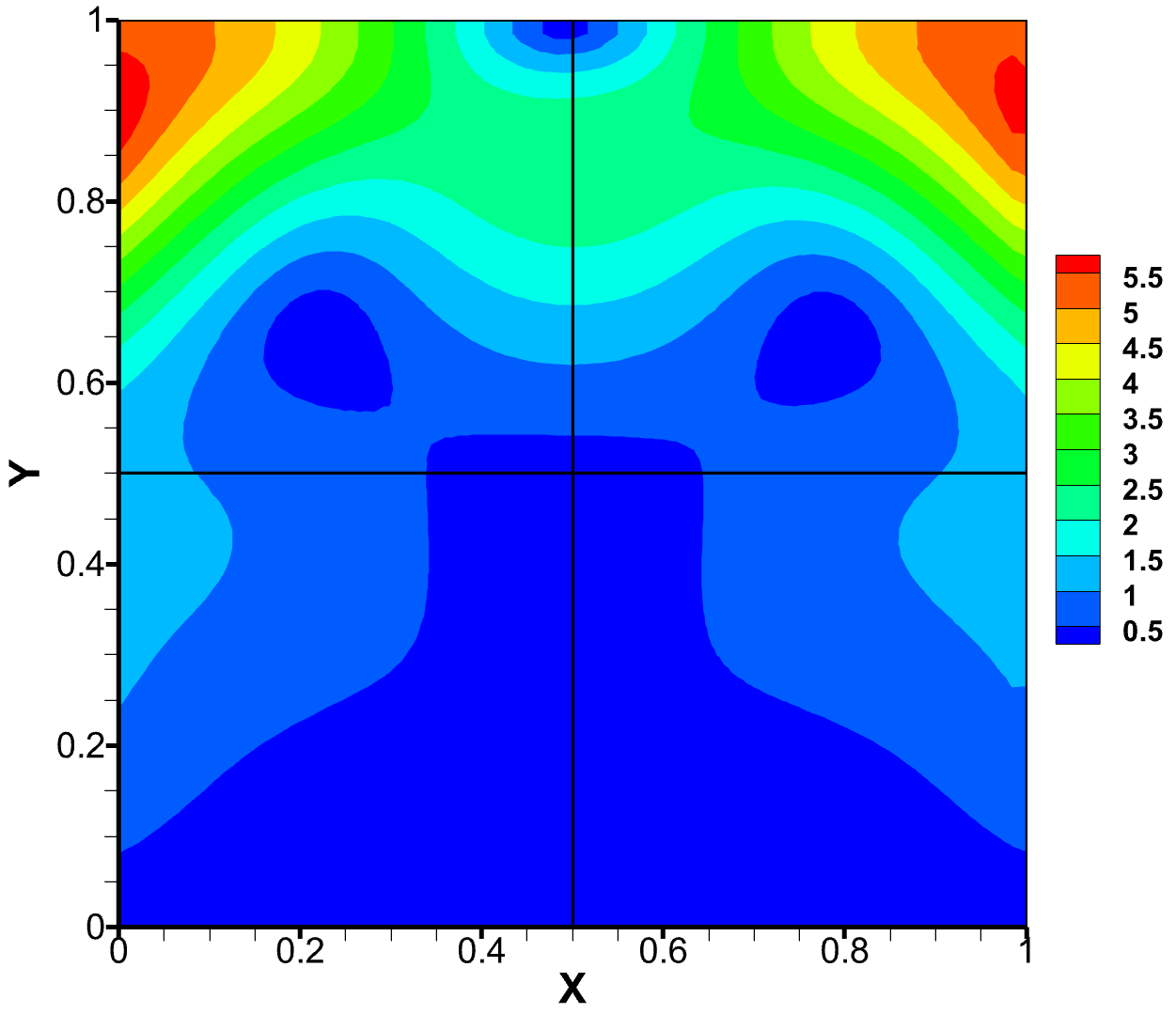}
\end{subfigure}
\end{centering}
\caption{\label{Fig22}\small{The flow speed of parallel algorithm in 2D. Left: the flow in macro-fractures and conduits; Middle: the flow in micro-fractures;
Right: the flow in stagnant-matrix.}}
\end{figure}

\begin{figure}[H]
\begin{centering}
\begin{subfigure}[t]{0.31\textwidth}
\centering
\includegraphics[width=1.05\textwidth]{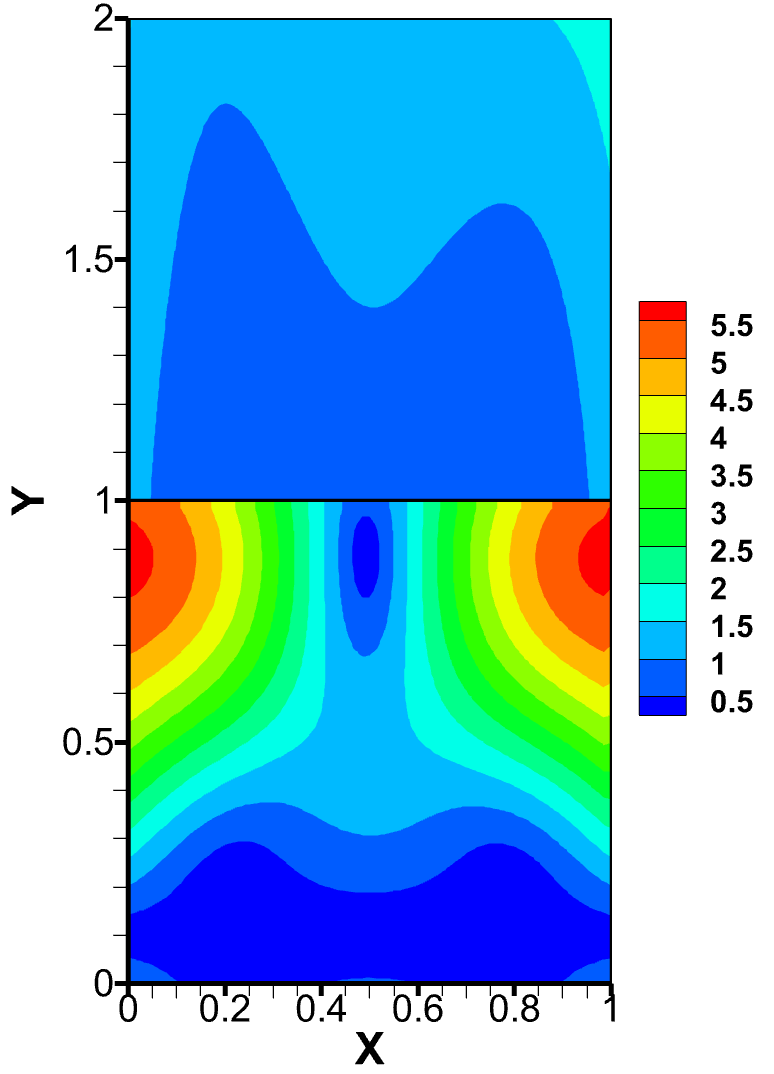}
\end{subfigure}
\quad
\begin{subfigure}[t]{0.31\textwidth}
\centering
\includegraphics[width=1.05\textwidth]{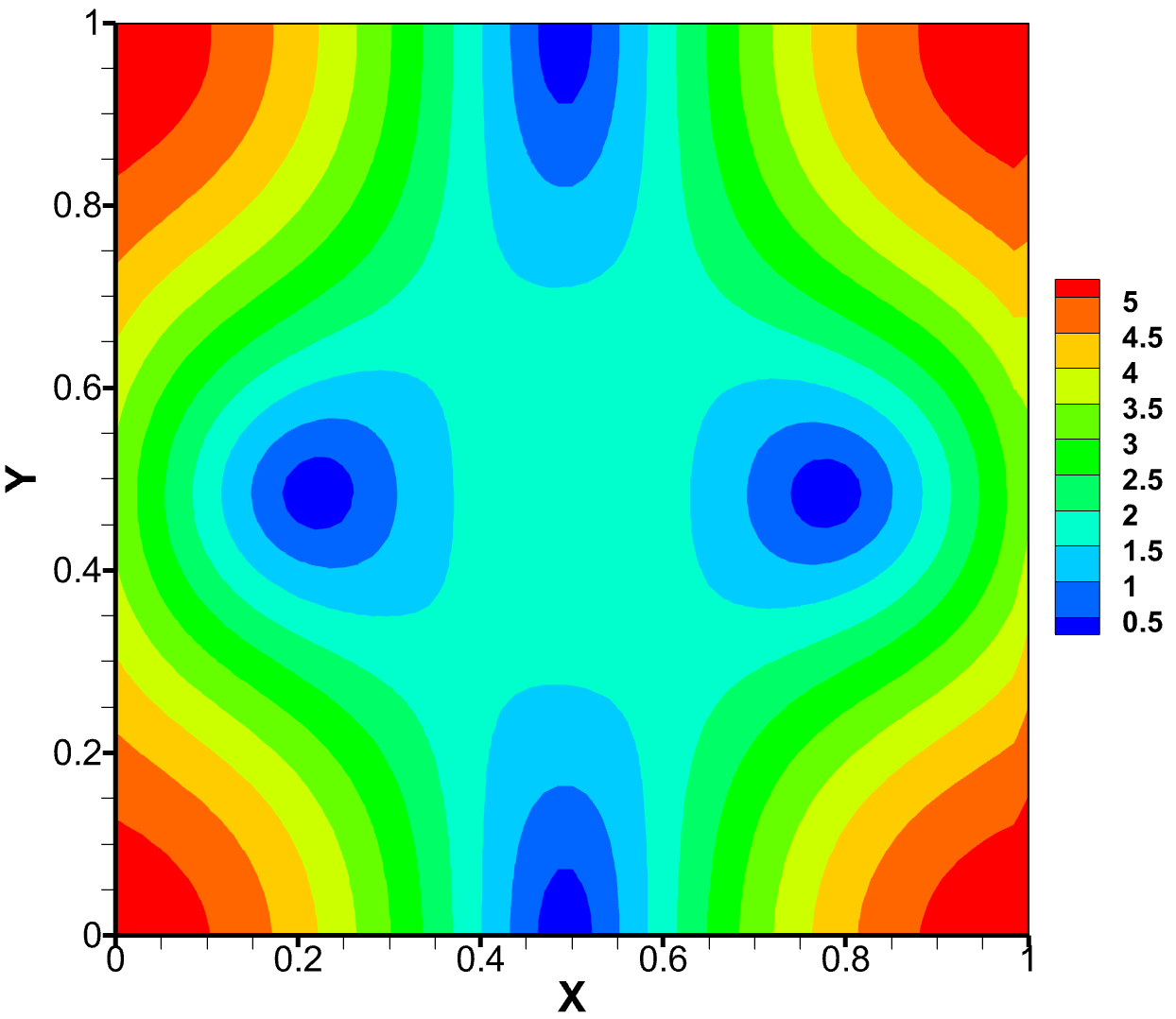}
\end{subfigure}
\quad
\begin{subfigure}[t]{0.31\textwidth}
\centering
\includegraphics[width=1.05\textwidth]{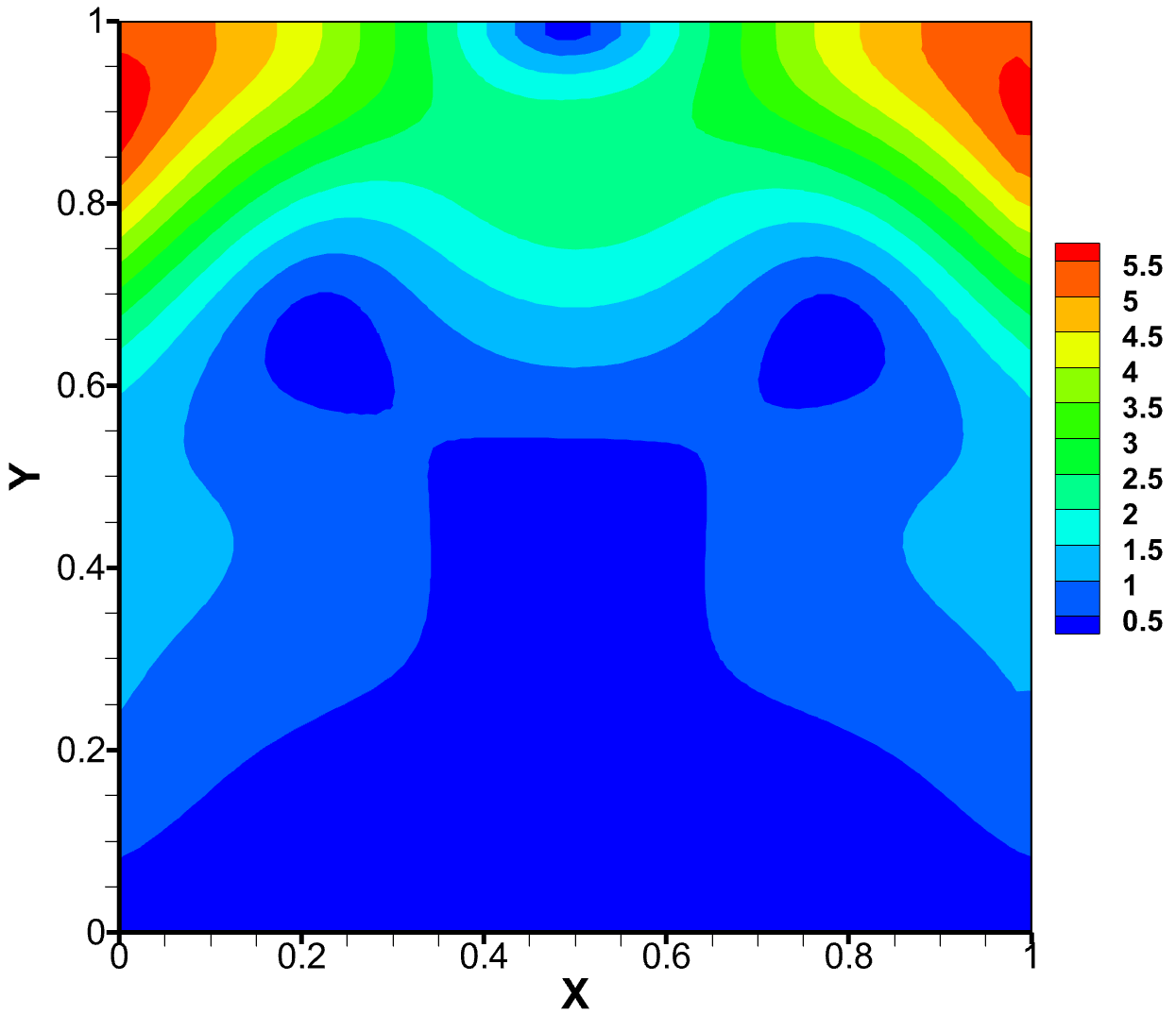}
\end{subfigure}
\end{centering}
\caption{\label{Fig23}\small{The flow speed of traditional algorithm in 2D. Left: the flow in macro-fractures and conduits; Middle: the flow in micro-fractures;
Right: the flow in stagnant-matrix.}}
\end{figure}

\begin{figure}[H]
\begin{centering}
\begin{subfigure}[t]{0.31\textwidth}
\centering
\includegraphics[width=1.03\textwidth]{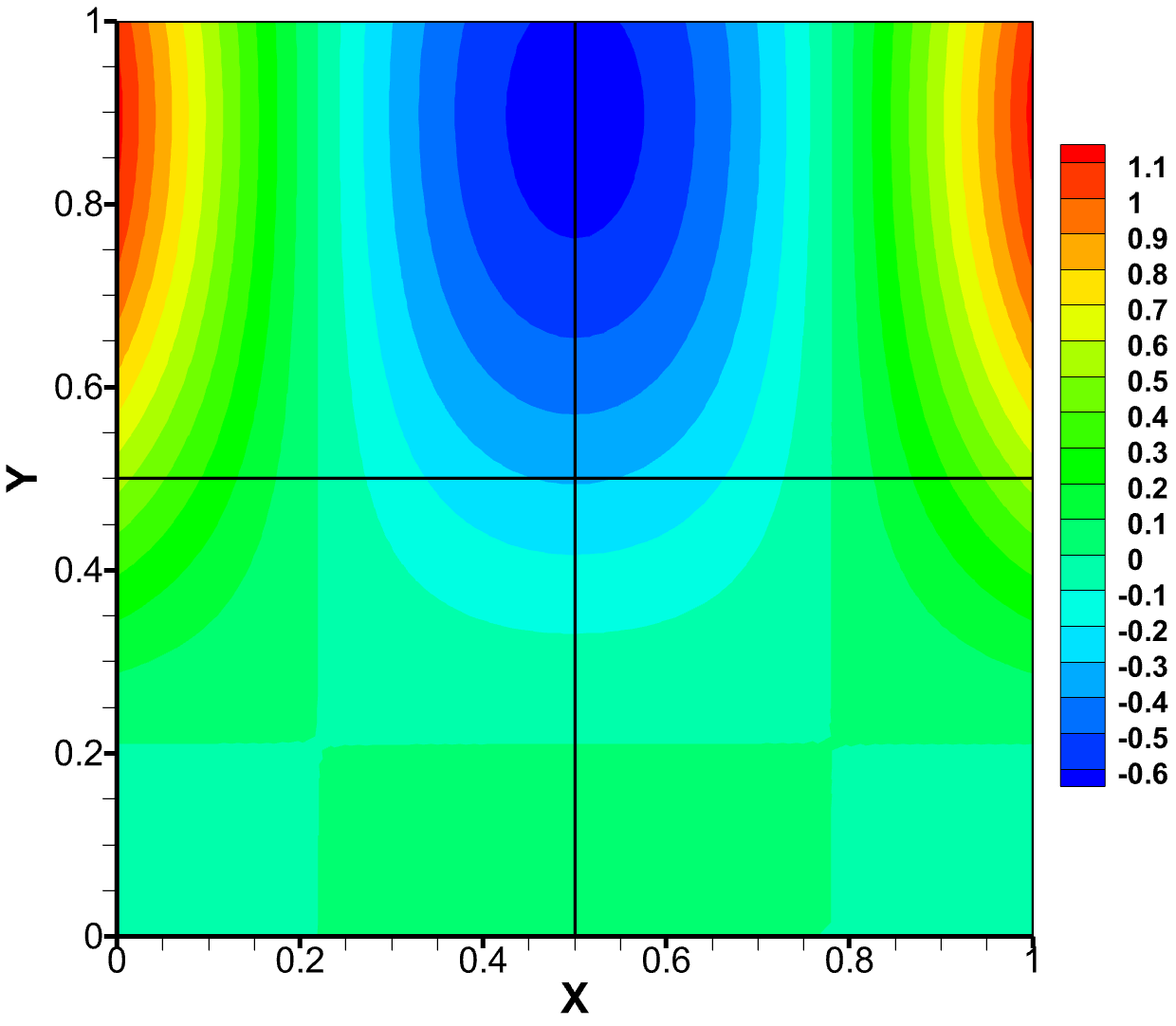}
\end{subfigure}
\quad
\begin{subfigure}[t]{0.31\textwidth}
\centering
\includegraphics[width=1.03\textwidth]{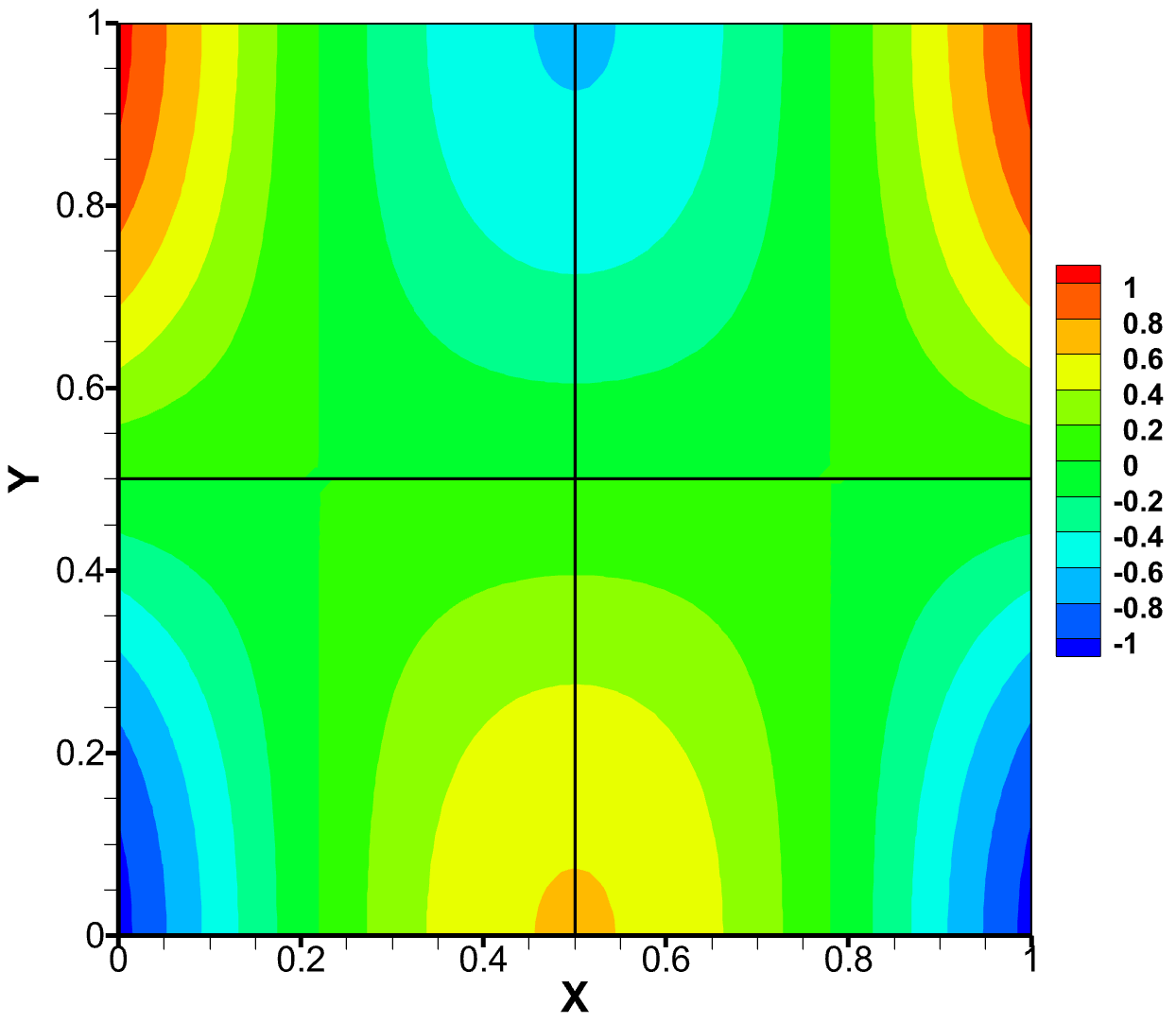}
\end{subfigure}
\quad
\begin{subfigure}[t]{0.31\textwidth}
\centering
\includegraphics[width=1.03\textwidth]{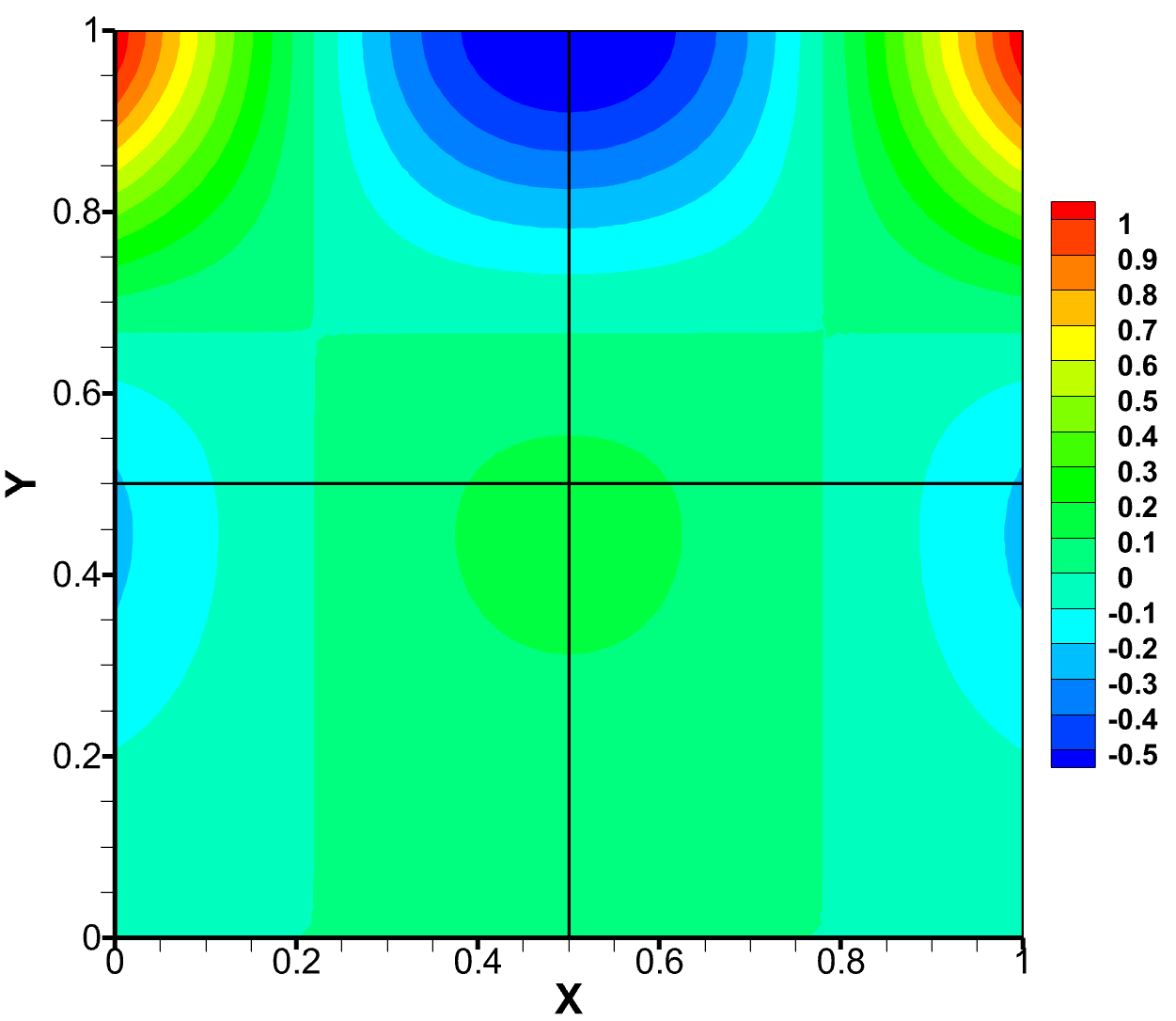}
\end{subfigure}
\end{centering}
\caption{\label{Fig24}\small{The pressure of parallel algorithm in 2D. Left: the flow in macro-fractures and conduits; Middle: the flow in micro-fractures;
Right: the flow in stagnant-matrix.}}
\end{figure}

\begin{figure}[H]
\begin{centering}
\begin{subfigure}[t]{0.31\textwidth}
\centering
\includegraphics[width=1.05\textwidth]{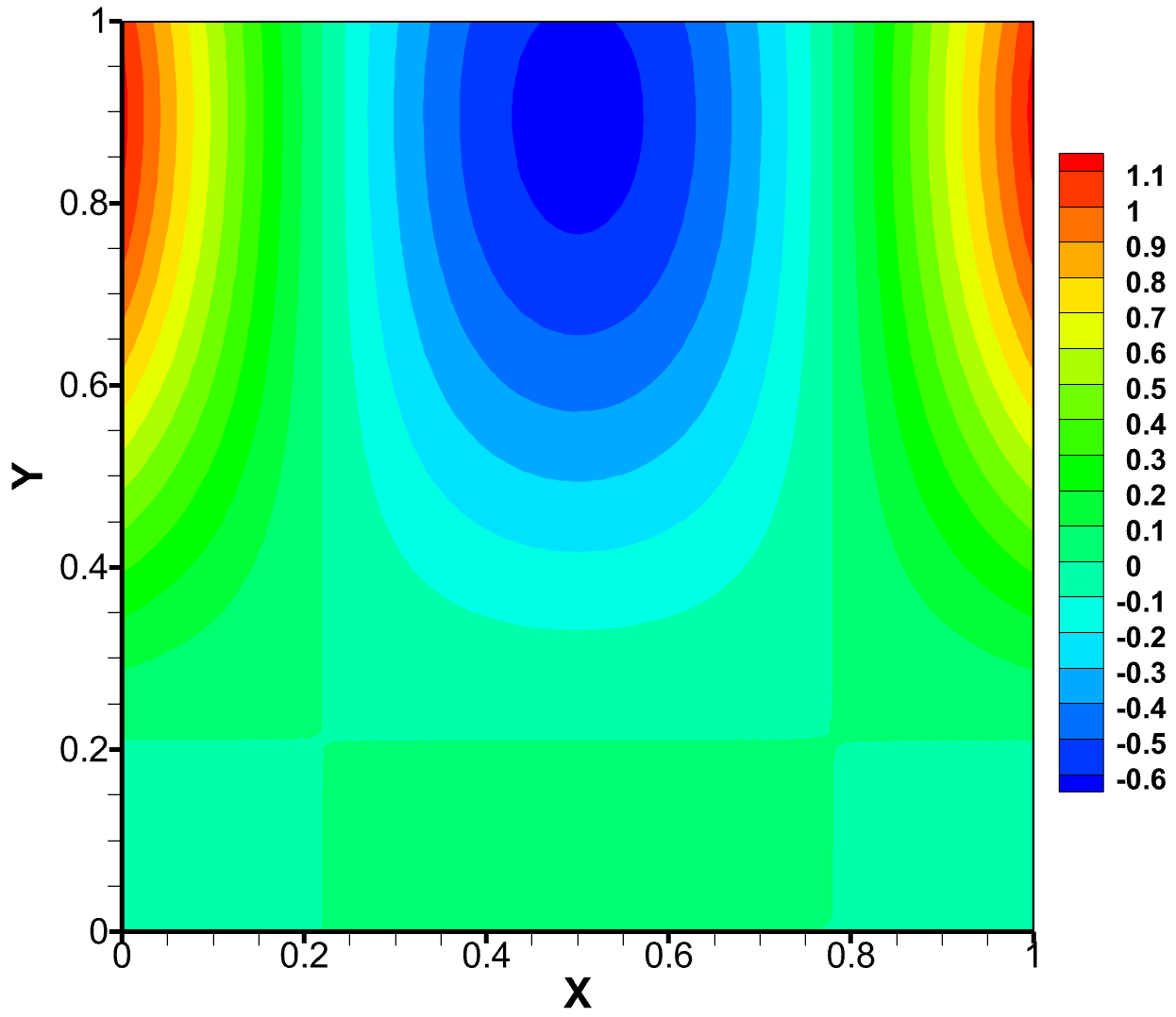}
\end{subfigure}
\quad
\begin{subfigure}[t]{0.31\textwidth}
\centering
\includegraphics[width=1.05\textwidth]{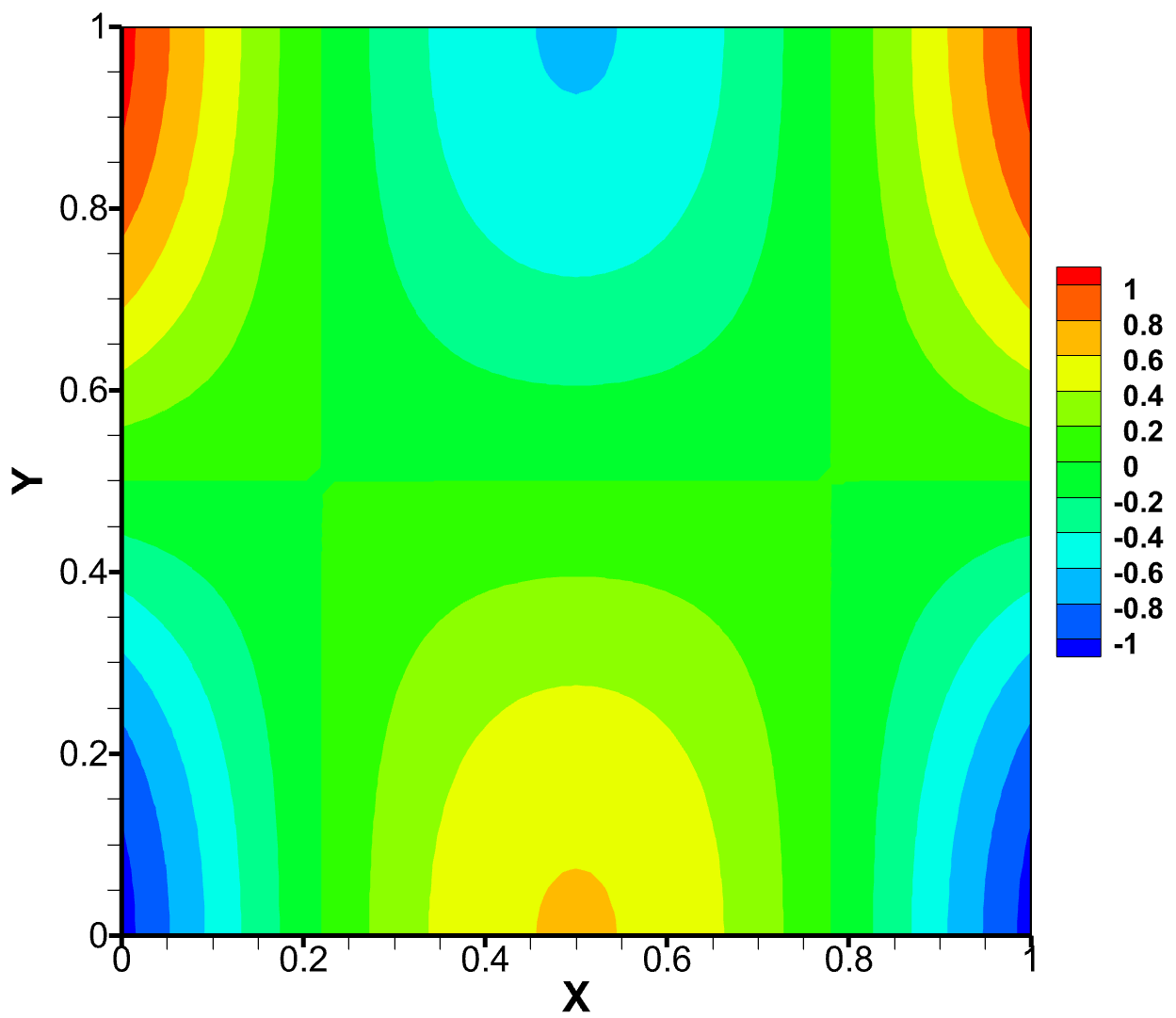}
\end{subfigure}
\quad
\begin{subfigure}[t]{0.31\textwidth}
\centering
\includegraphics[width=1.05\textwidth]{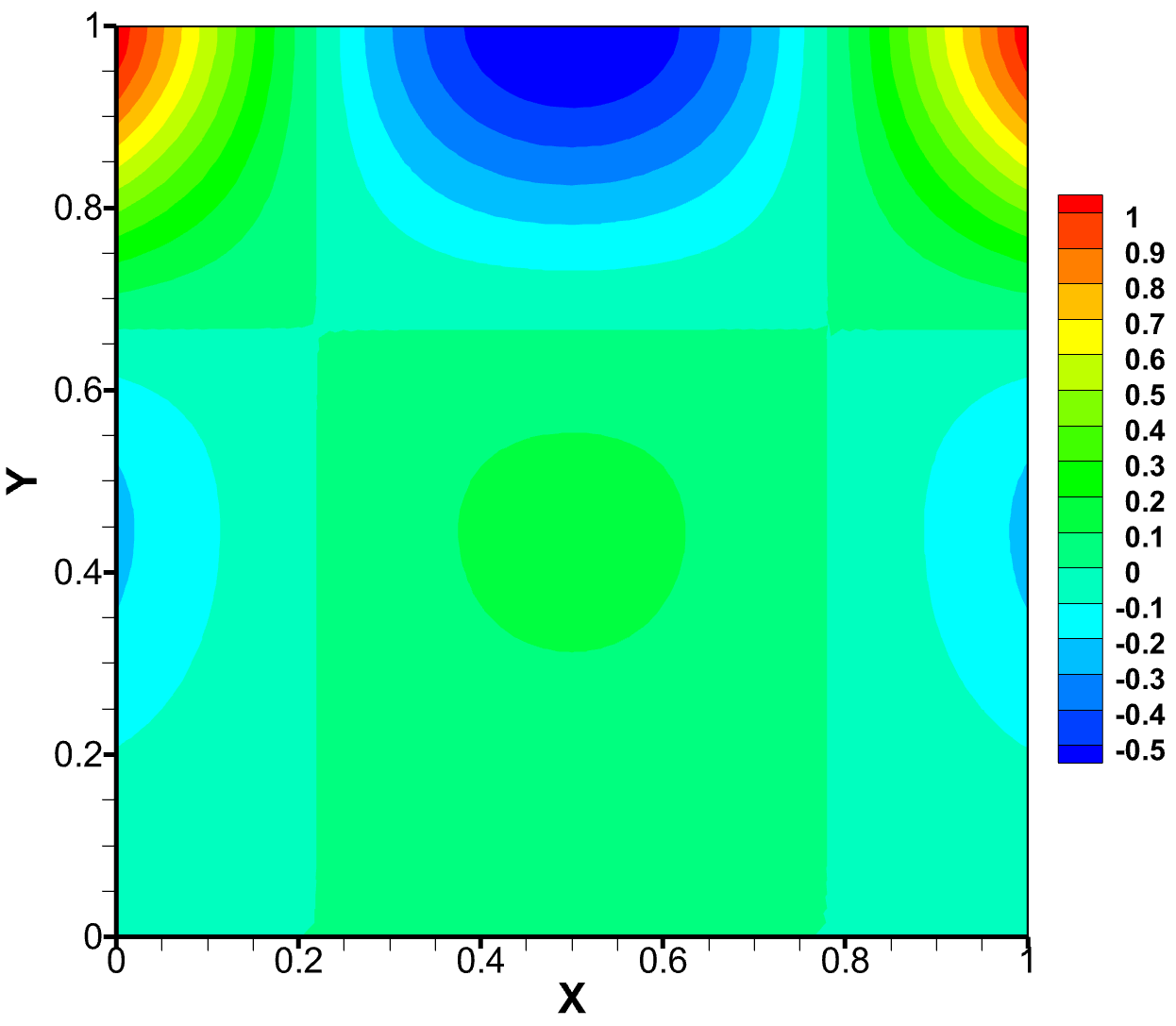}
\end{subfigure}
\end{centering}
\caption{\label{Fig25}\small{The pressure of traditional algorithm in 2D. Left: the flow in macro-fractures and conduits; Middle: the flow in micro-fractures;
Right: the flow in stagnant-matrix.}}
\end{figure}

Furthermore, the flow speed of the parallel algorithm(Algorithm \ref{algorithm3}) and traditional algorithm(Algorithm \ref{Algorithm-1}) at~$h=1/64$
are shown in Figures \ref{Fig22} and \ref{Fig23}, respectively. In addition, the pressures~$p_F, p_f$~and~$p_m$
are presented in Figures \ref{Fig24} and \ref{Fig25}.
We can see that these pictures are nearly the same.

\subsection{Example 2: Experimental rate of convergence in 3D}
Consider the model in 3D, which is shown in Figure \ref{P2}.
Let~$\Omega=(0,1) \times (0,1) \times (-0.25,0.75)$~with~$\Omega_p=\{(x,y,z)\in \Omega | z \geq 0\}$~and~$\Omega_c=\{(x,y,z)\in \Omega | z \leq 0\}$, and~$\Gamma=\{(x,y,z)\in \Omega | z = 0\}$.
\begin{figure}[htbp]
  \centering
  \includegraphics[width=10cm]{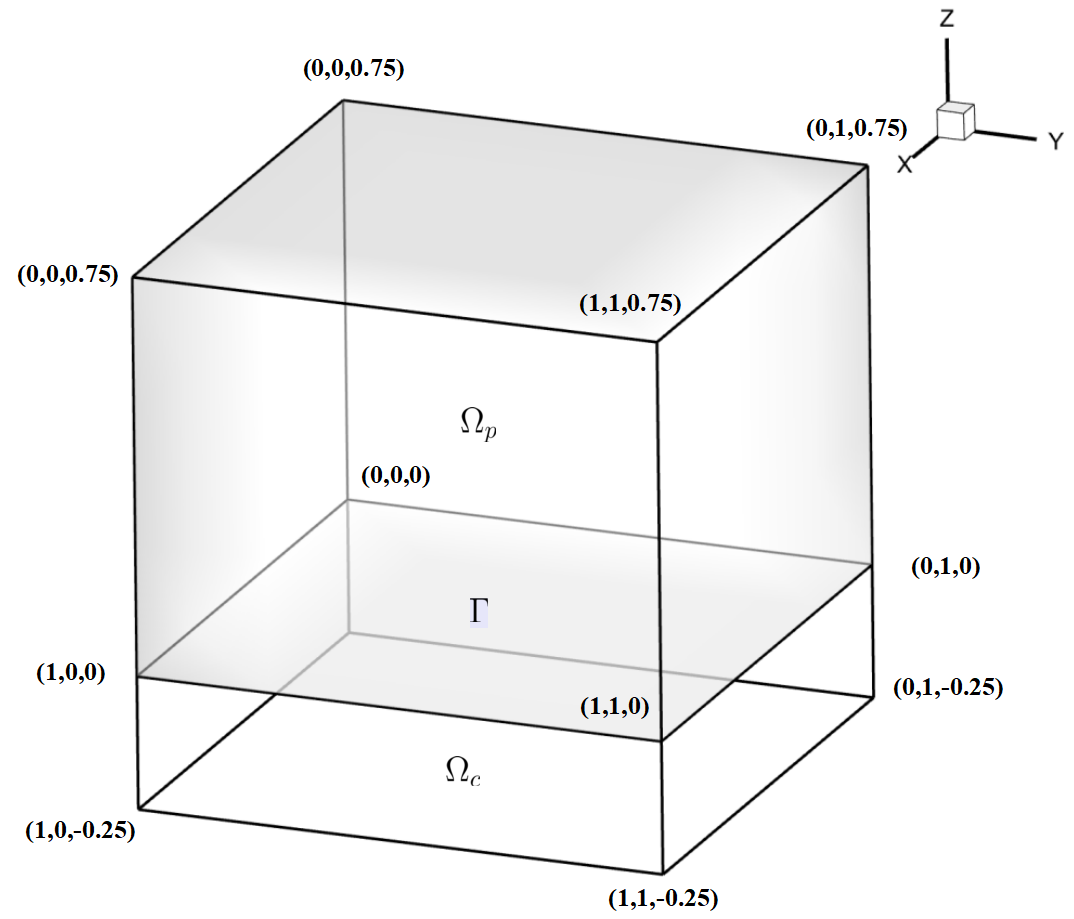}\\
  \caption{\label{P2}\small{3D example model with the triple-porosity region~$\Omega_p$, the conduit region~$\Omega_c$~and the interface~$\Gamma$.}}
\end{figure}
The physical parameters of this model are also simply set~$\phi_i, C_i(i=F,f,m), k_f, k_F, \tilde{\mu}, \sigma, \sigma^{\ast}, \nu, \alpha, \rho, \eta$~equal~1.0 and $k_m$~equals~0.01. We utilize the exact solution below:
\begin{align*}
&p_m=-z+\exp(z)-\exp(-t)\sin(xy)\cos(z),\\
&p_f=-z+\exp(z)-\exp(-t)\sin(xy)\cos(z),\\
&p_F=-z+(-x^2-y^2+8) \exp(-t) \sin(xy) \cos(z),\\
&\boldsymbol{u}_c=\begin{bmatrix} \big( 2x\sin(xy)+y(x^2+y^2-8)\cos(xy) \big) e^{-t}\\
\big( 2y\sin(xy)+x(x^2+y^2-8)\cos(xy) \big) e^{-t}\\
1+\big((x^2+y^2)(x^2+y^2-8)\sin(xy)-4\sin(xy)-8xy\cos(xy)\big)z e^{-t}
\end{bmatrix},\\
&p=\big( -16xy \cos(xy)+(x^2+y^2+z^2-8)(2x^2+2y^2+2z^2-1)\sin(xy)-8\sin(xy) \big) e^{-t},
\end{align*}
which satisfies the source terms, initial conditions, and Dirichlet boundary conditions of the
model.

In this example, the 3D global domain $\Omega$ is decomposed into $2 \times 2 \times 4$
subdomains, which consists of $2 \times 2 \times 2$ in triple-porosity domain $\Omega_p$ and $2 \times 2 \times 2$ in conduit domain, respectively(See Figure \ref{Pic3D}).
\begin{figure}[h]
  \centering
  \includegraphics[width=15cm]{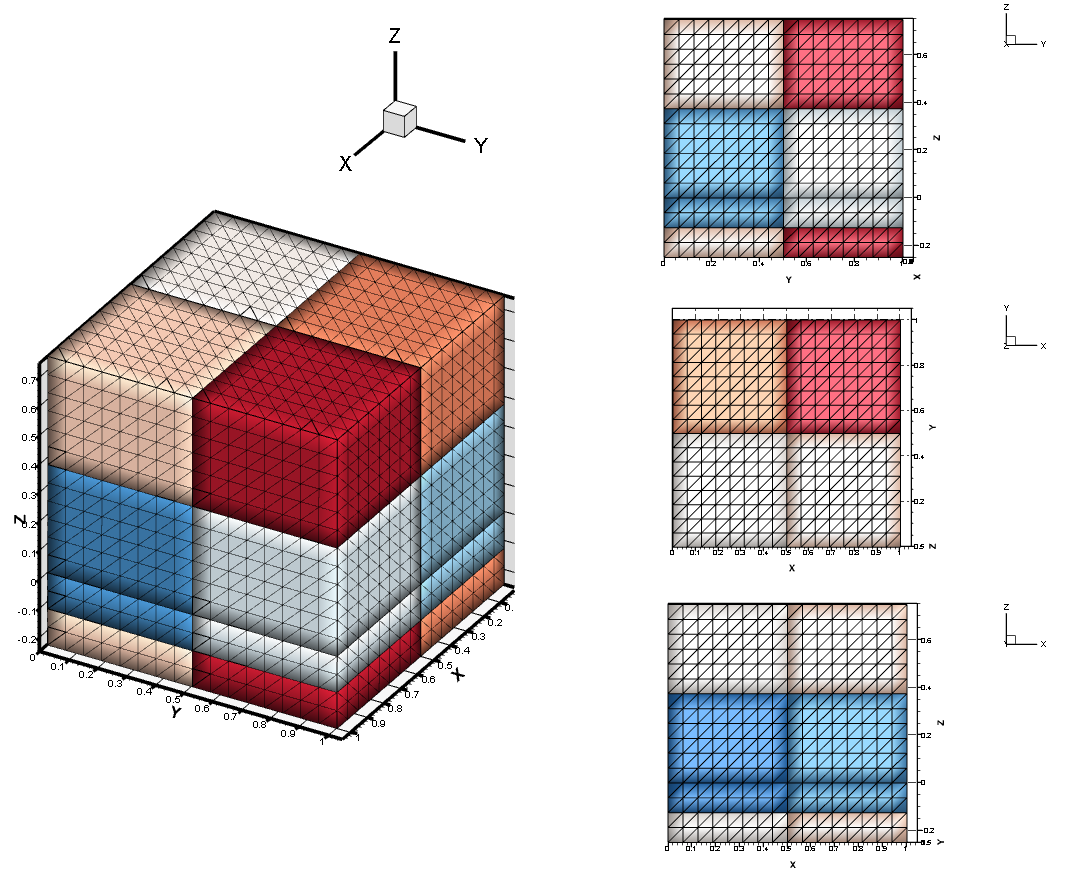}\\
  \caption{\label{Pic3D}\small{3D partitions of triple-porosity region and conduit region.
  Top: main view;
  Middle: top view;
  Bottom: left view.}}
\end{figure}
We set
\begin{align*}
&D_{c1}=[0,\frac{1}{2}] \times [0,\frac{1}{2}] \times [-\frac{1}{8},0],
~~~~~~~D_{c2}=[\frac{1}{2},1] \times [0,\frac{1}{2}] \times [-\frac{1}{8},0],\\
&D_{c3}=[0,\frac{1}{2}] \times [\frac{1}{2},1] \times [-\frac{1}{8},0],
~~~~~~~D_{c4}=[\frac{1}{2},1] \times [\frac{1}{2},1] \times [-\frac{1}{8},0],\\
&D_{c5}=[0,\frac{1}{2}] \times [0,\frac{1}{2}] \times [-\frac{1}{4},-\frac{1}{8}],
~~~~D_{c6}=[\frac{1}{2},1] \times [0,\frac{1}{2}] \times [-\frac{1}{4},-\frac{1}{8}],\\
&D_{c7}=[0,\frac{1}{2}] \times [\frac{1}{2},1] \times [-\frac{1}{4},-\frac{1}{8}],
~~~~D_{c8}=[\frac{1}{2},1] \times [\frac{1}{2},1] \times [-\frac{1}{4},-\frac{1}{8}],\\
&D_{p1}=[0,\frac{1}{2}] \times [0,\frac{1}{2}] \times [0,\frac{3}{8}],
~~~~~~~~~D_{p2}=[\frac{1}{2},1] \times [0,\frac{1}{2}] \times [0,\frac{3}{8}],\\
&D_{p3}=[0,\frac{1}{2}] \times [\frac{1}{2},1] \times [0,\frac{3}{8}],
~~~~~~~~~D_{p4}=[\frac{1}{2},1] \times [\frac{1}{2},1] \times [0,\frac{3}{8}],\\
&D_{p5}=[0,\frac{1}{2}] \times [0,\frac{1}{2}] \times [\frac{3}{8},\frac{3}{4}],
~~~~~~~~~D_{p6}=[\frac{1}{2},1] \times [0,\frac{1}{2}] \times [\frac{3}{8},\frac{3}{4}],\\
&D_{p7}=[0,\frac{1}{2}] \times [\frac{1}{2},1] \times [\frac{3}{8},\frac{3}{4}],
~~~~~~~~~D_{p8}=[\frac{1}{2},1] \times [\frac{1}{2},1] \times [\frac{3}{8},\frac{3}{4}],
\end{align*}
and extend $D_{cj'}, D_{pj}(j', j=1,2,...,8)$ to $\Omega_{cj'}, \Omega_{pj}$ as follows:
\begin{align*}
&\Omega_{c1}=[0,\frac{3}{4}] \times [0,\frac{3}{4}] \times [-\frac{3}{16},0],
~~~~~~~\Omega_{c2}=[\frac{1}{4},1] \times [0,\frac{3}{4}] \times [-\frac{3}{16},0],\\
&\Omega_{c3}=[0,\frac{3}{4}] \times [\frac{1}{4},1] \times [-\frac{3}{16},0],
~~~~~~~\Omega_{c4}=[\frac{1}{4},1] \times [\frac{1}{4},1] \times [-\frac{3}{16},0],\\
&\Omega_{c5}=[0,\frac{3}{4}] \times [0,\frac{3}{4}] \times [-\frac{1}{4},-\frac{1}{16}],
~~~~\Omega_{c6}=[\frac{1}{4},1] \times [0,\frac{3}{4}] \times [-\frac{1}{4},-\frac{1}{16}],\\
&\Omega_{c7}=[0,\frac{3}{4}] \times [\frac{1}{4},1] \times [-\frac{1}{4},-\frac{1}{16}],
~~~~\Omega_{c8}=[\frac{1}{4},1] \times [\frac{1}{4},1] \times [-\frac{1}{4},-\frac{1}{16}],\\
&\Omega_{p1}=[0,\frac{3}{4}] \times [0,\frac{3}{4}] \times [0,\frac{9}{16}],
~~~~~~~~~\Omega_{p2}=[\frac{1}{4},1] \times [0,\frac{3}{4}] \times [0,\frac{9}{16}],\\
&\Omega_{p3}=[0,\frac{3}{4}] \times [\frac{1}{4},1] \times [0,\frac{9}{16}],
~~~~~~~~~\Omega_{p4}=[\frac{1}{4},1] \times [\frac{1}{4},1] \times [0,\frac{9}{16}],\\
&\Omega_{p5}=[0,\frac{3}{4}] \times [0,\frac{3}{4}] \times [\frac{3}{16},\frac{3}{4}],
~~~~~~~~~\Omega_{p6}=[\frac{1}{4},1] \times [0,\frac{3}{4}] \times [\frac{3}{16},\frac{3}{4}],\\
&\Omega_{p7}=[0,\frac{3}{4}] \times [\frac{1}{4},1] \times [\frac{3}{16},\frac{3}{4}],
~~~~~~~~~\Omega_{p8}=[\frac{1}{4},1] \times [\frac{1}{4},1] \times [\frac{3}{16},\frac{3}{4}].
\end{align*}

We solve the model with time size $\Delta t=h$, space mesh sizes~$h=1/16, 1/25, 1/36, 1/49$
and coarse grid mesh size $H$ satisfy $h=H^2$.
The numerical results are shown in Table \ref{T3}, which are also consistent with the theoretical results in Theorem \ref{TheoremFinal}.
Compared to Table \ref{T4}, which is the partitioned time-stepping algorithm,
the parallel algorithm saved approximately $2/3$ of the time cost in computation.
Therefore, the parallel algorithm(Algorithm \ref{algorithm3}) exhibits greater efficiency in the three-dimensional case.


\begin{table}[H]
\caption{\label{T3}The convergence performance and computational cost of Algorithm \ref{algorithm3}(Local Parallel Algorithm) in 3D}
\centering
\resizebox{\textwidth}{!}{
\begin{tabular}{cccccccc}
\hline
$h ~\&~ H$ & $|||\vec{u}_c -\vec{u}_{n+1}^{ch} |||_1$ &
Rate &
$||| p_F -p_{n+1}^{Fh} |||_1$&
Rate &
$||| p_f -p_{n+1}^{fh} |||_0$&
Rate \\
\hline
$\frac{1}{16}~~~~~\frac{1}{4}$  & 0.144418    & --     & 0.067112    &--   &0.000962 & --    \\
$\frac{1}{25}~~~~~\frac{1}{5}$  & 0.090074    & 1.06   & 0.044063    &0.94 &0.000516 & 1.39  \\
$\frac{1}{36}~~~~~\frac{1}{6}$  & 0.063040    & 0.98   & 0.028193    &1.22 &0.000351 & 1.06  \\
$\frac{1}{49}~~~~~\frac{1}{7}$ & 0.046570     & 0.98   & 0.019680    &1.17 &0.000256 & 1.03 \\
\hline
$||| p_f-p_{n+1}^{fh} |||_1$ &
Rate &
$||| p_m-p_{n+1}^{mh} |||_0$ &
Rate &
$||| p_m-p_{n+1}^{mh} |||_1$ &
Rate &
CPU(s)\\
\hline
0.028714 &  --  & 0.002769  &--     &0.035343 & --   & 207.56\\
0.018469 & 0.99 & 0.001863  & 0.89  &0.022940 & 0.97 & 5036.27 \\
0.012953 & 0.97 & 0.001330  & 0.92  &0.016143 & 0.96 & 19146.10
\\
0.009548 & 0.99 & 0.000987  & 0.97  &0.011756 & 1.03 & 31217.50
\\
\hline
\end{tabular}
}
\end{table}

\begin{table}[H]
\caption{\label{T4}The convergence performance and computational cost of Algorithm \ref{Algorithm-1}(Traditional Algorithm) in 3D}
\centering
\resizebox{\textwidth}{!}{
\footnotesize{
\begin{tabular}{cccccccc}
\hline
$h $ & $|\vec{u}_c -\vec{u}_{ch}^{n+1}|_1$ &
Rate &
$|p_F -p_{Fh}^{n+1}|_1$&
Rate &
$\|p_f -p_{fh}^{n+1}\|_0$&
Rate \\
\hline
$\frac{1}{16} $  & 0.146752    & --     & 0.063144    &--   &0.001342 & --
\\
$\frac{1}{25} $  & 0.090678    & 1.08   & 0.043460    &0.84 &0.000612 & 1.76   \\
$\frac{1}{36} $  & 0.064652    & 0.93   & 0.028193    &1.19 &0.000364 & 1.43   \\
$\frac{1}{49} $ & 0.047235    & 1.02    & 0.020150    &1.09 & 0.000265 & 1.03  \\
\hline
$|p_f-p_{fh}^{n+1}|_1$ &
Rate &
$\|p_m-p_{mh}^{n+1}\|_0$ &
Rate &
$|p_m-p_{mh}^{n+1}|_1$ &
Rate &
CPU(s)\\
\hline
0.028997 &  --  & 0.002830  &--     &0.035546 & --   & 500.89\\
0.018563 & 1.00 & 0.001933  & 0.85  &0.023046 & 0.97 &8115.85 \\
0.012313 & 1.13 & 0.001443  & 0.80  &0.016263 & 0.96 &29756.13
\\
0.009163 & 0.96 & 0.000993  & 1.21  &0.011930 & 1.00 &56663.52
\\
\hline
\end{tabular}
}
}
\end{table}

\begin{figure}[H]
\begin{centering}
\begin{subfigure}[t]{0.31\textwidth}
\centering
\includegraphics[width=1.05\textwidth]{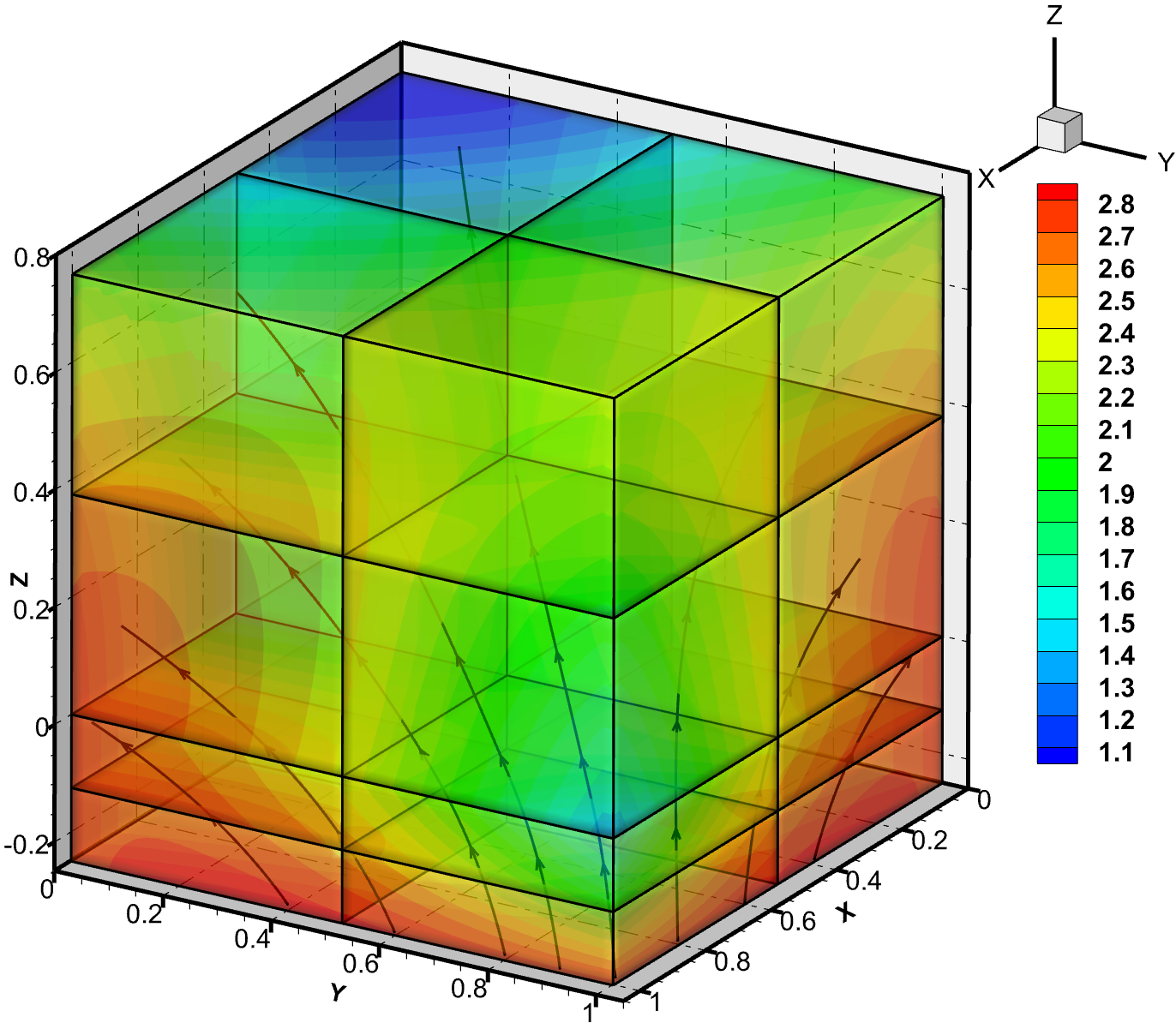}
\end{subfigure}
\quad
\begin{subfigure}[t]{0.31\textwidth}
\centering
\includegraphics[width=1.05\textwidth]{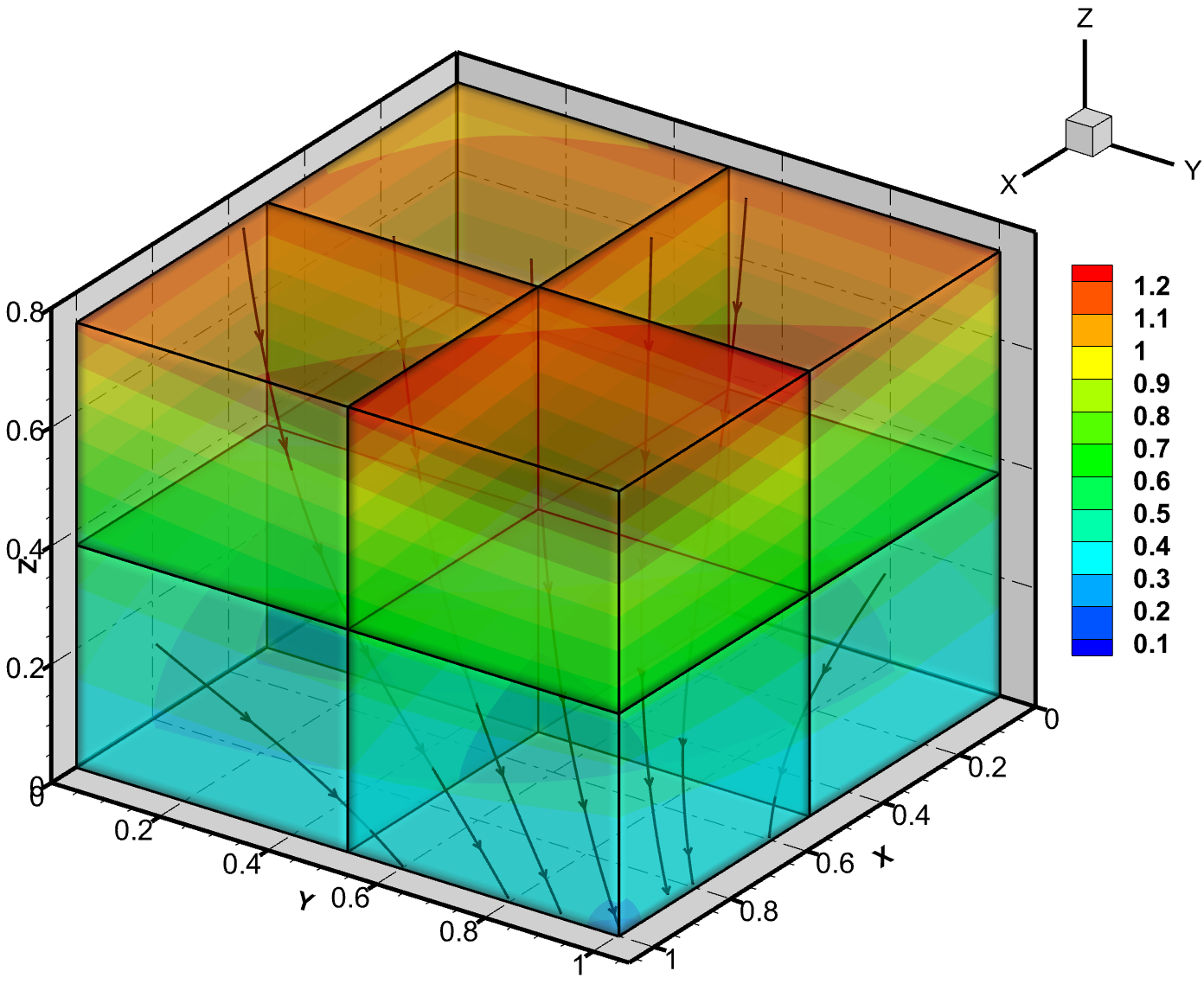}
\end{subfigure}
\quad
\begin{subfigure}[t]{0.31\textwidth}
\centering
\includegraphics[width=1.05\textwidth]{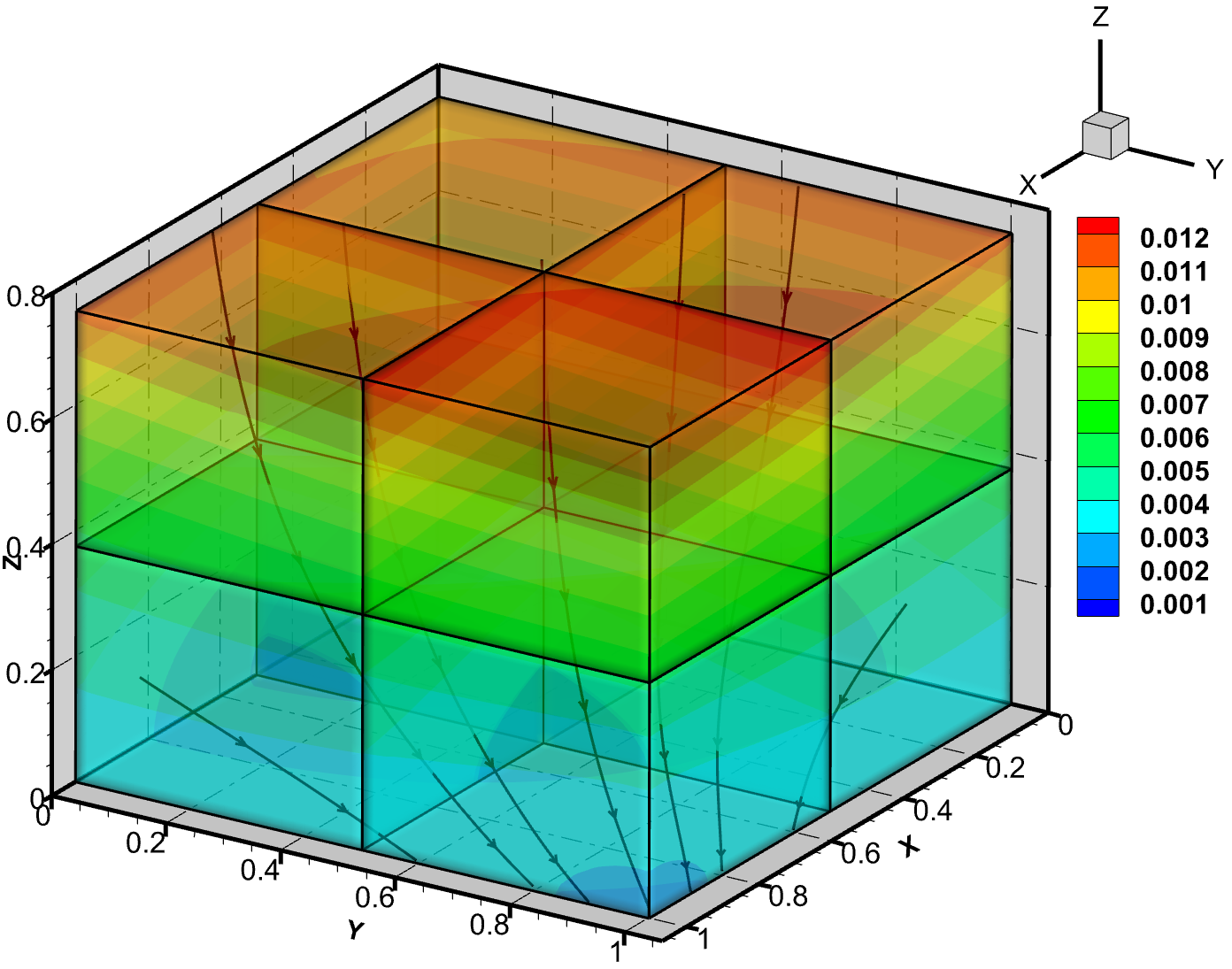}
\end{subfigure}
\end{centering}
\caption{\label{Fig32}\small{The flow speed and streamlines of parallel algorithm in 3D. Left: the flow in macro-fractures and conduits; Middle: the flow in micro-fractures;
Right: the flow in stagnant-matrix.}}
\end{figure}

\vspace{-0.35cm}

\begin{figure}[H]
\begin{centering}
\begin{subfigure}[t]{0.31\textwidth}
\centering
\includegraphics[width=1.05\textwidth]{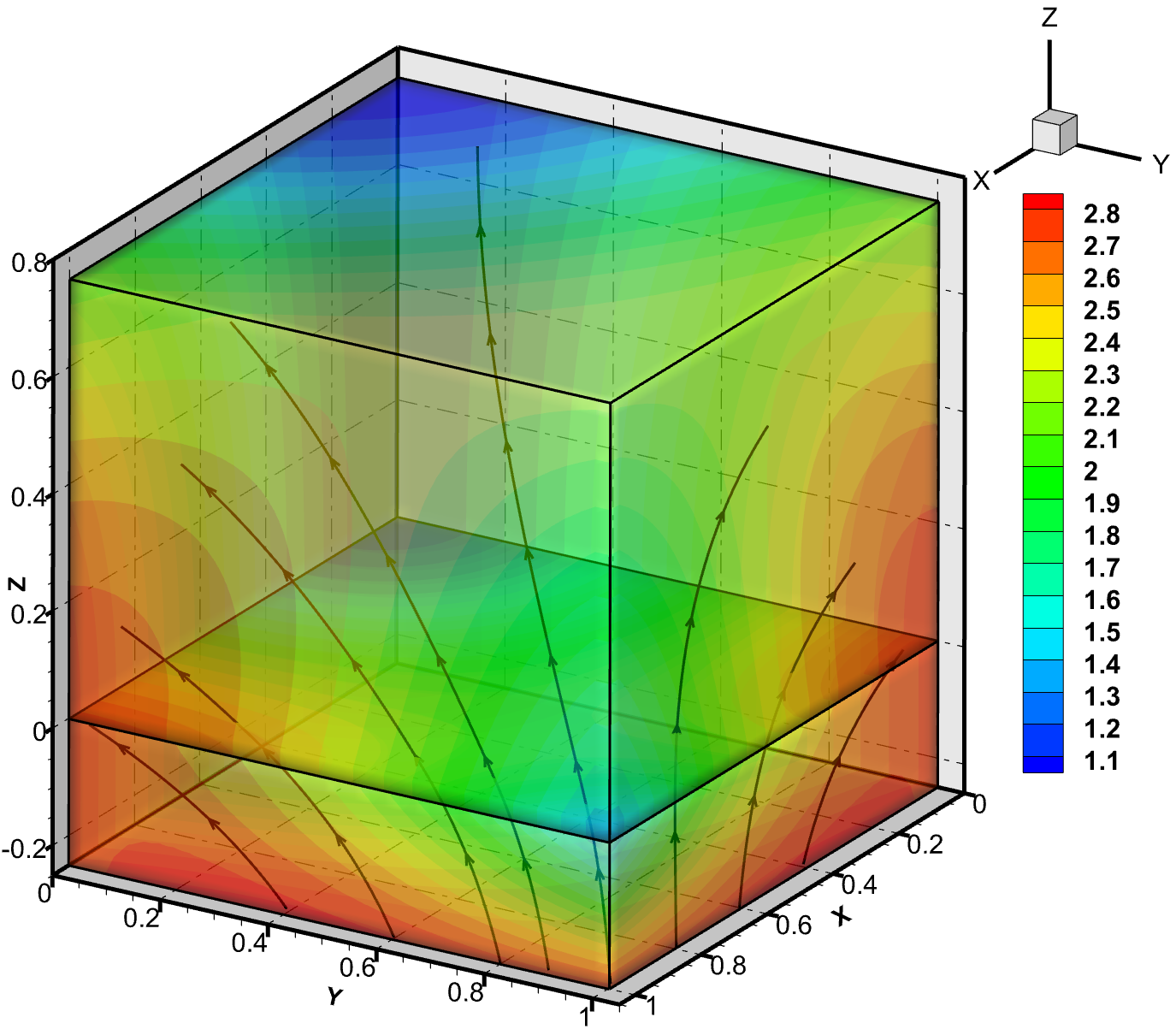}
\end{subfigure}
\quad
\begin{subfigure}[t]{0.31\textwidth}
\centering
\includegraphics[width=1.05\textwidth]{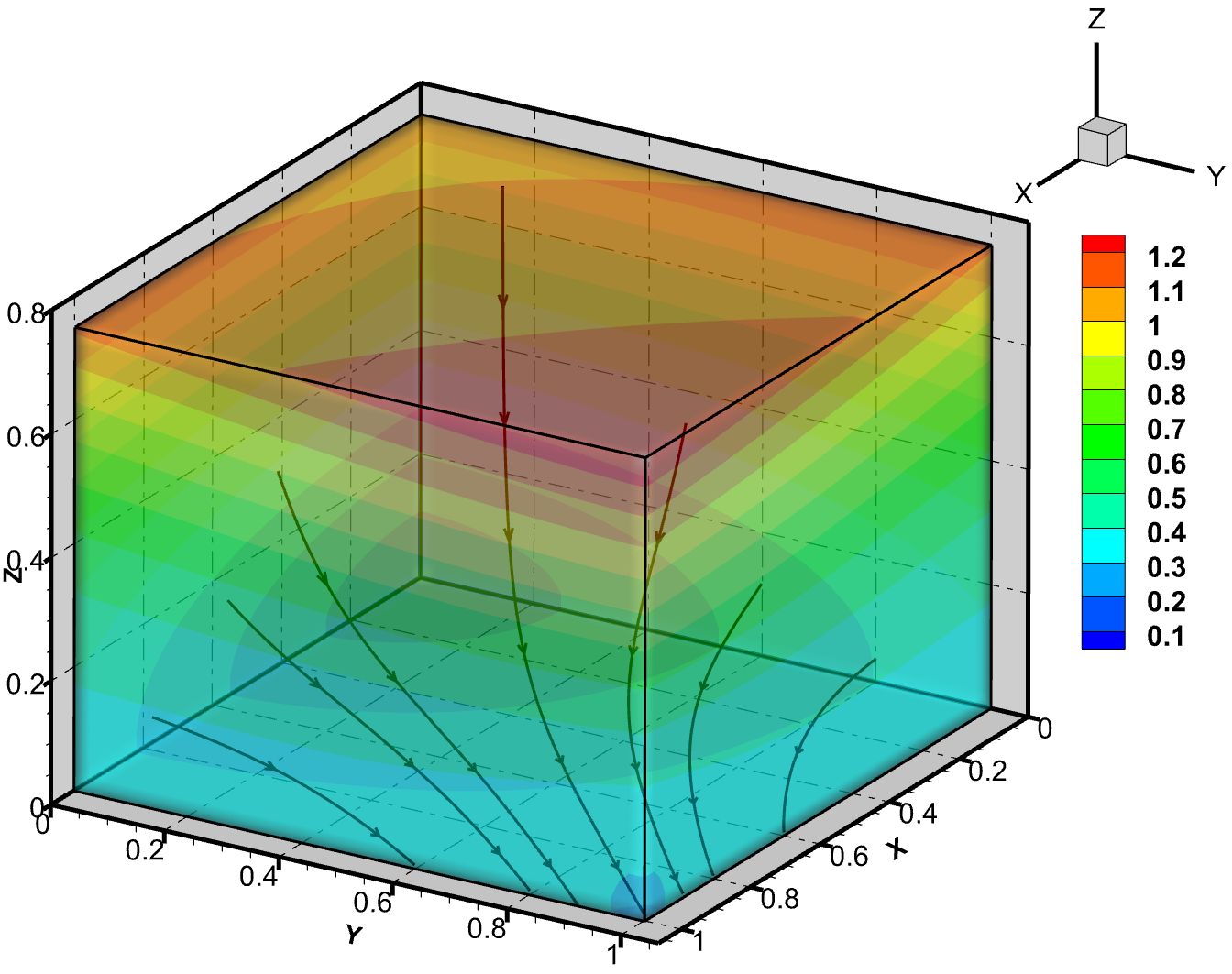}
\end{subfigure}
\quad
\begin{subfigure}[t]{0.31\textwidth}
\centering
\includegraphics[width=1.05\textwidth]{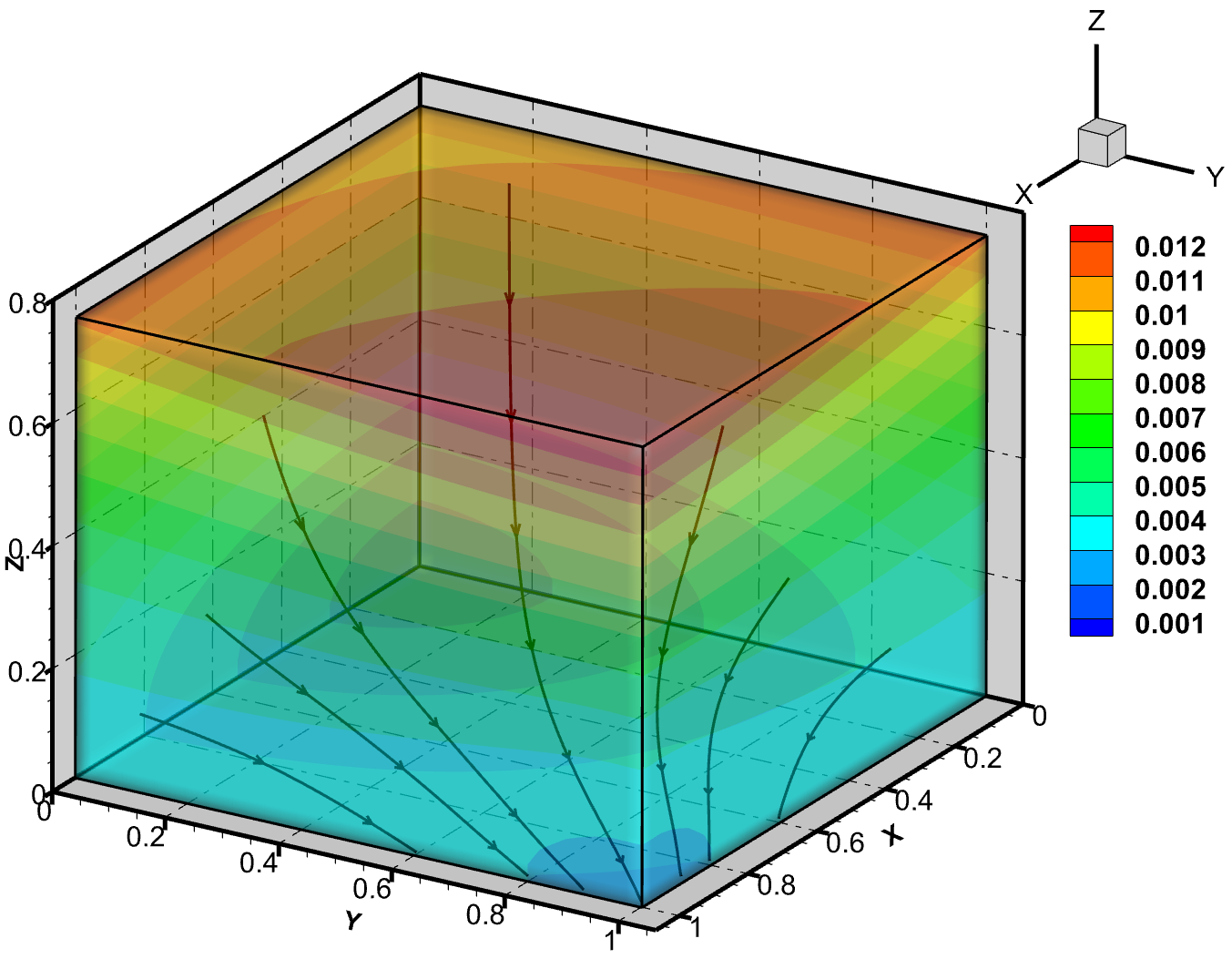}
\end{subfigure}
\end{centering}
\caption{\label{Fig33}\small{The flow speed and streamlines of traditional algorithm in 3D. Left: the flow in macro-fractures and conduits; Middle: the flow in micro-fractures;
Right: the flow in stagnant-matrix.}}
\end{figure}

\vspace{-0.35cm}

\begin{figure}[H]
\begin{centering}
\begin{subfigure}[t]{0.31\textwidth}
\centering
\includegraphics[width=1.05\textwidth]{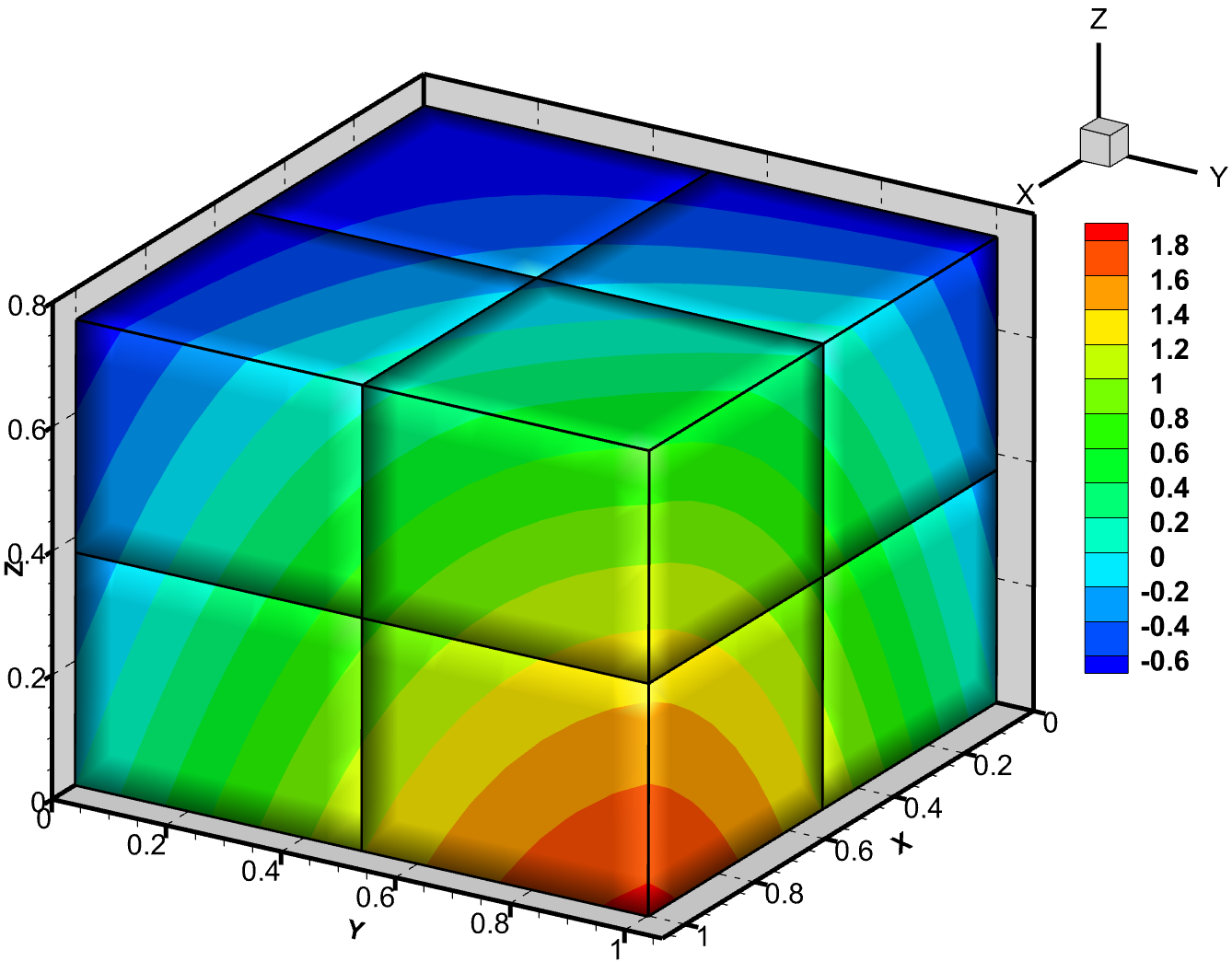}
\end{subfigure}
\quad
\begin{subfigure}[t]{0.31\textwidth}
\centering
\includegraphics[width=1.05\textwidth]{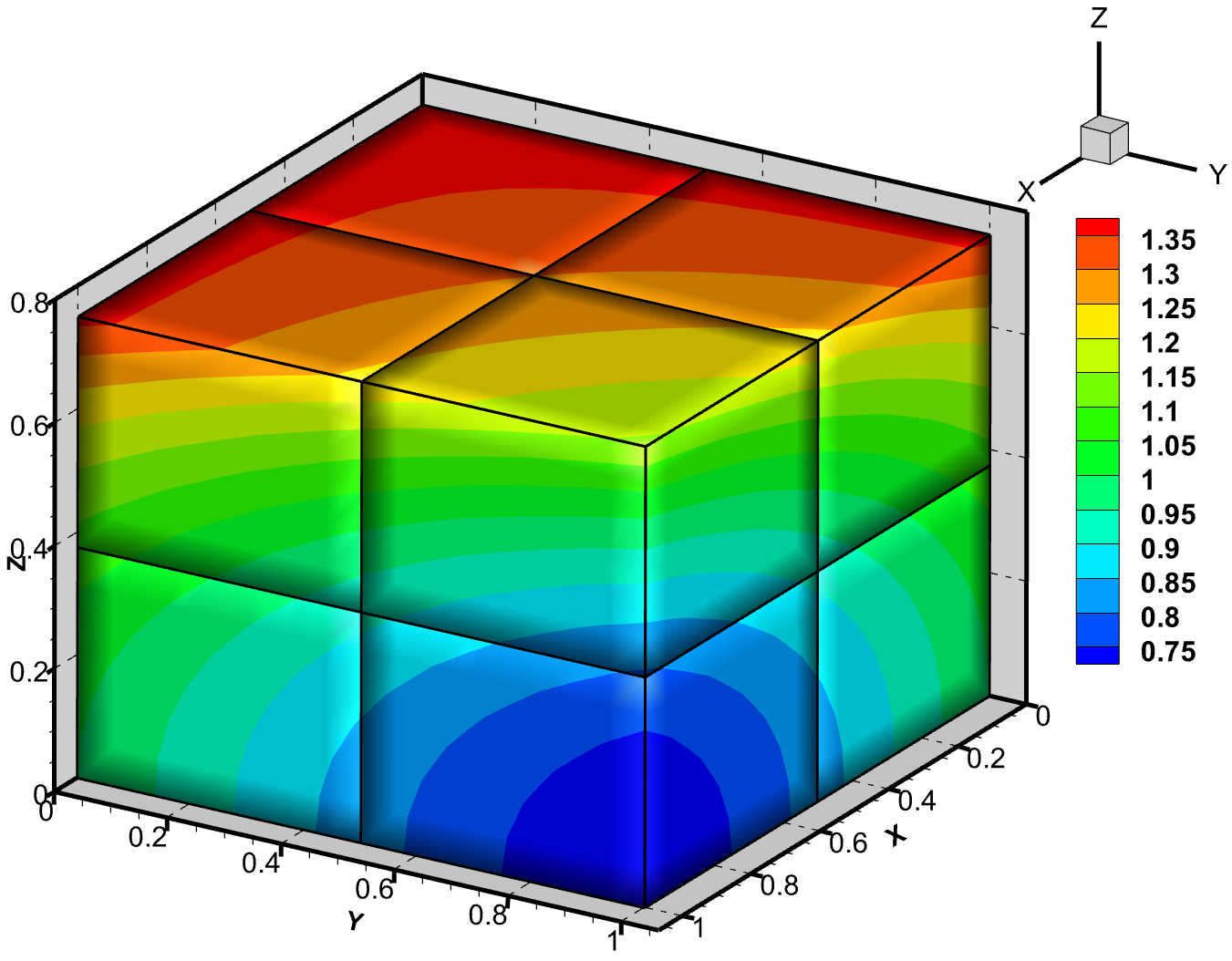}
\end{subfigure}
\quad
\begin{subfigure}[t]{0.31\textwidth}
\centering
\includegraphics[width=1.05\textwidth]{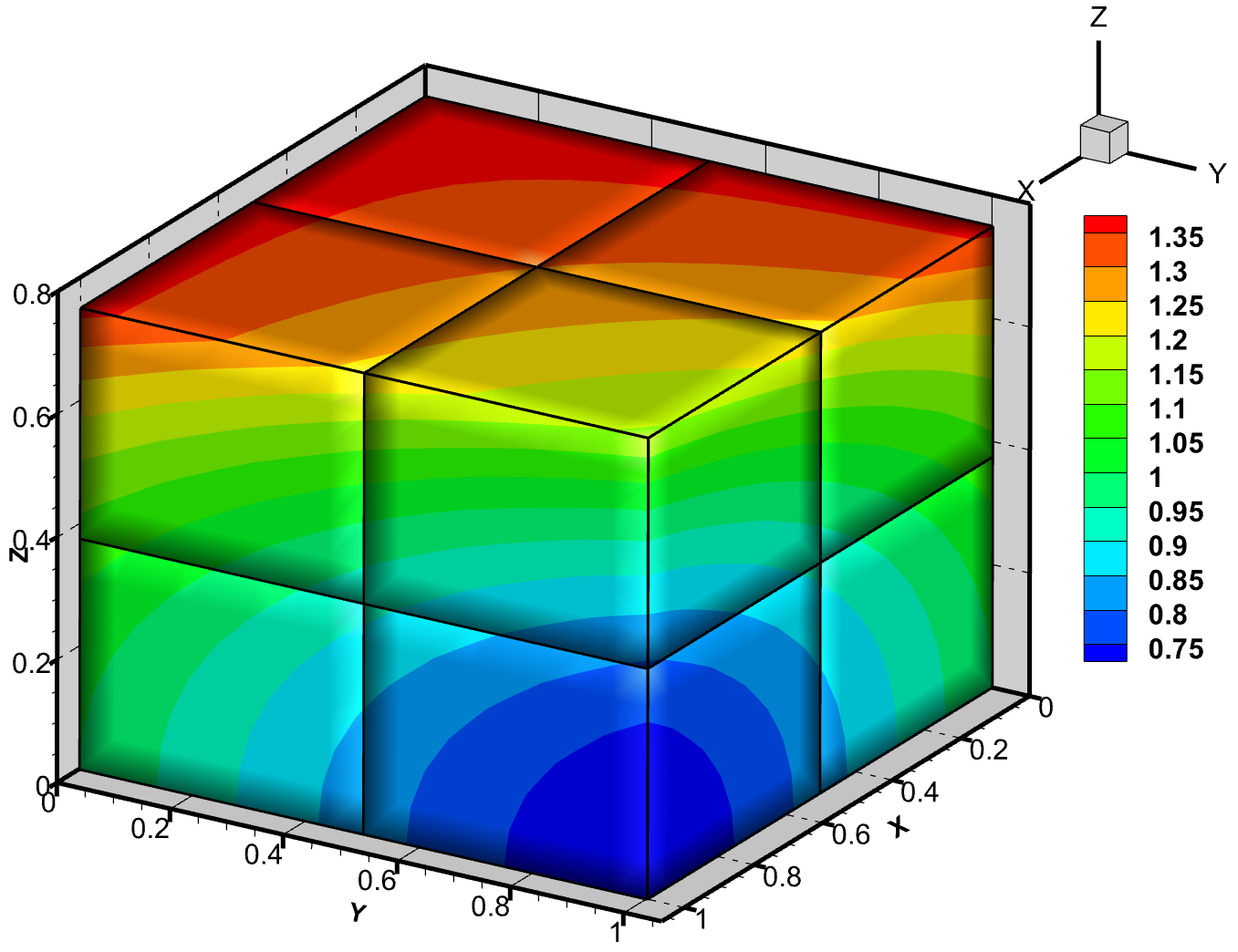}
\end{subfigure}
\end{centering}
\caption{\label{Fig33}\small{The pressure of parallel algorithm in 3D. Left: the flow in macro-fractures and conduits; Middle: the flow in micro-fractures;
Right: the flow in stagnant-matrix.}}
\end{figure}

\vspace{-0.35cm}

\begin{figure}[H]
\begin{centering}
\begin{subfigure}[t]{0.31\textwidth}
\centering
\includegraphics[width=1.05\textwidth]{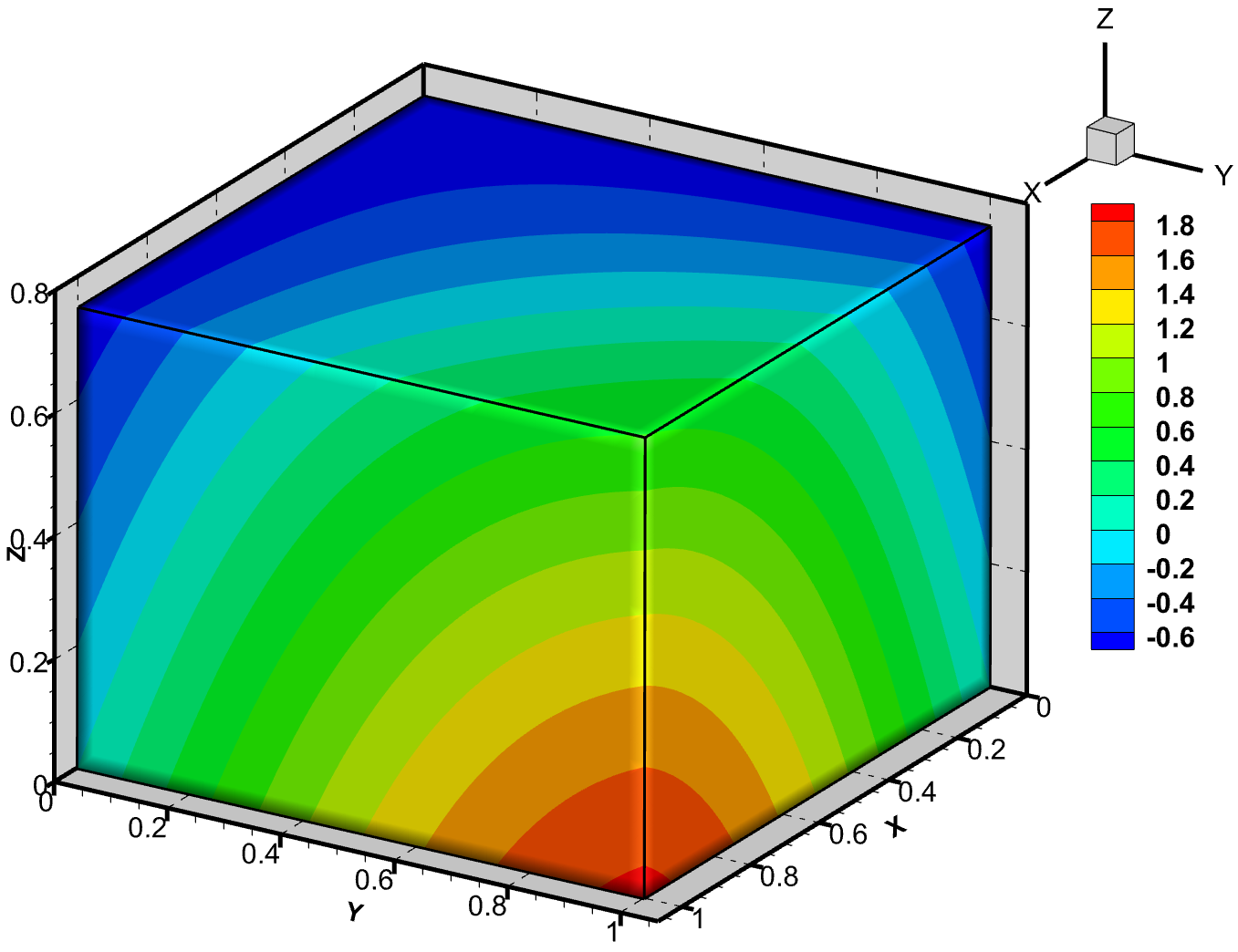}
\end{subfigure}
\quad
\begin{subfigure}[t]{0.31\textwidth}
\centering
\includegraphics[width=1.05\textwidth]{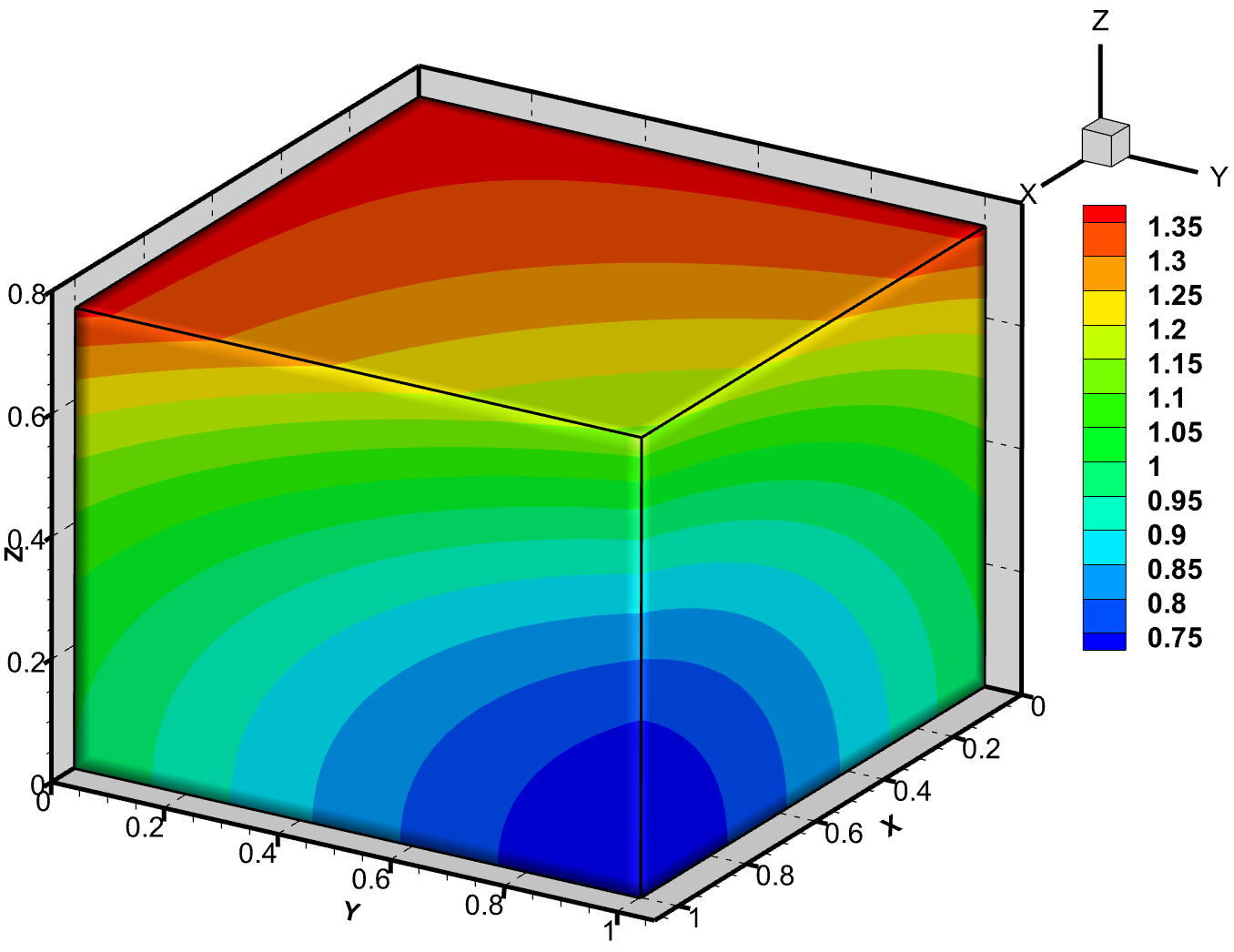}
\end{subfigure}
\quad
\begin{subfigure}[t]{0.31\textwidth}
\centering
\includegraphics[width=1.05\textwidth]{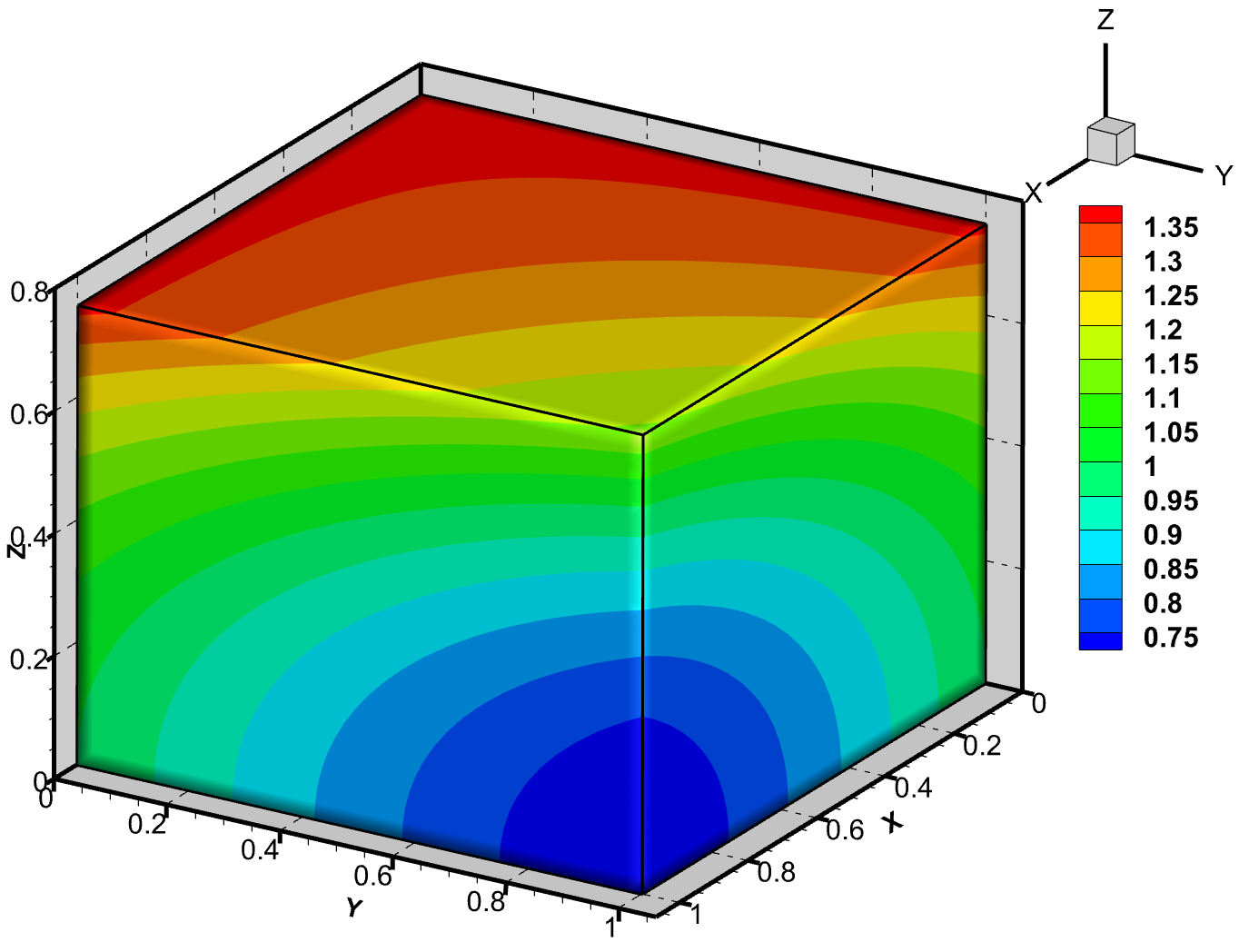}
\end{subfigure}
\end{centering}
\caption{\label{Fig34}\small{The pressure of traditional algorithm in 3D. Left: the flow in macro-fractures and conduits; Middle: the flow in micro-fractures;
Right: the flow in stagnant-matrix.}}
\end{figure}

Simultaneously, we have separately plotted the velocity streamlines and pressure diagrams corresponding to Algorithm \ref{algorithm3} and Algorithm \ref{Algorithm-1} in the three-dimensional case, as shown in the Figure \ref{Fig32}-\ref{Fig34}.
It is clear that they are nearly identical, and the computational performance in 3D is excellent.

\subsection{Example 3: Multistage hydraulically fractured horizontal wellbore completions with super-hydrophobic proppant }
The technique of multistage hydraulically fracturing a horizontal wellbore with cased-hole completions plays an important role in unconventional reservoirs, especially for shale oil and gas production \cite{seale2006multistage, bello2010multi, ZhangChenZhao}.
To improve the recovery rate, proppants with oil-permeable and water-resistant properties are typically used(See Figure \ref{Fracturing}).
Moreover, the material properties of the proppant can significantly and directly impact the permeability of the formation \cite{tan2018laboratory, tan2017experimental, zhang2020numerical}.
In this example, we simulate the flow around a multistage hydraulically fractured horizontal production wellbore with super-hydrophobic proppant and illustrate its flow pattern.
The horizontal cross-section is displayed in Figure \ref{PicPrac}.

\begin{figure}[h]
  \centering
  \includegraphics[width=13cm]{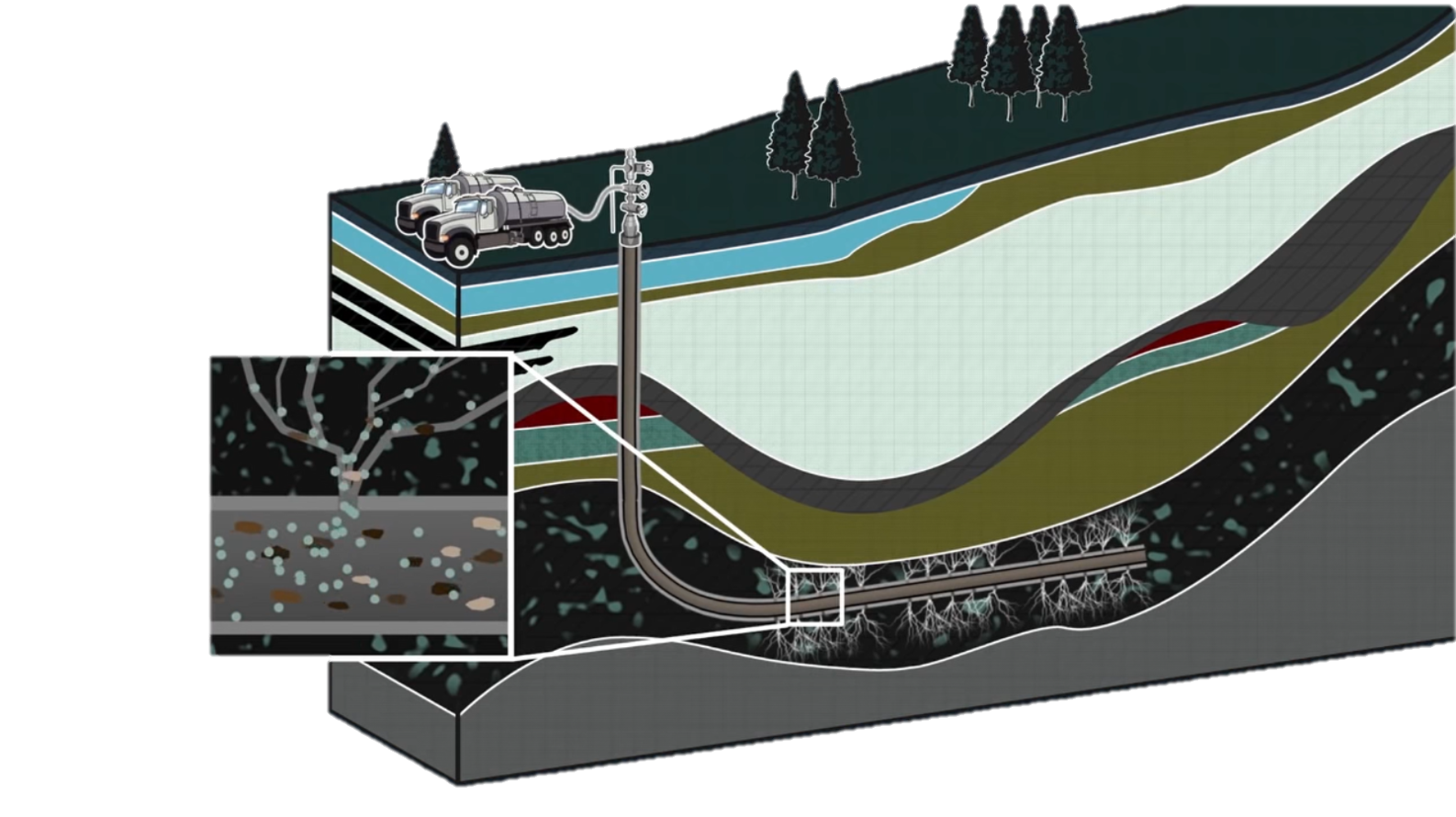}\\
  \caption{\label{Fracturing}\small{The pictorial illustration of the unconventional naturally fractured reservoir with multistage hydraulic fracturing.( https://www.youtube.com/watch?v=2PBCTXHqZec)}}
\end{figure}

\begin{figure}[htbp]
	\centering
	\begin{minipage}{0.49\linewidth}
		\centering
		\includegraphics[width=0.8\linewidth]{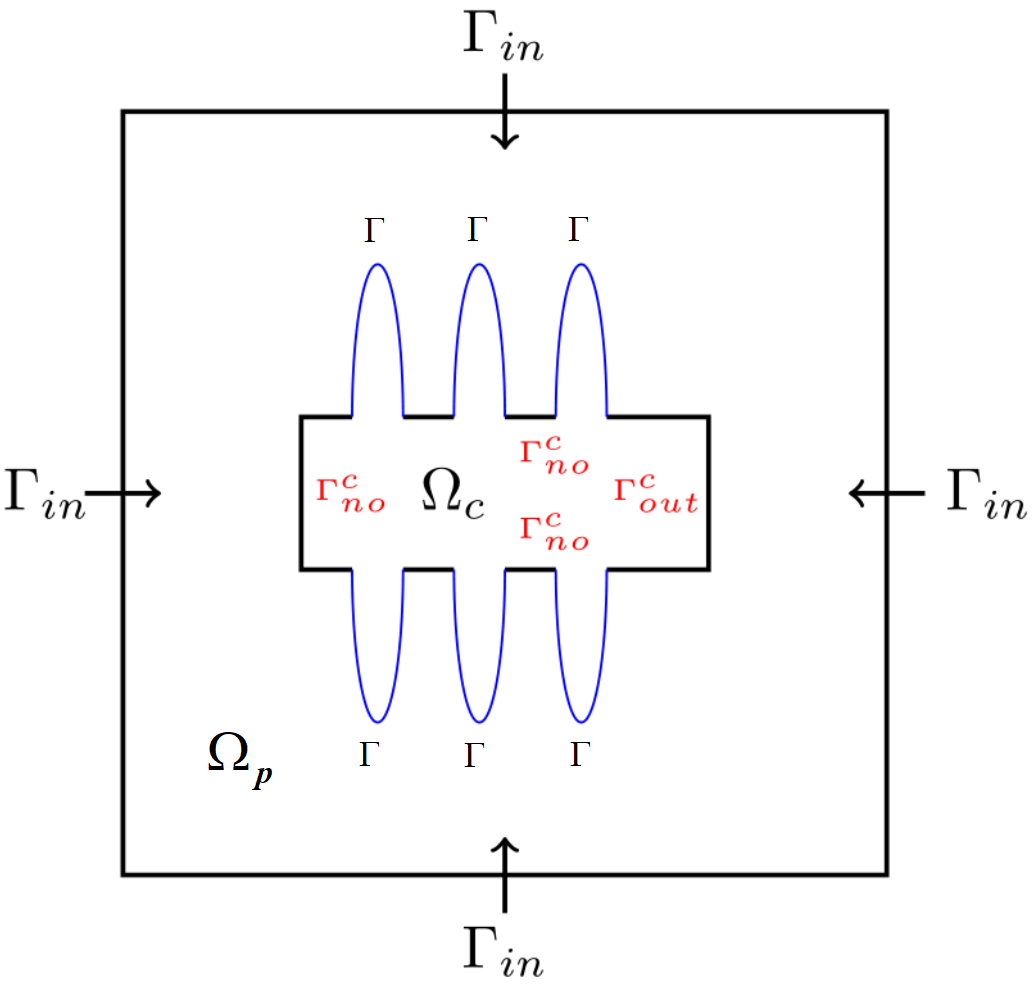}
\caption{\label{PicPrac}\small{The horizontal cross-section sketch of the multistage hydraulically fractured system.}}
	\end{minipage}
        \begin{minipage}{0.5\linewidth}
		\centering
		\includegraphics[width=0.8\linewidth]{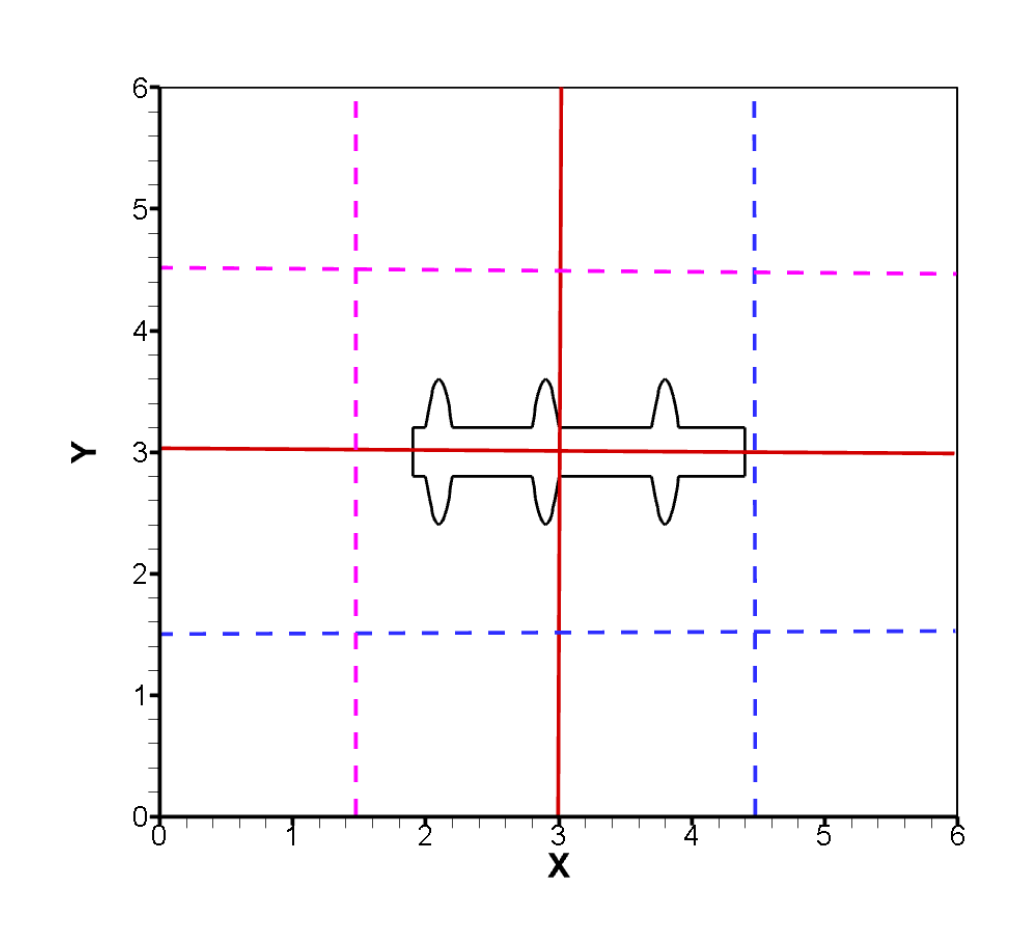}
\caption{\label{PicPracDivide}\small{The partitions of horizontal cross-section domain.}}
	\end{minipage}
\end{figure}

The simulation domain is the square~$[0,6]^2$, where the horizontal wellbore region~$\Omega_c$~is simplified as a rectangle of ~$[1.9, 4.4] \times [2.4, 3.6]$. The triple-porosity region is~$\Omega_p=[0,6]^2 \backslash \Omega_c$\\
and its boundary is~$\Gamma_{\text{in}}$. We assume the pressures~$p_m=4\times 10^3, p_f=1.6\times 10^3$ and $p_F=1.0\times 10^3$~on~$\Gamma_{\text{in}}$. Since the boundary~$\Gamma_{\text{no}}^c$~is equipped with cased holes, the horizontal wellbore does not directly communicate with the triple-porosity medium. Therefore, the following boundary conditions are imposed:
\begin{equation*}
-\frac{k_F}{\tilde{\mu}}\nabla p_F \cdot (-\vec{n}_{c})=0,
,~~~~-\frac{k_f}{\tilde{\mu}}\nabla p_f \cdot (-\vec{n}_{c})=0,~~~~-\frac{k_m}{\tilde{\mu}}\nabla p_m \cdot (-\vec{n}_{c}) =0,~~~~\vec{u}_c \cdot \vec{n}_{c}=0~~~~\text{on}~\Gamma_{\text{no}}^c.
\end{equation*}
However, the fluid in natural fractures can flow into hydraulic fractures through interface~$\Gamma$, which is the only path connecting the triple-porosity domain~$\Omega_p$~and the horizontal wellbore~$\Omega_c$.
The location of the super-hydrophobic proppant within natural fractures (more-permeable macrofractures) serves two purposes:
it provides support to the fractures and also functions to permit the passage of oil while obstructing water.

The interface boundary conditions~\eqref{km_interface}-\eqref{uc_interface} are applied in~$\Gamma$. In practice, a horizontal wellbore is connected to a vertical wellbore at the boundary~$\Gamma_{\text{out}}^c$~and we do not show this part in Figure \ref{PicPrac} for simplicity. In detail, the fluid in~$\Omega_c$~does not communicate with~$\Omega_p$~but directly flows out of the horizontal wellbore to the vertical wellbore. Therefore, the following boundary conditions are considered on~$\Gamma_{\text{out}}^c$:
\begin{equation*}
-\frac{k_F}{\tilde{\mu}} \nabla p_F \cdot (-\vec{n}_{c}) =0,
~~~~-\frac{k_f}{\tilde{\mu}} \nabla p_f \cdot (-\vec{n}_{c}) =0,~~~~-\frac{k_m}{\tilde{\mu}}\nabla p_m \cdot (-\vec{n}_{c})=0,~~~~\mathbb{T}(\vec{u}_c,p)\vec{n}_{c}=0~~~~\text{on}~\Gamma_{\text{out}}^c.
\end{equation*}

\subsubsection{Simulation of the flow behavior around multistage fractured horizontal wellbore completions}\label{Example3-1}

As we all know, the flow velocity in the triple-porosity region is slower than that in the conduit region.
The interface $\Gamma$ serves as a transitional layer between fluids in two different regions, and the fluid closer to the interface in the pipe flow appears in a laminar flow form.
Therefore, the flow in horizontal wellbore is described by the Stokes equation.
Some parameters of this model are chosen as~$\phi_m=10^{-2}, \phi_f=10^{-3},
\phi_F=10^{-4}, C_{m}=10^{-4}, C_{f}=10^{-4}, C_F=10^{-4},
k_m=10^{-8}, k_f=10^{-6}, k_F=10^{-3},
\mu=10^{-2}, \nu=10^{-2}, \sigma=0.5, \rho=10.0, \alpha=1.0, \eta=1.0, q_F=0, q_f=0,
q_m=0$ and $\vec{f}_c=0$. The simulation is carried out with the step sizes of~$H=1/3, h=1/9$~and
$\Delta t=0.05$.

Due to the complex physical geometry and comprehensive interface/boundary conditions of this hydraulic fracturing system, the computational domain is divided into four subdomains as shown in Figure \ref{PicPracDivide}. The blue dashed line and the purple dashed line represent the extension of the subdomains.

As we can see, Figure \ref{Fig55p} and \ref{Fig56p} present the pressure around the multistage hydraulically fractured
horizontal production wellbore with cased-hole completions at~$T=10.0$, which is used in Algorithm \ref{algorithm3} and Algorithm \ref{Algorithm-1}, respectively.
The pictures from left to right are the pressure in macro-fractures and multistage hydraulically fractured horizontal wellbore,
the pressure in micro-fractures
and the pressure in the stagnant-matrix.
One can observe that the matrix has
higher pressure which supplies the fluid to the macro-fractures. The vertical wellbore connected to
$\Gamma^c_{\text{out}}$~provides a pathway to the outside environment for the fluid. Therefore,
the pressure in the horizontal wellbore is lower compared with that in the domain farther away from the well
which is represented by the blue color. Correspondingly,
Figure \ref{Fig55} and \ref{Fig56} display the velocity and streamlines, which are also nearly the same.
The pictures from left to right are the flow in macro-fractures and multistage hydraulically fractured horizontal wellbore,
the flow in micro-fractures
and the flow in the stagnant-matrix.
As expected, the fluid in the matrix domain which has higher pressure is pushed into the micro-fractures, macro-fractures and then the horizontal wellbore sequentially. The cased-hole completions seal the interface between the horizontal wellbore and the triple-porosity medium, and the horizontal wellbore does not directly communicate with the triple-porosity
medium but is only fed by the hydraulic fractures through interface~$\Gamma$. These observations also match
with the benchmark study performed by Hou et al. \cite{hou2016a, hou2022modeling} and Mahbub et al. \cite{al2019coupled}.

\vspace{-0.4cm}

\begin{figure}[H]
\begin{centering}
\begin{subfigure}[t]{0.31\textwidth}
\centering
\includegraphics[width=1.05\textwidth]{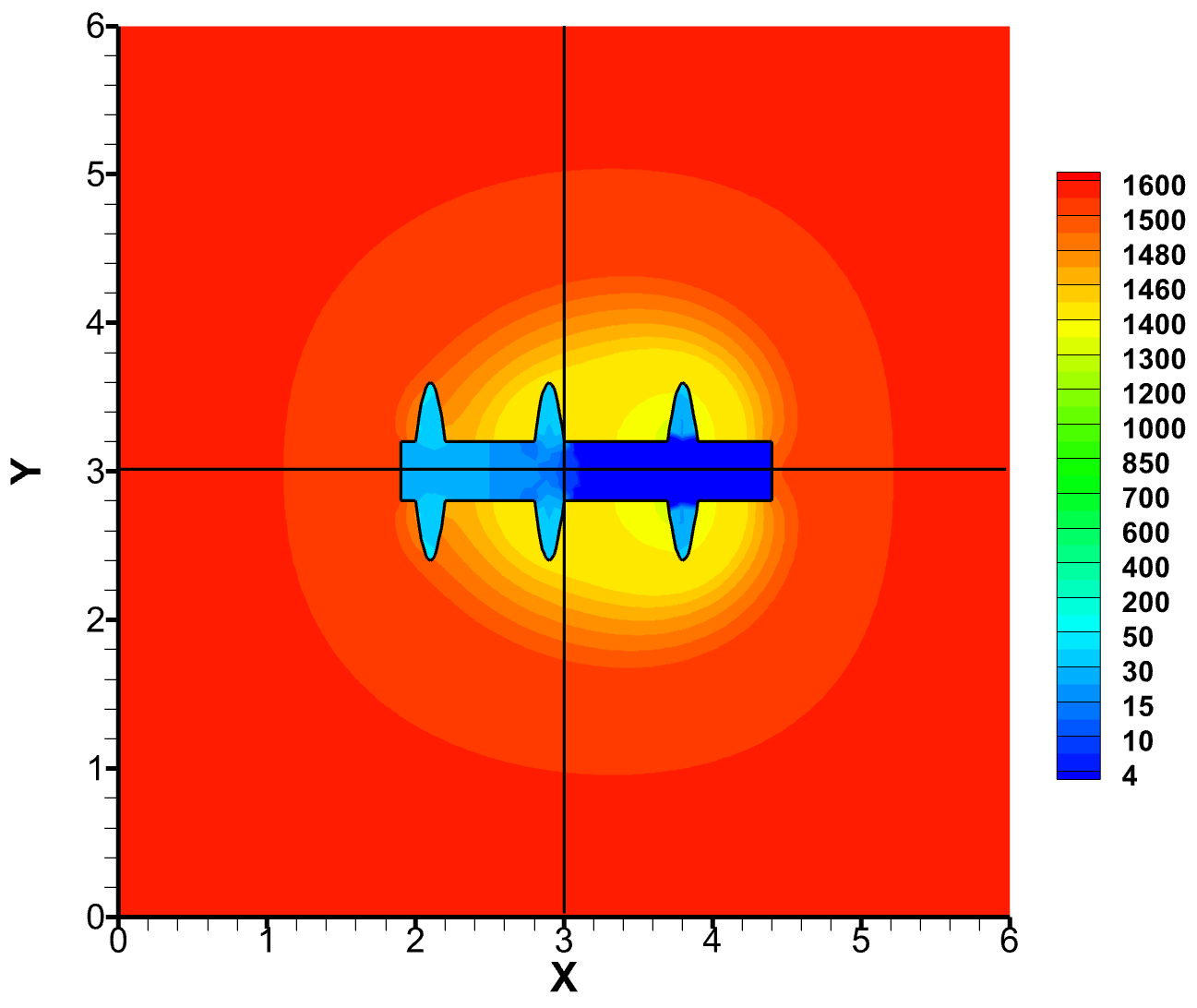}
\end{subfigure}
\quad
\begin{subfigure}[t]{0.31\textwidth}
\centering
\includegraphics[width=1.05\textwidth]{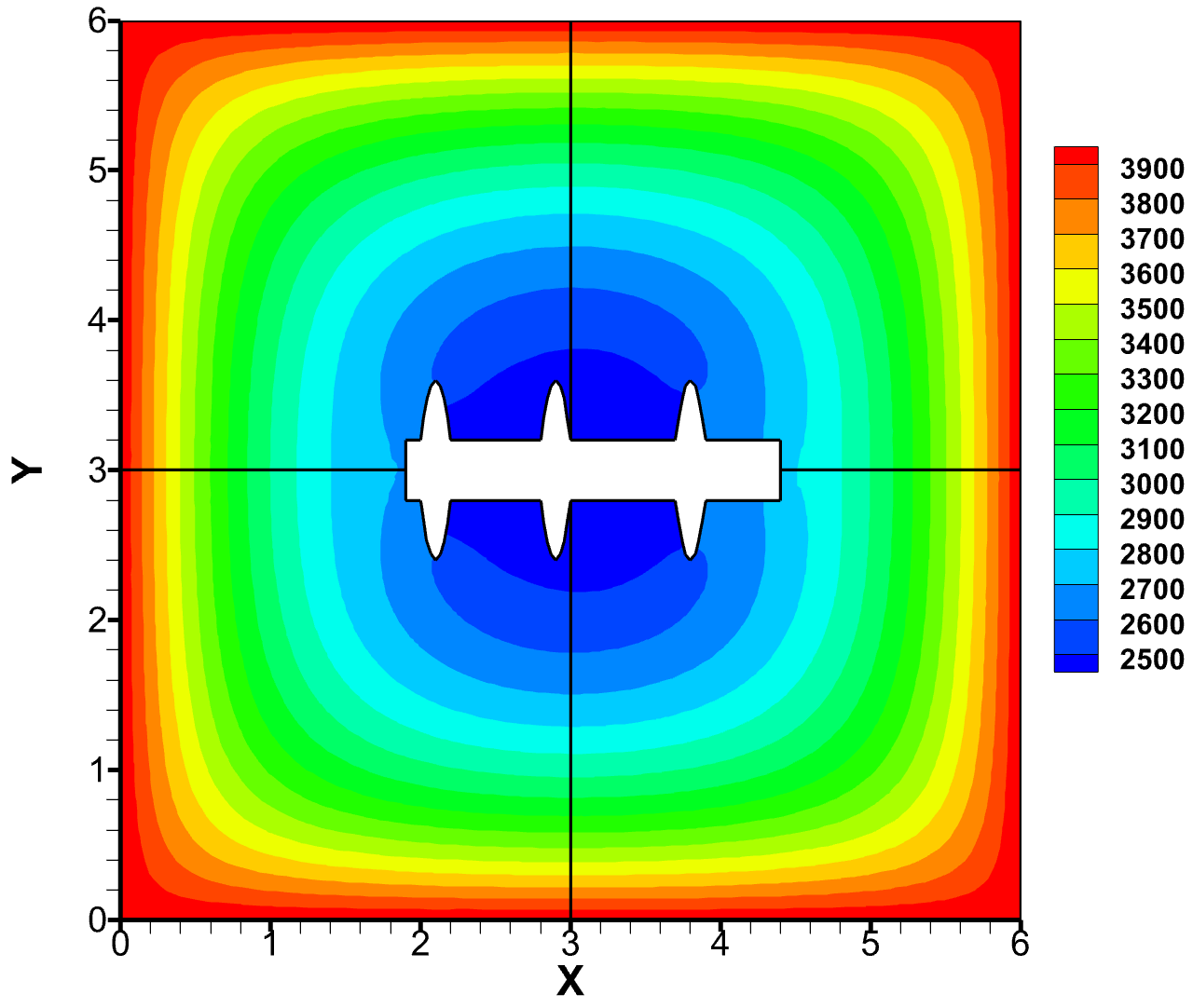}
\end{subfigure}
\quad
\begin{subfigure}[t]{0.31\textwidth}
\centering
\includegraphics[width=1.05\textwidth]{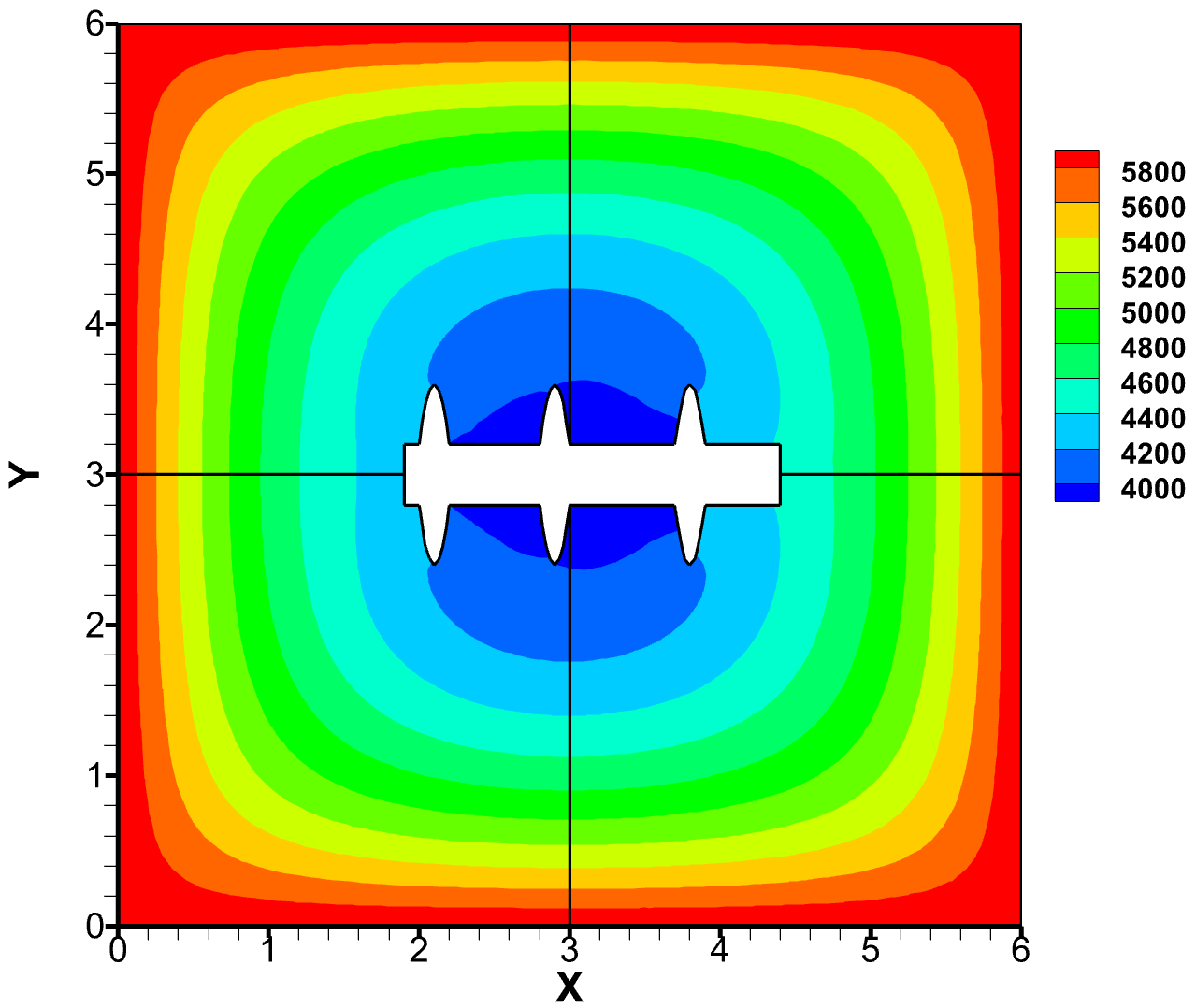}
\end{subfigure}
\end{centering}
\caption{\label{Fig55p}\small{The pressure around the multistage hydraulically fractured horizontal production wellbore with cased-hole completions in Algorithm \ref{algorithm3}. 
}}
\end{figure}

\vspace{-0.5cm}

\begin{figure}[H]
\begin{centering}
\begin{subfigure}[t]{0.31\textwidth}
\centering
\includegraphics[width=1.05\textwidth]{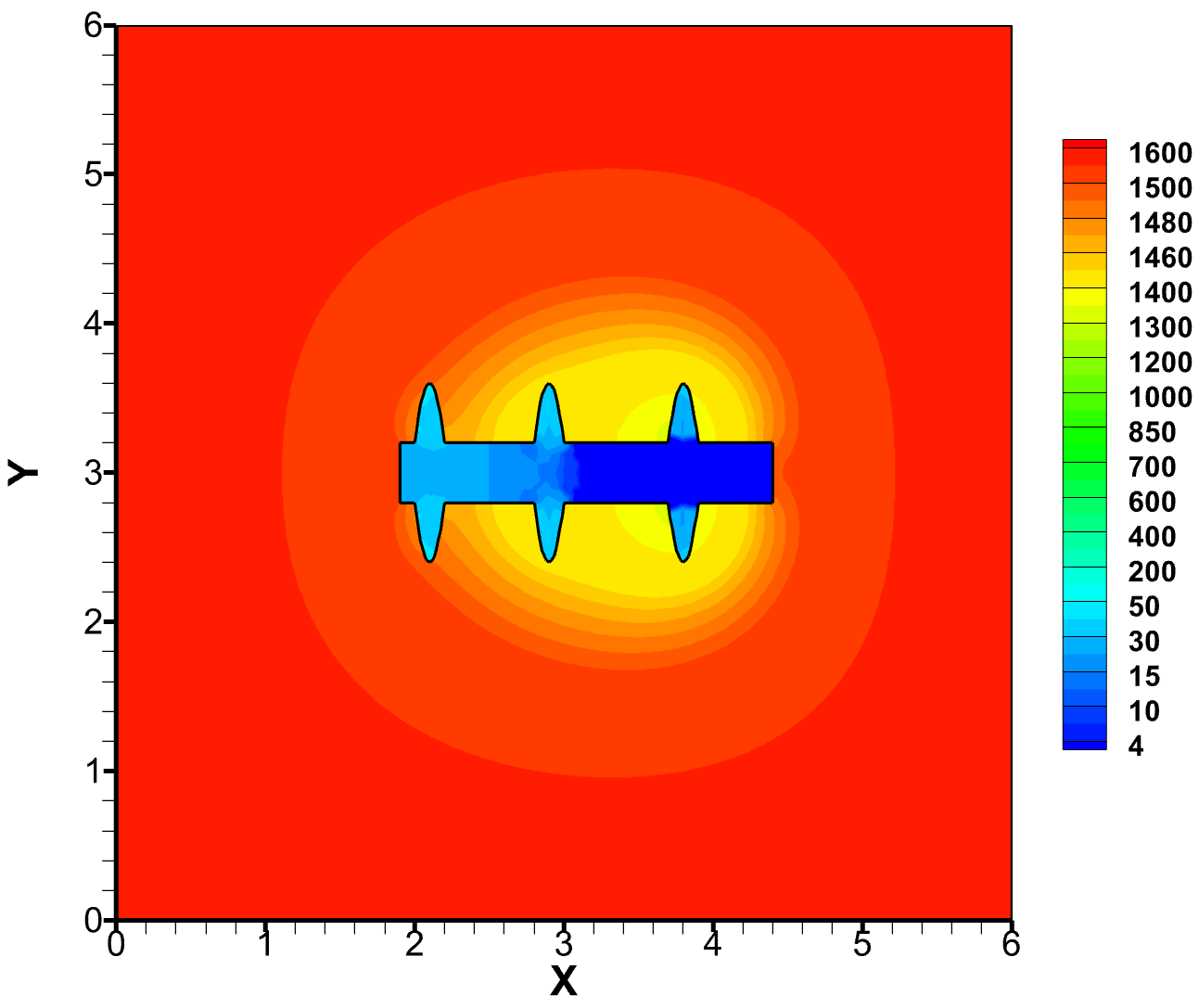}
\end{subfigure}
\quad
\begin{subfigure}[t]{0.31\textwidth}
\centering
\includegraphics[width=1.05\textwidth]{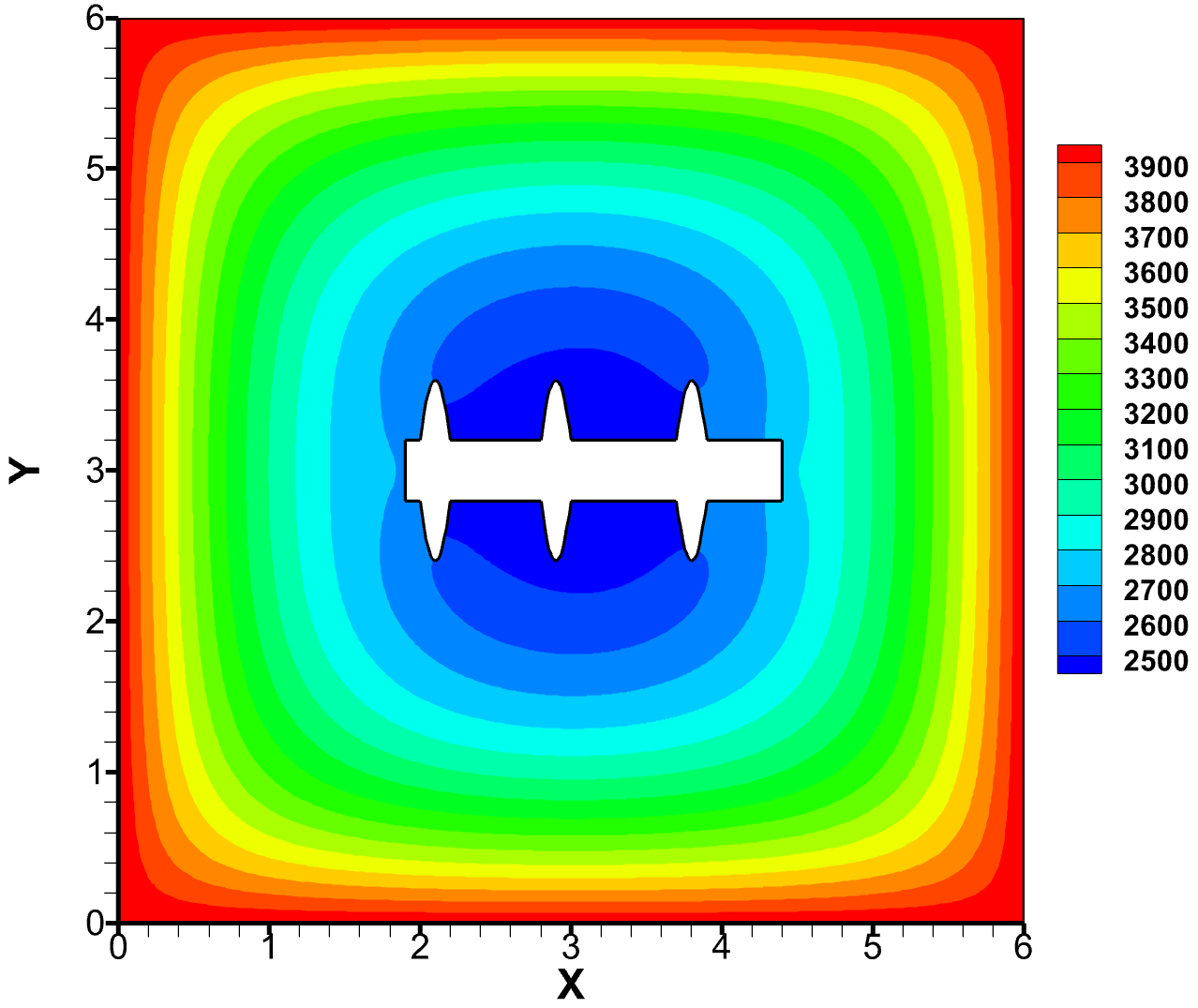}
\end{subfigure}
\quad
\begin{subfigure}[t]{0.31\textwidth}
\centering
\includegraphics[width=1.05\textwidth]{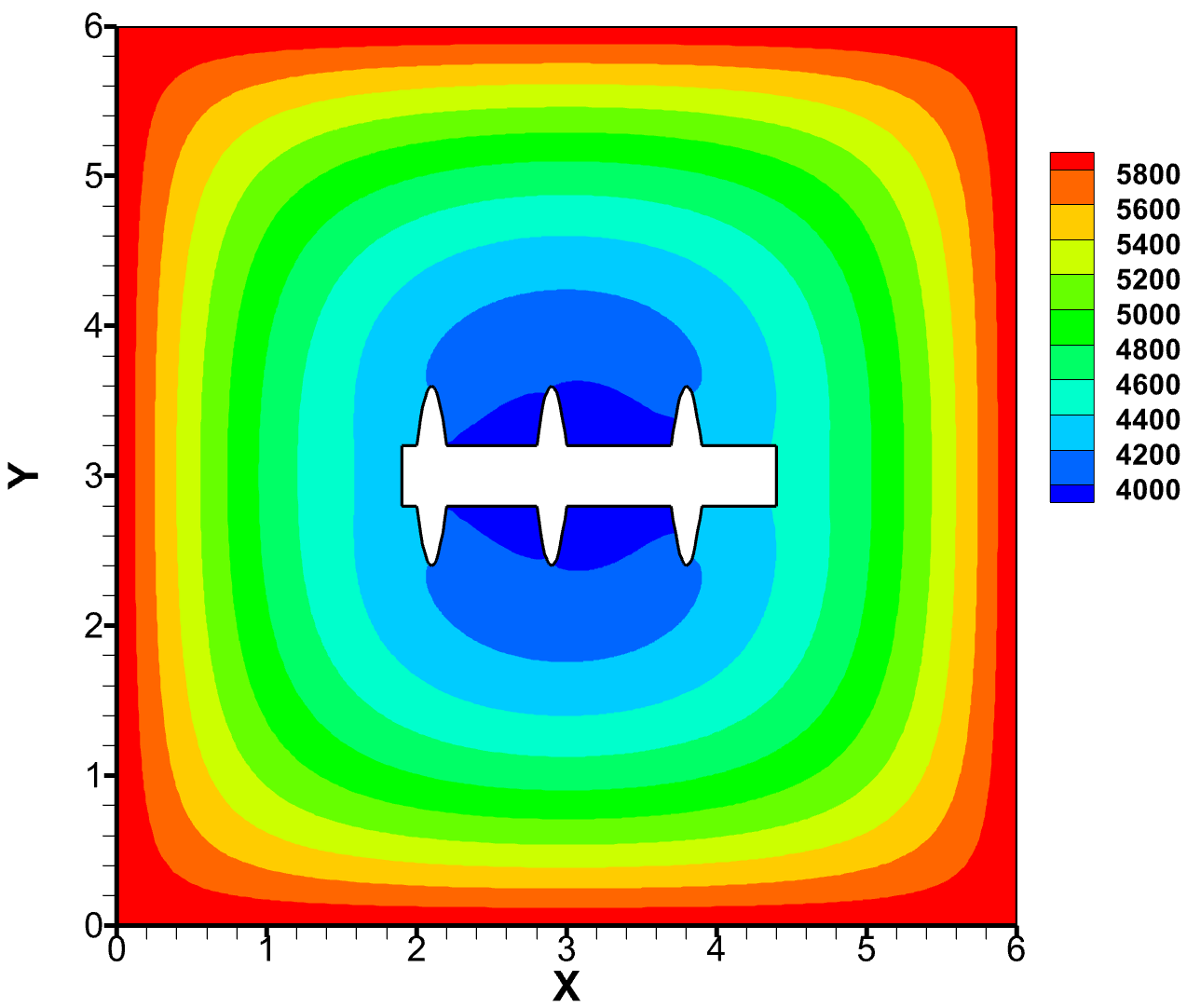}
\end{subfigure}
\end{centering}
\caption{\label{Fig56p}\small{The pressure around the multistage hydraulically fractured horizontal production wellbore with cased-hole completions in Algorithm \ref{Algorithm-1}. 
}}
\end{figure}


\begin{figure}[H]
\begin{centering}
\begin{subfigure}[t]{0.31\textwidth}
\centering
\includegraphics[width=1.05\textwidth]{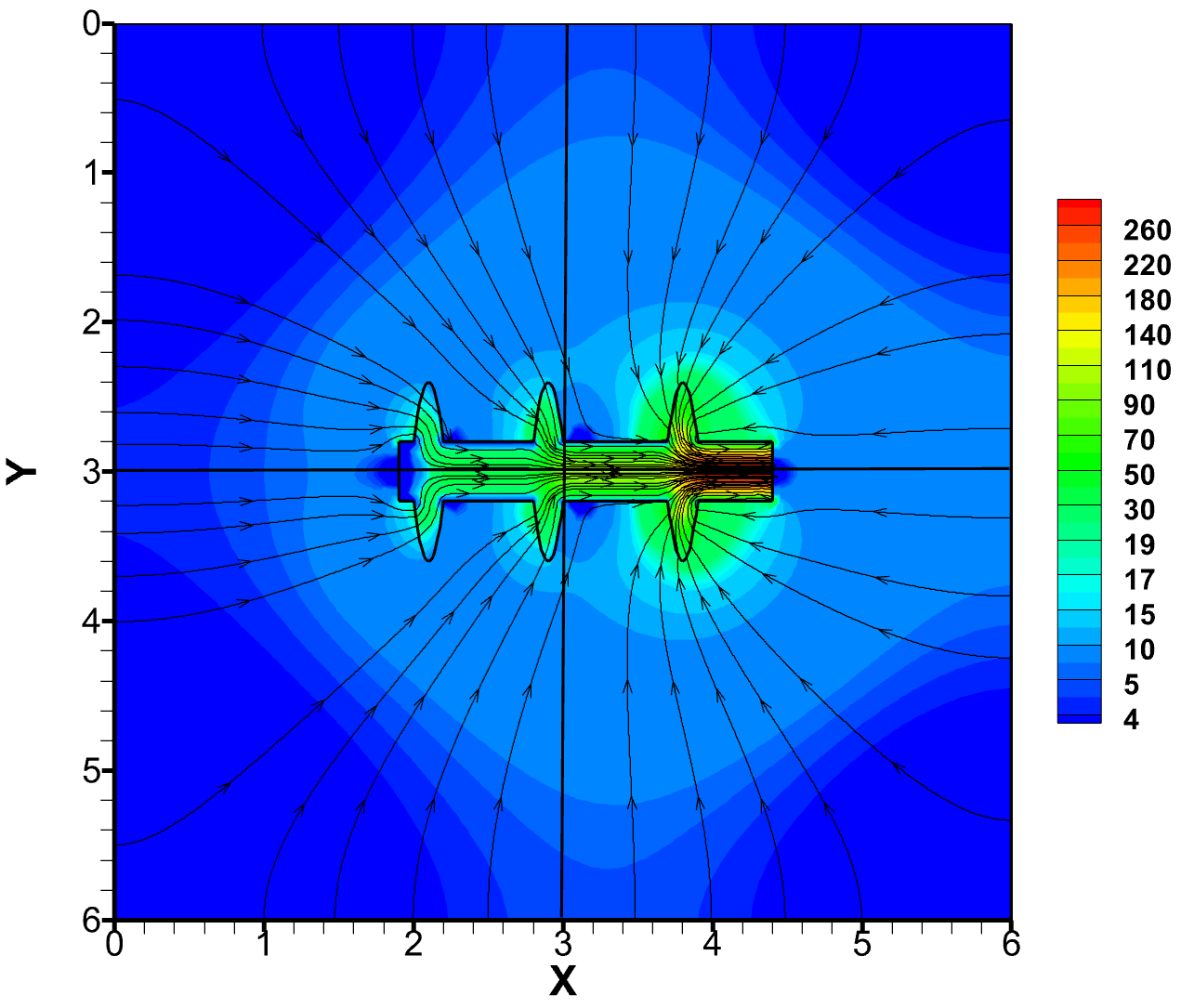}
\end{subfigure}
\quad
\begin{subfigure}[t]{0.31\textwidth}
\centering
\includegraphics[width=1.05\textwidth]{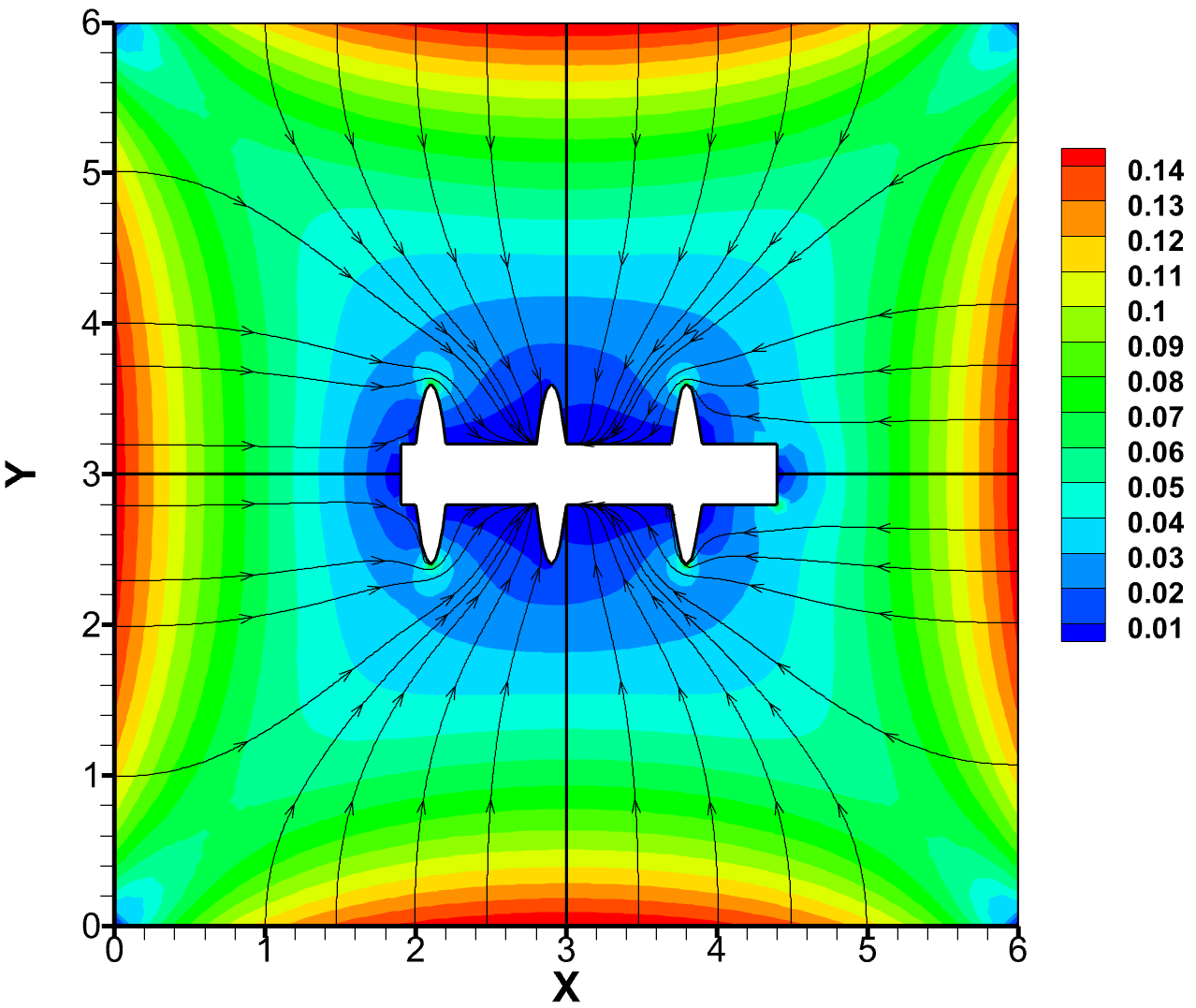}
\end{subfigure}
\quad
\begin{subfigure}[t]{0.31\textwidth}
\centering
\includegraphics[width=1.05\textwidth]{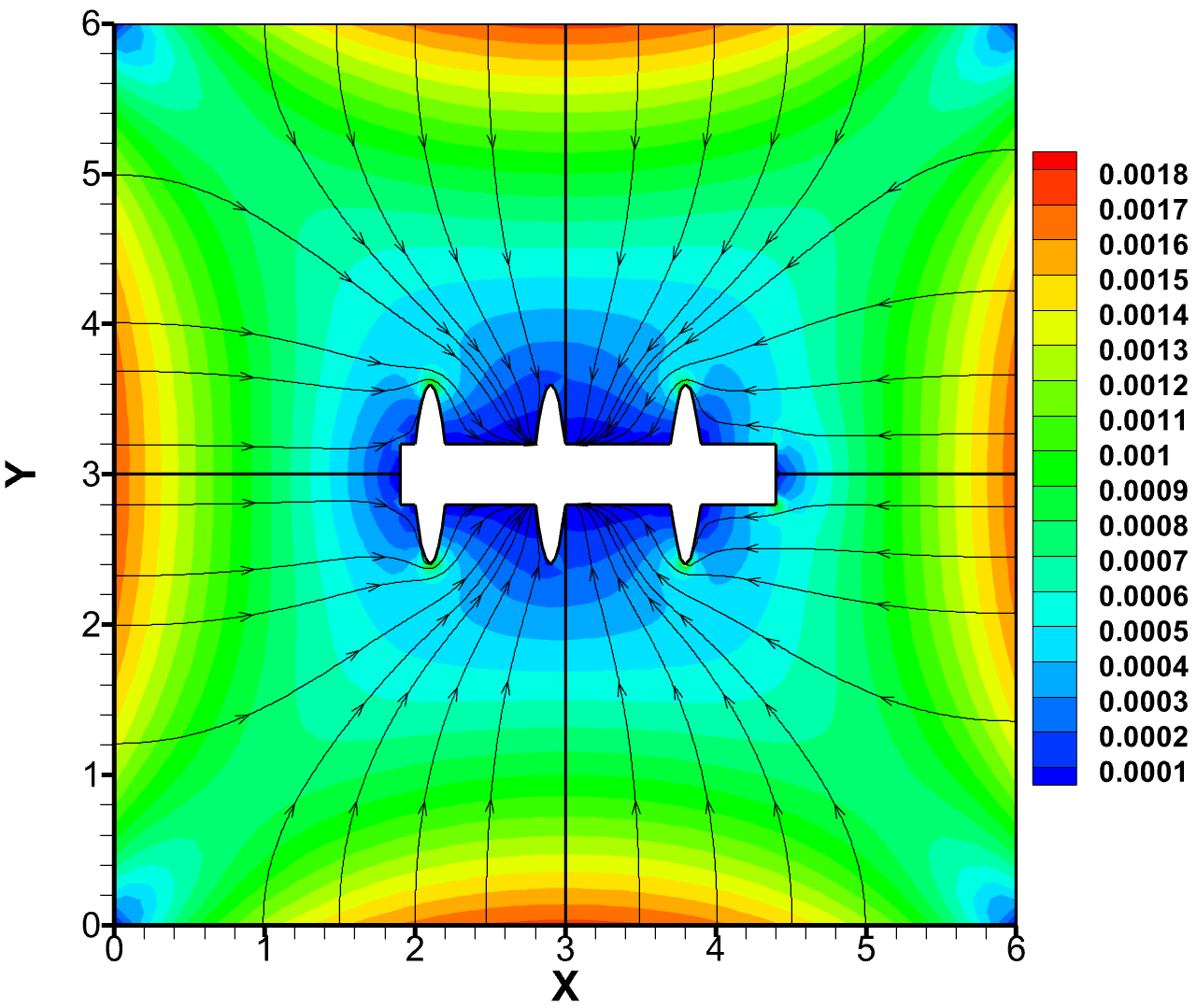}
\end{subfigure}
\end{centering}
\caption{\label{Fig55}\small{The velocity and streamlines around the multistage hydraulically fractured horizontal production wellbore with cased-hole completions in Algorithm \ref{algorithm3}.
}}
\end{figure}

\begin{figure}[H]
\begin{centering}
\begin{subfigure}[t]{0.31\textwidth}
\centering
\includegraphics[width=1.05\textwidth]{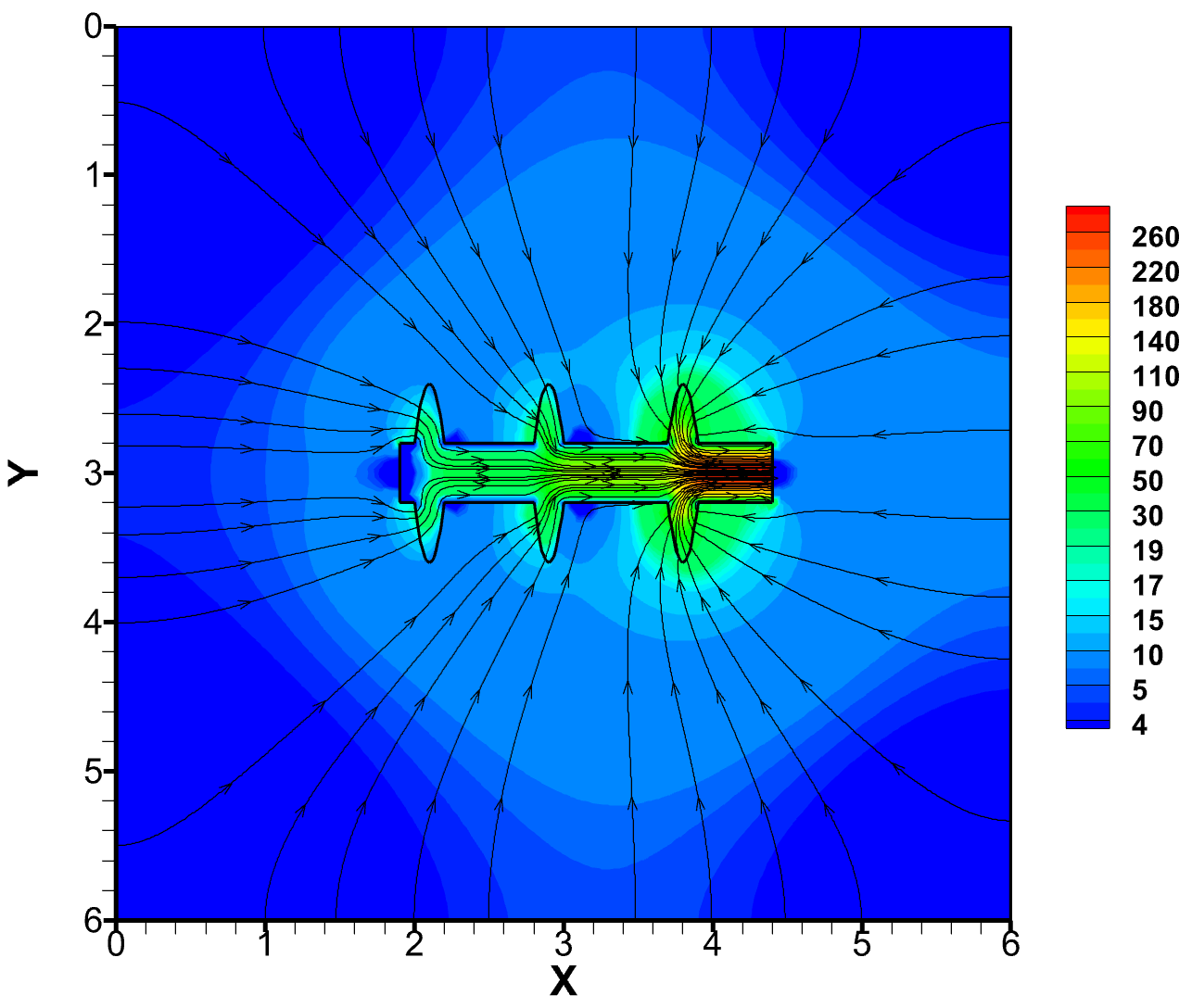}
\end{subfigure}
\quad
\begin{subfigure}[t]{0.31\textwidth}
\centering
\includegraphics[width=1.05\textwidth]{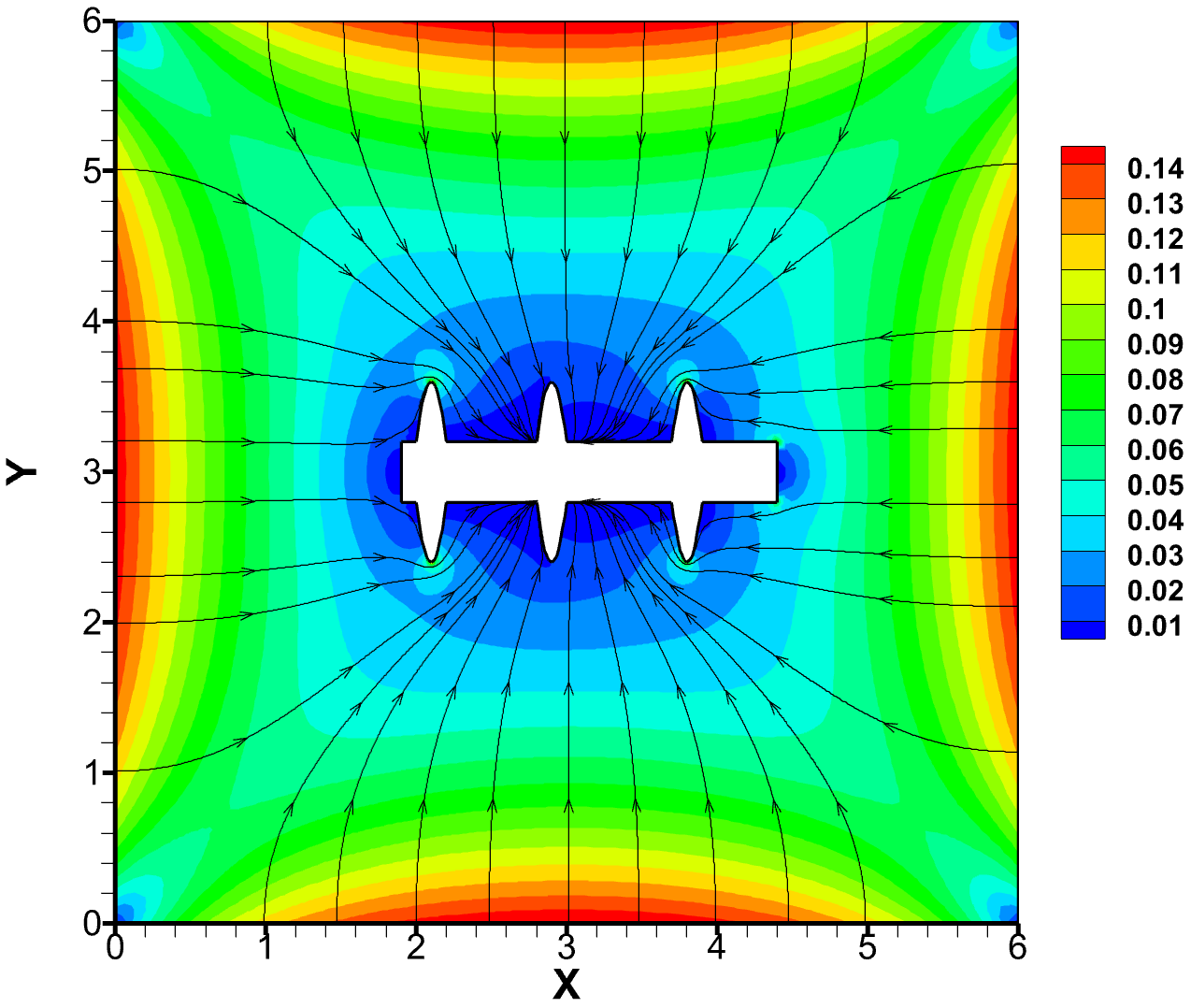}
\end{subfigure}
\quad
\begin{subfigure}[t]{0.31\textwidth}
\centering
\includegraphics[width=1.05\textwidth]{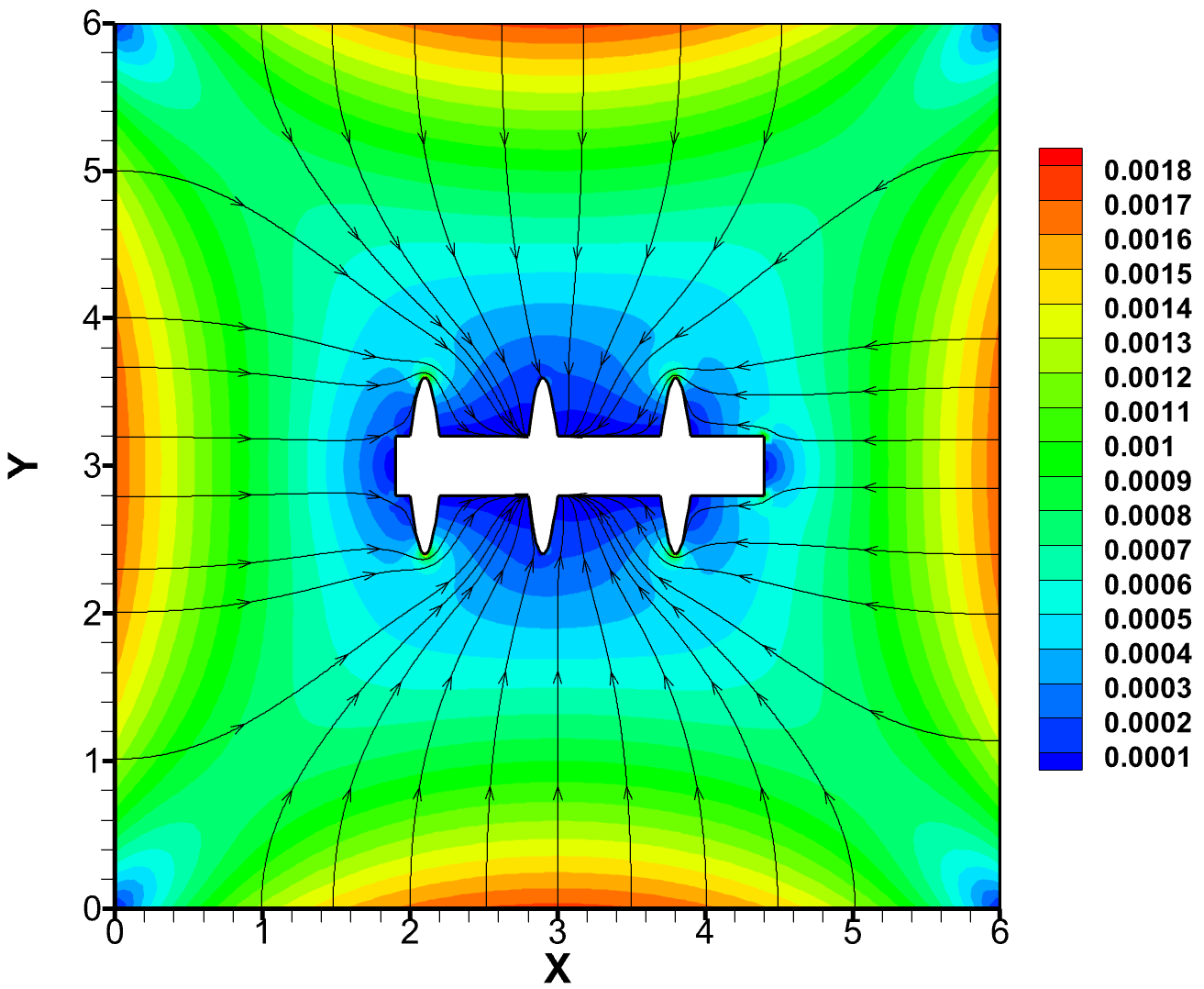}
\end{subfigure}
\end{centering}
\caption{\label{Fig56}\small{The velocity and streamlines around the multistage hydraulically fractured horizontal production wellbore with cased-hole completions in Algorithm \ref{Algorithm-1}.
}}
\end{figure}

\subsubsection{Simulation of the impact of different super-hydrophobic proppants on oil recovery rate}

In hydraulic fracturing engineering applications, different super-hydrophobic proppant materials directly influence the permeability $k_F$.
Based on the parameter settings in Section \ref{Example3-1},
we take different values $k_F=2 \times 10^{-2}, 4 \times 10^{-2}, 6 \times 10^{-2}, 8 \times 10^{-2}, 2 \times 10^{-1}, 4 \times 10^{-1}, 6 \times 10^{-1}, 8 \times 10^{-1}$ to
obtain different velocities using Algorithm \ref{Algorithm-1} and Algorithm \ref{algorithm3}.
Using the horizontal cross-section area $A$, we can get the oil production rate $Q=\vec{u}_c A$, which is shown in Figure \ref{OilRate}.

\begin{figure}[h]
  \centering
  \includegraphics[width=10cm]{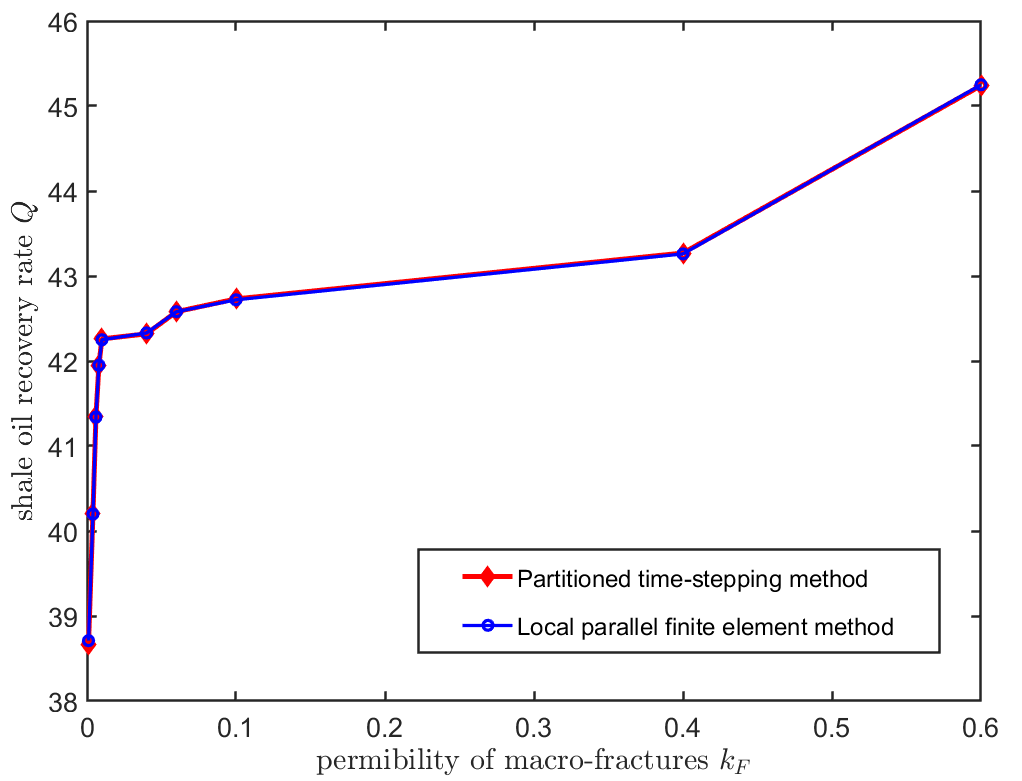}\\
  \caption{\label{OilRate}\small{The relationship between permeability of the macro-fractures and shale oil recovery rate.}}
\end{figure}

The relationship between permeability of the macro-fractures and shale oil recovery rates are shown in Figure \ref{OilRate}.
We can see that the recovery rates calculated by the two algorithms are the same.
In addition, as the proppant's oil permeability and water resistance performance increase, the oil production rate gradually rises.
Therefore, the super-hydrophobic proppant can enhance the oil recovery rate.

\section{Conclusions}
In this paper, we present a local parallel algorithm for super-hydrophobic proppants in a hydraulic fracturing system based on a 2D/3D transient triple-porosity Navier-Stokes Model.
Numerical examples are demonstrated to showcase the effectiveness and efficiency of the algorithm, as well as to illustrate its advantages in practical applications.
In the future, we will pursue additional research focusing on the two aspects.
One is to improve the results of Lemma \ref{prioriLem} to obtain the optimal error estimates for the proposed theorem.
The other one is to establish a more refined model in the porous media region, considering the influence of crack length, proppant volume, and capillary forces on the diversion capacity of cracks in the application of super-hydrophobic proppants.

\section*{Acknowledgements}
This work is supported in part by NSF of China(No. 11771259), Shaanxi Provincial Joint Laboratory of Artificial Intelligence(No.2022JC-SYS-05), National program for the introduction of high-end foreign experts(No.G2023041032L), Innovative team project of Shaanxi Provincial Department of Education(No.21JP013) and Shaanxi Province Natural Science basic research program key project(No.2023-JC-ZD-02).

~



\section*{References}
\bibliographystyle{plain}
\bibliography{bibfile}
\end{document}